\documentclass[english]{amsart}
\usepackage{amssymb,amsmath,amsfonts,amsthm}
\usepackage{stmaryrd}
\usepackage[all,cmtip,poly]{xy}
\usepackage{svn}
\usepackage{enumerate}

\SVN $LastChangedDate: 2013-08-28 15:15:23 -0700 (Wed, 28 Aug 2013) $
\SVN $LastChangedRevision: 395 $

\newcommand{\Ghat}{\hat G}
\newcommand{\hodgetateembedding}{{\kappa}}
\newcommand{\hodgetateembeddingbar}{{\overline{\kappa}}}
\newcommand{\Fr}{\operatorname{Fr}}

\theoremstyle{plain}
\newtheorem{conj}[equation]{Conjecture}
\newtheorem{cor}[equation]{Corollary}

\newtheorem{prop}[equation]{Proposition}
\newtheorem{lem}[equation]{Lemma}
\newtheorem{lemma}[equation]{Lemma}

\newtheorem{thm}[equation]{Theorem}

\newtheorem{ithm}{Theorem}

\theoremstyle{definition}
\newtheorem{defn}[equation]{Definition}
\theoremstyle{remark}

\newtheorem{ex}[equation]{Example}

\newtheorem{hypothesis}[equation]{Hypothesis}

\newtheorem{rem}[equation]{Remark}
\newtheorem{remark}[equation]{Remark}

\numberwithin{equation}{section}
\numberwithin{figure}{subsection}

\newcommand{\F}{\FF}

\newcommand{\Q}{\QQ}

\newcommand{\Z}{\ZZ}

\newcommand{\D}{\cD}

\renewcommand{\O}{\cO}

\newcommand{\m}{\frakm}

\newcommand{\rbar}{\bar{r}}
\newcommand{\Lbar}{\overline{L}}
\newcommand{\Wconj}{W^{\operatorname{BDJ}}}
\newcommand{\Fpbartimes}{{\overline{\F}_p^\times}}

\newcommand{\varepsilonbar  }{\overline{\varepsilon}}

 \newcommand{\rhobar   }{{\overline{\rho}}}

 \newcommand{\psibar   }{\overline{\psi}}

\newcommand{\betat         }{\widetilde{\beta}}

\newcommand{\FF}{{\mathbb F}}

\newcommand{\QQ}{{\mathbb Q}}

\newcommand{\ZZ}{{\mathbb Z}}

\newcommand{\cD}{{\mathcal D}}
\newcommand{\cE}{{\mathcal E}}

\newcommand{\cG}{{\mathcal G}}

\newcommand{\cO}{{\mathcal O}}

\newcommand{\frakf}{\mathfrak{f}}
\newcommand{\frakm}{\mathfrak{m}}

\newcommand{\frakL}{\mathfrak{L}}
\newcommand{\frakN}{\mathfrak{N}}

\newcommand{\Fbar}{\overline{\F}}
\newcommand{\Qbar}{\overline{\Q}}
\newcommand{\Zbar}{\overline{\Z}}

\newcommand{\Fp}{\F_p}

\newcommand{\Fpbar}{\Fbar_p}

\newcommand{\Zp}{\Z_p}

\newcommand{\Zpbar}{\Zbar_p}
\newcommand{\Zpbartimes}{\Zbar_p^\times}

\newcommand{\Ql}{\Q_{\ell}}
\newcommand{\Qp}{\Q_p}

\newcommand{\Qpbar}{\Qbar_p}

\DeclareMathOperator{\Ext}{Ext}
\DeclareMathOperator{\Fil}{Fil}
\DeclareMathOperator{\Gal}{Gal}
\DeclareMathOperator{\GL}{GL}

\DeclareMathOperator{\Hom}{Hom}

\DeclareMathOperator{\Mod}{Mod}

\DeclareMathOperator{\PGL}{PGL}

\DeclareMathOperator{\PSL}{PSL}
\DeclareMathOperator{\Rep}{Rep}

\DeclareMathOperator{\Sym}{Sym}

\newcommand{\ab}{\mathrm{ab}}
\newcommand{\cris}{\mathrm{cris}}

\newcommand{\Frob}{\mathrm{Frob}}

\newcommand{\HT}{\mathrm{HT}}

\newcommand{\st}{\mathrm{st}}

\newcommand{\ur}{\mathrm{ur}}

\newcommand{\llb}{\llbracket}
\newcommand{\rrb}{\rrbracket}
 
                               \newcommand{\into}{\hookrightarrow}

\newcommand{\lto}{\longrightarrow}

\newcommand{\ue}{\underline \epsilon}
\newcommand{\tor}{\mathrm{tor}}
\newcommand{\col}{\colon}

\newcommand{\fm}{\mathfrak{m}}
\newcommand{\fM}{{\mathfrak M}}
\newcommand{\fS}{{\mathfrak S}}

\newcommand{\gt}{{\mathfrak t}}
\newcommand{\Acris}{{A_{\textnormal{cris}}}}

\newcommand{\M}{{\mathfrak M}}
\newcommand{\Md}{{\rm{M}}_{d \times d}}
\newcommand{\e}{{\mathfrak e}}

\newcommand{\E}{{\mathcal E}}
\newcommand{\hR}{\widehat{\mathcal R}}

\newcommand{\mc}{\mathcal}
\newcommand{\mf}{\mathfrak}

\newcommand{\Art}{{\operatorname{Art}}}
\newcommand{\barM}{\overline{\M}}
\newcommand{\barN}{\overline{\mathfrak{N}}}
\newcommand{\barK}{\overline{{K}}}
\newcommand{\barQQ}{\overline{{\Q}}}
\newcommand{\ve}{\varepsilon}
\newcommand{\KisinFree}{\Mod_{\fS}^{\varphi, r}}
\newcommand{\KisinTor}{\Mod_{\fS,\tor}^{\varphi, r}}
\newcommand{\barMpp}{{\overline{\M''}}}
\newcommand{\barMp}{{\overline{\M'}}}

\begin{document}
\title{The Buzzard--Diamond--Jarvis conjecture for unitary groups}

\author{Toby Gee} \email{toby.gee@imperial.ac.uk} \address{Department of
  Mathematics, Imperial College London}\author{Tong
  Liu}\email{tongliu@math.purdue.edu}\address{Department of
  Mathematics, Purdue University}\author{David Savitt} \email{savitt@math.arizona.edu}
\address{Department of Mathematics, University of Arizona}
\thanks{The second author was partially
  supported by NSF grant  DMS-0901360; the third author was partially
  supported by NSF  grant DMS-0901049 and NSF CAREER grant
  DMS-1054032. This version of the paper is more recent than the
  published version and contains minor corrections.}  
\maketitle

\begin{abstract}Let $p>2$ be prime. We prove the weight part of
  Serre's conjecture for rank two unitary groups for mod $p$
  representations in the unramified case (that is, the
  Buzzard--Diamond--Jarvis conjecture for unitary groups), by proving
  that any Serre weight which occurs is a predicted weight. 
  Our methods are purely local, using
  the theory of $(\varphi,\hat{G})$-modules to determine the possible
  reductions of certain two-dimensional crystalline representations.
\end{abstract}

\section{Introduction}

Let $p$ be a prime number.  Classically, given a continous, odd, irreducible representation
$\rbar : G_{\Q} \to \GL_2(\Fpbar)$, the weight part of Serre's conjecture predicts the set of
weights $k$ such that $\rbar$ is isomorphic to the mod $p$
representation $\rbar_{f,p}$ attached to some eigenform of
weight $k$ (and prime-to-$p$ level).  In recent years, generalisations
of the weight part of Serre's conjecture have taken on an increasing
importance, at least in part because they can be viewed as statements about
local-global compatibility in a possible mod $p$ Langlands
correspondence, as we now (briefly) recall.  

Let $F$ be a number field and $\rbar : G_F \to
\GL_n(\Fpbar)$ a representation that is modular in a
suitable sense.  For simplicity, suppose
that $F$ has a single place $w$ lying above $p$; in this context a
Serre weight is an isomorphism class of irreducible mod $p$ representations of $\GL_n(\O_{F_w})$.
One may hope that there exists a mod~$p$ local Langlands
correspondence that attaches to $\rbar |_{G_{F_w}}$  a mod $p$
representation $\overline\Pi$ of $\GL_n(F_w)$.  Although our present understanding of
the putative representation $\overline\Pi$ is rather
limited, one ultimately
expects that $\rbar$ should be modular of Serre weight $a$ if and
only if $a$ is a subrepresentation of $\overline\Pi|_{\GL_n(\O_{F_w})}$.

In this paper we establish the weight part of Serre's conjecture for
rank two unitary groups in the case where
$F$ is unramified at $p$.  To be precise, we prove the
following. 
\begin{ithm}[Theorem \ref{thm: main global theorem - modular of a weight iff it is a
     predicted weight}]
  \label{thm: intro: the main result, modular if and only if
    predicted}Let $F$ be an imaginary CM field with maximal totally
  real subfield ~$F^+$, and suppose that $F/F^+$ is unramified at all
  finite places, that each place of $F^+$ above $p$ splits in $F$, and that $[F^+:\Q]$ is even. Suppose $p>2$, and
  that $\rbar\col G_F\to\GL_2(\Fpbar)$ is an irreducible modular
  representation with split ramification such that
  $\rbar(G_{F(\zeta_p)})$ is adequate. Assume that $p$ is unramified
  in~$F$.

 Let $a$ be a Serre weight. Then
 $a\in\Wconj(\rbar)$ if and only if $\rbar$ is modular of
 weight~$a$.
\end{ithm}

Here $\Wconj(\rbar)$ is the set of Serre weights
in which $\rbar$ is predicted to be modular.  We will recall
the definition of $\Wconj(\rbar)$ in Section~\ref{sec: serre
  weight definitions} below (as well as what we mean for $\rbar$ to be
modular of weight $a$, and any other unfamiliar terminology in the
statement of the theorem), but for now we give some
motivation and context.  Theorem~\ref{thm: intro: the main result, modular if and only if
    predicted} is the natural variant of the Buzzard--Diamond--Jarvis
conjecture for unitary groups; recall that the original
conjecture~\cite{bdj} was formulated for automorphic forms on indefinite
quaternion algebras. Note that strictly speaking, this is not the most general result that
one could hope to prove, because of the (mild) assumption that
$\rbar(G_{F(\zeta_p)})$ is adequate. In fact
  we prove unconditionally that if
  $\rbar$ is modular of weight~$a$, then $a\in\Wconj(\rbar)$; see Proposition~\ref{prop: modular of some weight implies crystalline lifts
     exist} and Theorem~\ref{thm: crystalline lift implies explicit
     crystalline lift}. The
  assumption that $\rbar(G_{F(\zeta_p)})$ is adequate is needed for
  the converse, which is proved in~\cite{blggU2} via automorphy
  lifting theorems.

To explain this in greater depth, suppose for simplicity that $F^+$ has a single place $v$ above $p$,
write the factorization of $v$ in $F$ as $ww^c$, and assume now that
$F_w/\Qp$ is unramified.  
If $\rbar$ is modular of weight~$a$, then $\rbar \simeq \rbar_{\pi}$
for some cuspidal automorphic representation~$\pi$ whose infinitesimal
character is determined by the weight $a$.  In particular, the local
representation $\rbar|_{G_{F_w}}$ has a lift $r_{\pi}|_{G_{F_w}}$ that
is crystalline with specific Hodge--Tate weights: to be precise, the lift
$r_{\pi}|_{G_{F_w}}$ has Hodge type~$a$ in the sense of Definition~\ref{defn:
  Galois representation of Hodge type some weight} below.  

One plausible definition for the set of predicted weights
$\Wconj(\rbar)$ (which is not the definition that we will use, although the main result of this paper shows that it is
in fact equivalent to our definition) would be the set of Serre
weights $a$ such that $\rbar|_{G_{F_w}}$ has a crystalline lift of
Hodge type $a$.  (There is a natural modification of this definition in
the case where $F_w/\Qp$ is ramified.)  Under this description of the
set of predicted weights, it would be
essentially automatic that if $\rbar$ is modular of weight~$a$ then $a
\in \Wconj(\rbar)$, and the problem would be to prove that every predicted
weight actually occurs.  
Significant progress towards 
establishing this result was made (irrespective
of any ramification conditions on $F$) in~\cite{blggU2}.  In particular,~\cite{blggU2} show that under the hypotheses of Theorem~\ref{thm: intro: the main result, modular if and only if
    predicted}, if
$\rbar|_{G_{F_w}}$ has a crystalline lift of
Hodge type $a$ that furthermore is \emph{potentially diagonalisable}
in the sense of~\cite{BLGGT}, then $\rbar$ is modular of weight $a$.  

Temporarily adopting this definition of $\Wconj(\rbar)$, our task,
therefore, is to remove the  potential
diagonalisability hypothesis; or in other words, we are left with the purely local problem
of showing that if $\rbar|_{G_{F_w}}$ has a crystalline lift of
Hodge type $a$, then it has a potentially diagonalisable
such lift.   This is a consequence of the following theorem, which is
our main local result.

\begin{ithm}[Theorem~\ref{thm: elimination in the reducible case}]\label{thm: intro: main local result}
Suppose that $p > 2$ and $K/\Qp$ is a finite unramified extension.
Let $\rho \col G_{K} \to \GL_2(\Zpbar)$ be a crystalline representation
whose $\hodgetateembedding$-labeled Hodge--Tate weights  for each
embedding $\hodgetateembedding \col K \into \Qpbar$ are $\{0,r_{\hodgetateembedding}\}$ with
$r_{\hodgetateembedding} \in [1,p]$.
If $\rhobar$ is reducible, then there exists a reducible
crystalline representation $\rho' \col G_{K} \to \GL_2(\Zpbar)$ with the
same labeled Hodge--Tate weights as $\rho$ such that $\rhobar \simeq \rhobar'$.
\end{ithm}

Before discussing the proof of Theorem~\ref{thm: intro: main local
  result}, we make a few additional comments about the global setting
of our paper, and about the actual definition of $\Wconj(\rbar)$ with
which we work.

\subsection*{Remark on the definition of $\Wconj(\rbar)$.}
One often builds the potential
diagonalisability hypothesis into the definition of $\Wconj(\rbar)$.
In fact this is what is done in~\cite{blggU2}, and for consistency we will
adopt the same definition here.  In this optic, the results
of~\cite{blggU2} prove that if $a \in \Wconj(\rbar)$, then $\rbar$ is
modular of weight $a$ (assuming of course that $\rbar$ is modular to
begin with); but then it becomes nontrivial to show that if
$\rbar$ is modular of weight $a,$ then $a$ is a predicted weight, and
that is what is done in the present paper.   One advantage of this
alternative definition is that it is relatively easier to make 
completely explicit.  Such a  description of the set of Serre
weights in the unramified case was made in~\cite{bdj}, and it is in
these explicit terms that we define the set $\Wconj(\rbar)$ in Section~\ref{sec: serre
  weight definitions} below.  (In the case that
$\rbar$ is reducible but indecomposable, the description is in terms
of certain crystalline extension classes.) 

\subsection*{Remarks on related papers.}
Theorem~\ref{thm: intro: the main result, modular if and only if
    predicted} had previously been established in the case of \emph{generic} (or \emph{regular})  weights in
\cite{geebdj}, by a rather different method.  In particular, the
regularity hypothesis allowed the author to avoid the difficulties
that arise when dealing with Hodge--Tate weights outside the
Fontaine--Laffaille range, i.e., the Hodge--Tate weight range
$[0,p-2]$.  The main contribution of this paper is a method for
addressing these difficulties.  It is perhaps also worth emphasizing
that for many applications (for instance the
work of the first author and Kisin \cite{GeeKisin} on the Breuil-M\'ezard
conjecture for potentially Barsotti-Tate representations) it is
essential that one know the weight part of Serre's conjecture in its
entirety, rather than generically.

We also recall that our previous paper \cite{GLS11} established the
weight part of Serre's conjecture
for unitary groups in the totally ramified case.   In that paper we used a mixture
of local and global techniques to complete the proof. These techniques relied on a combinatorial relationship
between Serre weights and the existence of potentially Barsotti--Tate
lifts, which does not hold in general; in particular we were able to
avoid having to prove the analogue of Theorem~\ref{thm: intro: main local result} in that setting.

Finally, we remark that Theorem~\ref{thm: intro: the main result, modular if and only if
    predicted}  is rather more general than anything that has been
  proved directly  for inner
forms of $\GL_2$ over totally real fields, where there is a parity
obstruction due to the unit group: algebraic Hilbert modular forms
must have paritious weight, which prevents one from applying the
techniques of \cite{blggU2} for non-paritious mod $p$ weights.
However,  there are now two proofs (due to
Newton~\cite{Newton},  and to Gee--Kisin~\cite{GeeKisin}) that the weight part of Serre's conjecture for inner forms of
$\GL_2$ is equivalent to the conjecture for unitary groups.  In
combination with the results in this paper and in  \cite{blggU2}, the conjecture for inner
forms of $\GL_2$ (that is, the original Buzzard--Diamond--Jarvis
conjecture) has thus been established, under a mild Taylor--Wiles
hypothesis on the image of the global representation.

\subsection*{Discussion of our approach to proving Theorem~\ref{thm: intro: main local result}.}
In the special case that the Hodge--Tate
weights~$r_{\hodgetateembedding}$ are all contained in the
interval $[1,p-2]$, Theorem~\ref{thm: intro: main local result}
follows easily from Fontaine--Laffaille theory. However,
Fontaine--Laffaille theory cannot be extended to the required range,
and so new methods are required.   

Perhaps the most direct approach to Theorem~\ref{thm: intro: main
  local result} would be to write down all the filtered $\varphi$-modules
corresponding to crystalline representations
$\rho$ of the sort considered in the theorem, and attempt to compute each
$\rhobar$ explicitly, for instance using the theory of
$(\phi,\Gamma)$-modules and Wach modules.  Some partial results towards Theorem
\ref{thm: intro: main local result} have been obtained 
by other authors working along these lines (\textit{cf}.~\cite{MR2776609}, \cite{MR2745686},
\cite{0807.1078}; the results of \cite{MR2776609} are
limited primarily to the case that $[K:\Qp]=2$, whereas the other two
references consider only semisimple~$\rhobar$ and restricted classes
of representations).  However, the general case has so
far been
resistant to these methods.

Instead, our idea is to proceed indirectly, by characterising the mod $p$
representations $\rhobar$ that arise in Theorem~\ref{thm: intro: main
  local result} without actually computing the reduction mod $p$ of
any specific $\rho$.  The key technical innovation in our paper is that it is possible to carry
out such an approach using
 the theory of $(\varphi,\hat{G})$-modules introduced in
 \cite{LiuLattices}.  
In particular, we are able to prove a structure theorem for 
$(\varphi,\hat{G})$-modules attached to crystalline Galois
representations of \emph{arbitrary dimension} with Hodge--Tate weights in
$[0,p]$ (Theorem \ref{shape}); this result is best possible, in the
sense that it does not extend to any wider Hodge--Tate weight range.
 We expect this structure theorem to be of broader
interest.  For instance, it can be used to study the possible reductions mod
$p$ of $n$-dimensional crystalline representations with Hodge--Tate
weights in the range $[0,p]$; we hope to report on this in a
future paper.

The proof of the structure theorem is rather delicate
and relies on a close study of the monodromy operator; the result does
not extend to a wider range of Hodge--Tate weights, nor do
we know how to extend it to the ramified case. 

Now assume that $\rho$ is as in Theorem~\ref{thm: intro: main local result}.
We use our structure theorem and an elementary argument to
determine the list of possible subcharacters of $\rhobar$ (Corollary
\ref{cor:form of the characters in reducible crystalline reduction,
  nothing about extensions}). This essentially completes the proof in the
completely decomposable case, but in the indecomposable case we need
to show that we have a lift of $\rhobar$ to a particular crystalline extension of
characters. To do this, we begin by making a careful study of
the possible extensions of rank one Kisin modules. We then examine the
possibility of extending these Kisin modules to
$(\varphi,\hat{G})$-modules, and show that in most cases such an
extension is unique. Together with some combinatorial arguments, this
enables us to show that all of the Galois representations resulting
from these $(\varphi,\hat{G})$-modules have reducible crystalline
lifts with the desired Hodge--Tate weights, completing the proof of
Theorem~\ref{thm: intro: main local result}.  Finally, note that Theorem~\ref{thm: intro: main local result}
addresses only the case where $\rhobar$ is reducible; we conclude by
deducing the irreducible case of Theorem~\ref{thm: intro: the main result, modular if
  and only if predicted} from the reducible one, using the fact that an irreducible
$\rhobar$ becomes reducible after restriction to an unramified
quadratic extension, together with another combinatorial argument.

It is natural to ask whether our methods could be extended to handle
the general case, where $F_w/\Qp$ is an arbitrary
extension. Unfortunately we do not know how to do this, because
the proof of the key Theorem \ref{shape} relies on the assumption that the
base field is unramified.

\subsection*{Outline of the paper.}  In Section \ref{sec: serre
  weight definitions} we recall some material from \cite{blggU2}, and
in particular explain the precise local results that we will need to prove in
the remainder of the paper.  The next three sections are concerned with the
general theory of Kisin modules and $(\varphi,\Ghat)$-modules attached
to crystalline representations.
 In Section~\ref{sec:coeffs} we review
what we will need of the theory of Kisin modules from~\cite{KisinCrys}.
In Section \ref{sec:shape}, which is the technical heart of the paper, we prove our structure
theorem for the
$(\varphi,\hat{G})$-modules attached to crystalline Galois
representations (of arbitrary dimension) with Hodge--Tate weights in
$[0,p]$. Section \ref{sec: crystalline phi Ghat modules} proves a
variety of foundational results on the $(\varphi,\Ghat)$-modules
associated to crystalline representations.

With our technical foundations established, we then begin the proofs
of Theorems~\ref{thm: intro: the main result, modular if
  and only if predicted} and~\ref{thm: intro: main local result}.
Section~\ref{sec:kisin-modules-varphi} contains basic results about
rank one Kisin modules and $(\varphi,\Ghat)$-modules.
In Section
\ref{sec:extensions-rank-one} a detailed study of the possible
extensions of rank one torsion Kisin modules is carried out;
crucially, thanks to our
work in Section~\ref{sec:shape} we are able to specialize these
results for Kisin modules coming from the reduction mod $p$
of crystalline representations with Hodge--Tate weights in $[0,p]$.
This work is extended
to the case of $(\varphi,\Ghat)$-modules in Section~\ref{sec:extensions-rank-one-G-hat}. Finally, we deduce our main results in
Sections \ref{sec:reducible} and \ref{sec:irreducible}.

\subsection{Acknowledgments} We would like to thank Brian Conrad, Fred
Diamond, Matthew Emerton, Jean-Marc Fontaine, and Mark Kisin for
helpful conversations, and Matthew Emerton and Florian Herzig for
their comments on an early draft of this manuscript. We are grateful
to the anonymous referees for their comments, which have improved the
exposition of this paper. D.S.~would like
to thank the mathematics department of Northwestern University for its
hospitality during a sabbatical year.

\subsection{Notation and conventions}\label{subsec:notation}

\subsubsection{Galois theory}
\label{sec:galois-theory}
If $M$ is a field, we let $G_M$ denote its
absolute Galois group.
If~$M$ is a global field and $v$ is a place of $M$, let $M_v$ denote
the completion of $M$ at $v$.   If
~$M$ is a finite extension of $\Ql$ for some $\ell$, we let $M_0$
denote the maximal unramified extension of $\Ql$ contained in $M$, and
we write $I_M$
for the inertia subgroup of~$G_M$. If $R$ is a local ring we write
$\mf{m}_{R}$ for the maximal ideal of $R$.

 Let $p$ be a prime number.   Let $K$ be a finite extension of $\Qp$, with ring of integers $\cO_K$
 and residue field~$k$.   Fix a uniformiser $\pi$ of $K$, let
 $E(u)$ denote the minimal polynomial of $\pi$ over $K_0$, and set $e
 = \deg E(u)$.  We also fix an algebraic closure $\barK$ of $K$.  The
 ring of Witt vectors $W(k)$ is the ring of integers in $K_0$.

Our representations of $G_K$ will have coefficients in $\Qpbar$,
another fixed algebraic closure of $\Qp$, whose residue field we denote $\Fpbar$.   Let $E$ be a finite
extension of $\Qp$ contained in $\Qpbar$ and containing the image of every
embedding of $K$ into $\Qpbar$; let $\O_E$ be the ring of integers in
$E$, with uniformiser $\varpi$ and residue field $k_E \subset \Fpbar$.

We write $\Art_K \col K^\times\to W_K^{\ab}$ for
the isomorphism of local class field theory, normalised so that
uniformisers correspond to geometric Frobenius elements.  For each $\sigma\in \Hom(k,\Fpbar)$ we
define the fundamental character $\omega_{\sigma}$ corresponding
to~$\sigma$ to be the composite $$\xymatrix{I_K \ar[r] & W_K^{\ab} \ar[r]^{\Art_K^{-1}} &
  \O_{K}^{\times}\ar[r] & k^{\times}\ar[r]^{\sigma} &
  \Fpbar^{\times}.}$$
In the case that $k\simeq\Fp$, we will sometimes write $\omega$ for
$\omega_\sigma$. Note that in this case we have
$\omega^{[K:\Qp]}=\varepsilonbar$; here~$\varepsilon$ denotes the $p$-adic cyclotomic
character, and $\bar{\varepsilon}$ the mod~$p$ cyclotomic
character.

We fix a compatible system of $p^n$th roots of $\pi$: that is, we set
$\pi_0 = \pi$ and for all $n  > 0$ we fix a choice of $\pi_n$
satisfying $\pi_{n}^p = \pi_{n-1}$.   Similarly fix a compatible system
of primitive $p^n$th roots of unity $\zeta_{p^n}$.   Define the
following fields:
$$ K_\infty = \bigcup\limits_{n=0}^\infty K(\pi_n), \qquad K_{p^\infty}=  \bigcup \limits _{n=1}^\infty
K(\zeta_{p^n}), \qquad \hat K =\bigcup \limits_{n=1}^\infty K_\infty (\zeta_
{p^n}).$$
Note that $\hat K$ is the Galois closure of $K_\infty$ over $K$.
Write $G_\infty = \Gal (\barK / K_\infty)$,
$\hat G_{p^\infty} := \Gal(\hat K/ K_{p^\infty})$, $\hat G =\Gal (\hat K/K)$, and  $H_{K}:=
\Gal (\hat K/ K_\infty)$.

 If $p > 2$ then $\hat G \simeq \hat G_{p^\infty} \rtimes H_K $ and
 $\hat G_{p^\infty}\simeq  \Z_p(1)$ (see
 e.g.~\cite[Lem.~5.1.2]{MR2388556} for a proof), and so we can (and do) fix a
 topological generator $\tau \in \hat G _{p ^\infty}$.  In that case,
 we take our choice of $\zeta_{p^n}$ to be $\tau(\pi_n)/\pi_n$ for all
 $n$.

\subsubsection{Hodge--Tate weights}
\label{sec:p-adic-hodge}

If $W$ is a de Rham representation of $G_K$ over
$\barQQ_p$ and $\hodgetateembedding$ is an embedding $K \into \barQQ_p$ then the multiset
$\HT_\hodgetateembedding(W)$ of Hodge--Tate weights of $W$ with respect to $\hodgetateembedding$ is
defined to contain the integer $i$ with multiplicity $$\dim_{\barQQ_p} (W
\otimes_{\hodgetateembedding,K} \widehat{\barK}(-i))^{G_K},$$ with the usual notation
for Tate twists. (Here $\widehat{\barK}$ is the completion of $\barK$.)  Thus for example
$\HT_\hodgetateembedding(\varepsilon)=\{ 1\}$. We will refer to the
elements of $\HT_\hodgetateembedding(W)$  as the
``$\hodgetateembedding$-labeled Hodge--Tate weights of $W$'', or simply as the
``$\hodgetateembedding$-Hodge--Tate weights of $W$''.

\subsubsection{$p$-adic period rings}
\label{sec:p-adic-period-rings}
Define $\fS = W(k)\llbracket u\rrbracket$.  The ring $\fS$
is equipped with a Frobenius endomorphism $\varphi$ via $u\mapsto u^p$
along with the natural Frobenius on $W(k)$.

We denote by $S$  the $p$-adic completion of the divided power
envelope of $W(k)[u]$ with respect to the ideal generated by $E(u)$.
Let $\Fil^r S$ be the closure in $S$ of the ideal generated by
$E(u)^i/i!$ for $i \ge r$.   Write $S_{K_0} = S[1/p]$ and
$\Fil^r S_{K_0} = (\Fil^r S)[1/p]$.
There is a unique Frobenius map $\varphi\col S \to S$ which extends
the Frobenius on $\fS$. We write $N_S$ for the $K_0$-linear derivation on $S_{K_{0}}$ such that
$N_S(u)= -u$.

Let $R= \varprojlim \O_{\overline K }/p$ where the transition maps are
the $p$th power map.    The ring $R$ is a valuation ring with valuation defined by $v_R((x_n)_{n \ge 0}) =
 \lim_{n\to\infty} p^n v_p(x_n)$, where $v_p(p)=1$; the residue field
 of $R$ is $\overline{k}$, the residue field of $\barK$.

 By the universal property of  the Witt vectors
$W(R)$ of $R$, there is a unique surjective projection map $\theta
\col W(R) \to \widehat \O_{\overline K}$ to the $p$-adic completion of
$\O_{\overline K}$ which lifts the  projection $R \to \O_{\overline K}/ p$
onto the first factor in the inverse limit. We denote by $A_{\text{cris}}$ the $p$-adic completion
of the divided power envelope of $W(R)$ with respect to
$\ker(\theta)$.
Write $\underline \pi=(\pi_n)_{n\geq 0}\in R$
and let $[\underline \pi ]\in W(R)$ be the Teichm\"uller
representative. We embed the $W(k)$-algebra $W(k)[u]$ into
$W(R)\subset\Acris$ by the map $u\mapsto [\underline \pi]$. This
embedding extends to embeddings $\fS \hookrightarrow S \hookrightarrow \Acris$ which  are  compatible with
Frobenius endomorphisms. As usual, we write $B_\text{cris}^+=
\Acris[1/p]$.  As a subring of $\Acris$, the ring $S$ is not stable
under the action of $G$; however, $S$ is the subring of
$G_{\infty}$-invariants in $\Acris$ (see
  \cite[\S 4]{BreuilGriffith}).

  Let $\cO_{\cE}$ denote the $p$-adic completion of
  $\fS[\frac{1}{u}]$, a discrete valuation ring with residue field
  $k(\!(u)\!)$. Write $\cE$ for the field of fractions of $\cO_\cE$. The
  inclusion $\fS\into W(R)$ extends to an inclusion $\cO_\cE\into
  W(\Fr R)$, and thus to $\cE\into W(\Fr R)[1/p]$. We let
  $\cE^{\text{ur}}$ denote the maximal unramified extension of $\cE$
  in $W(\Fr R)[1/p]$, with ring of integers $\cO^{\text{ur}}$. Write
  $\widehat{\cE^{\text{ur}}}$ for the $p$-adic completion of
  $\cE^{\text{ur}}$, with ring of integers
  $\widehat{\cO^{\text{ur}}}$. Write
  $\fS^{\text{ur}}=\widehat{\cO^{\text{ur}}}\cap W(R)\subset W(\Fr
  R)$.

Set  $\ue:= (\zeta_{p^i}) _{i \geq
  0} \in R$ and $t = -\log([\ue])\in \Acris$.  For any $g\in G_K$,
write $\underline \varepsilon(g) = g(\underline \pi)/\underline \pi$,
which is a
cocycle from $G_K$ to  $R^\times$.  Note that $\underline \varepsilon(\tau) =\underline \varepsilon$.

By \cite[Ex. 5.3.3]{liu-Fontaine} (see also the discussion before
Theorem 3.2.2 of \emph{ibid.}) there exists an element $\mathfrak{t} \in W(R)$
such that $t = c \varphi(\mathfrak{t})$ with $c \in S^{\times}$.  It
is shown in the course of the proof of \cite[Lem 3.2.2]{LiuLattices}
that the image of $\mathfrak{t}$ in $R$ has valuation~$\frac{1}{p-1}$.
Following \cite[\S 5]{MR1293971} we define
$$ I^{[m]} B_{\cris}^+ = \{ x \in B_{\cris}^+ : \varphi^n(x) \in
  \Fil^m B_{\cris}^+ \ \text{for all} \ n > 0\}.$$
(See \cite[\S 5]{MR1293971} for the definition of the filtration on
$B_{\cris}^+$.)  For any subring $A \subset B_{\cris}^{+}$ write
$I^{[m]} A = A \cap I^{[m]} B_{\cris}^+$.  By \cite[Prop
5.1.3]{MR1293971} the ideal $I^{[m]} W(R)$ is principal, generated by $\varphi(\mathfrak{t})^m$.

\section{Serre weight conjectures}\label{sec: serre
  weight definitions}In this section we
explain the definition of the sets of weights $\Wconj(\rbar)$, and
recall some results from \cite{blggU2}. We refer the reader to Section
4 of \cite{blggU2} for a detailed discussion of these definitions and
their relationship with other definitions in the
literature.\subsection{Local definitions} Let $K$ be a finite
unramified extension of $\Qp$ of degree $f$ with residue field $k$,
and let $\rhobar\col G_K\to\GL_2(\Fpbar)$ be a continuous
representation.
\begin{defn}
  A \emph{Serre weight} is an isomorphism class of irreducible representations of
  $\GL_2(k)$ over $\Fpbar$. Up to isomorphism, any such representation is of the
  form \[F_a:=\otimes_{\sigma\col k\into\Fpbar}\det{}^{a_{\sigma,2}}\otimes\Sym^{a_{\sigma,1}-a_{\sigma,2}}k^2\otimes_{\sigma,k}\overline{\F}_p\]
  where $0\le a_{\sigma,1}-a_{\sigma,2}\le p-1$ for each $\sigma$.  
We recall that
$F_a\simeq F_b$ as representations of $\GL_2(k)$ if and only if we
have $a_{\sigma,1}-a_{\sigma,2}=b_{\sigma,1}-b_{\sigma,2}$ for all
$\sigma$, and the character $k^\times\to\Fpbartimes$,
$x\mapsto\prod_{\sigma\col k\into\Fpbar}\sigma(x)^{a_{\sigma,2}-b_{\sigma,2}}$
is trivial.
\end{defn}

Write $\Z^2_+$ for the set of pairs of integers $(n_1,n_2)$ with
$n_1\ge n_2$.  We also use the term Serre weight to refer to tuples 
 $a = (a_{\sigma,1},a_{\sigma,2})_\sigma \in
 (\Z^2_+)^{\Hom(k,\Fpbar)}$ with the property that
 $a_{\sigma,1}-a_{\sigma,2} \le p-1$ for all $\sigma \in \Hom(k,\Fpbar)$, and we identify
 the Serre weight $a \in (\Z^2_+)^{\Hom(k,\Fpbar)}$ with the
 Serre weight represented by $F_a$.  (Note that a Serre weight in the
 latter sense will be represented by infinitely many Serre weights in the
 former sense.)
Since there is a natural bijection
between $\Hom(k,\Fpbar)$ and $\Hom_{\Qp}(K,\Qpbar)$, we will feel free
to regard a Serre weight as an element of
$(\Z^2_+)^{\Hom_{\Qp}(K,\Qpbar)}$. (In the terminology of
\cite{blggU2} we are regarding the Serre weight as a {\em lift} of
itself; as such lifts are unique in the unramified case, we choose not to
use this terminology in this paper.)

\begin{defn}
  \label{defn: Galois representation of Hodge type some weight}Let
  $K/\Qp$ be a finite extension, let
  $\lambda\in(\Z^2_+)^{\Hom_{\Qp}(K,\Qpbar)}$, and let
  $\rho\col G_K\to\GL_2(\Qpbar)$ be a de Rham representation. Then we say
  that $\rho$ has \emph{Hodge type} $\lambda$ if for each
  $\hodgetateembedding\in\Hom_{\Qp}(K,\Qpbar)$ we have $\HT_\hodgetateembedding(\rho)=\{\lambda_{\hodgetateembedding,1}+1,\lambda_{\hodgetateembedding,2}\}$.
\end{defn}

Following \cite{bdj} (as explained in \cite[\S 4]{blggU2}), we
define an explicit set of Serre weights $\Wconj(\rhobar)$.
\begin{defn}
  \label{defn: W? niveau 1}If $\rhobar$ is reducible, then a Serre
  weight $a\in(\Z^2_+)^{\Hom(k,\Fpbar)}$ is in $\Wconj(\rhobar)$ if
  and only if $\rhobar$ has a crystalline lift of the
  form \[ \begin{pmatrix}\chi_1&*\\ 0& \chi_2
  \end{pmatrix}\] which has Hodge type $a$. 
In particular, if $a\in \Wconj(\rhobar)$ then by
 \cite[Lem.~6.2]{geesavitttotallyramified} (or by Lemma \ref{lem: lift rank-1
   object} and Proposition \ref{prop:calculation-on-inertia} below) it
 is necessarily the case that  there is a decomposition 
 $\Hom(k,\Fpbar)=J\coprod J^c$ such that  \[\rhobar|_{I_K}\simeq
  \begin{pmatrix}
 \prod_{ \sigma\in
      J}\omega_{\sigma}^{a_{\sigma,1}+1}\prod_{\sigma\in
      J^c}\omega_\sigma^{a_{\sigma,2}}&*\\ 0&  \prod_{\sigma\in
      J^c}\omega_\sigma^{a_{\sigma,1}+1}\prod_{\sigma\in
      J}\omega_\sigma^{a_{\sigma,2}}   \end{pmatrix}.\]
\end{defn}

 Let $K_2$ denote the quadratic unramified
extension of $K$ inside~$\barK$, with residue field $k_2$.
\begin{defn}
  \label{defn: W? niveau 2} If $\rhobar$ is irreducible, then a Serre
  weight $a\in(\Z^2_+)^{\Hom(k,\Fpbar)}$ is in $\Wconj(\rhobar)$ if
  and only if there is a subset $J\subset\Hom(k_2,\Fpbar)$ containing
  exactly one element extending each element of $\Hom(k,\Fpbar)$,
   such that if we write
  $\Hom(k_2,\Fpbar)=J\coprod J^c$, then  \[\rhobar|_{I_K}\simeq
  \begin{pmatrix}\prod_{\sigma\in
      J}\omega_{\sigma}^{a_{\sigma,1}+1}\prod_{\sigma\in
      J^c}\omega_\sigma^{a_{\sigma,2}}&0\\ 0& \prod_{\sigma\in
      J^c}\omega_\sigma^{a_{\sigma,1}+1}\prod_{\sigma\in
      J}\omega_\sigma^{a_{\sigma,2}}
  \end{pmatrix}.\]
\end{defn}
We remark that by \cite[Lem.~4.1.19]{blggU2}, if
$a\in\Wconj(\rhobar)$ and $\rhobar$ is irreducible then
$\rhobar$ necessarily has a crystalline lift of Hodge type $a$. 

It is worth stressing that in all cases, if $a \in (\Z_{+}^2)^{\Hom(k,\Fpbar)}$ is a Serre weight,
  then whether or not $a \in \Wconj(\rhobar)$ depends only on the
  representation $F_a$; this can be seen by twisting by suitable
  crystalline characters.  It is also worth remarking again (\textit{cf.}\ the discussion in the
introduction)
that there are other definitions one could
  make of a set of conjectural weights. For example, one could define
  the set of conjectural weights for $\rhobar$ to be the set of
  weights $a$ for which $\rhobar$ has a crystalline lift of Hodge
  type~$a$; 
this would be the most natural definition from the perspective
  of local-global compatibility, \textit{cf.}~Proposition~\ref{prop: modular of
    some weight implies crystalline lifts exist}, which shows that any
  set of conjectural weights should be contained in this set.  We choose
  our definition of $\Wconj(\rhobar)$ in order to be consistent with
  \cite{blggU2}; ultimately, it follows from the results of this paper that
  these two definitions are equivalent.

 
  \subsection{Global definitions}\label{ss:global} The point of the local definitions
  above is to allow us to formulate global Serre weight
  conjectures. Following \cite{blggU2}, we work with rank two unitary
  groups which are compact at infinity. As we will not need to make
  any arguments that depend on the particular definitions made in
  \cite{blggU2}, and our methods are purely local, we simply
  recall some notation and basic properties of the definitions,
  referring the reader to \cite{blggU2} for precise formulations.

We emphasise that our conventions for Hodge--Tate weights are the
opposite of those of \cite{blggU2}; for this reason, we must introduce
a dual into the definitions.

Fix an imaginary CM field $F$ in which $p$ is unramified, and let
$F^+$ be its maximal totally real subfield. We define a global notion
of Serre weight by taking a product of local Serre weights in the
following way.

For each place $w|p$ of $F$, let $k_w$ denote the residue
field of $F_w$. If $w$ lies over a place $v$ of $F^+$, write
$v=ww^c$. Write $S:=\coprod_{w|p}\Hom(k_w,\Fpbar)$, and let $(\Z^2_+)_0^S$ denote the
subset of $(\Z^2_+)^S$ consisting of
elements $a$ such that for each $w|p$, if $\sigma\in\Hom(k_w,\Fpbar)$
then \[a_{\sigma,1}+a_{\sigma c,2}=0.\]
If $a \in (\Z^2_+)^S$ and $w|p$ is a
place of $F$, then let
$a_w$ denote the element $(a_\sigma)_{\sigma\in
  \Hom(k_w,\Fpbar)}$ of $(\Z^2_+)^{\Hom(k_w,\Fpbar)}$.

\begin{defn}
  \label{defn:global-serre-wts}
  We say that an element $a\in(\Z^2_+)_0^{S}$ is a \emph{Serre weight}
  if for each $w | p$ and $\sigma \in \Hom(k_w,\Fpbar)$ we have \[p-1\ge a_{\sigma,1}-a_{\sigma,2}.\]
\end{defn}

Let $\rbar\col G_F\to\GL_2(\Fpbar)$ be a continuous irreducible
representation.  We refer the reader to
\cite[Def.~2.1.9]{blggU2} for an explanation of what it
means for $\rbar$ to be modular, and more precisely for $\rbar$ to be
modular of some Serre weight $a$; roughly speaking, $\rbar$ is modular
of weight $a$ if there is a cohomology class on some unitary group
with coefficients in a certain local system corresponding to $a$ whose
Hecke eigenvalues are determined by the characteristic polynomials of
$\rbar$ at Frobenius elements. Since our conventions for Hodge--Tate
weights are the opposite of those of \cite{blggU2}, we make the
following definition.

\begin{defn}
   Suppose that $\rbar\col G_F\to\GL_2(\Fpbar)$ is a continuous
  irreducible modular representation. Then we say that $\rbar$ \emph{is modular
  of weight} $a\in(\Z^2_+)_0^S$ if
$\rbar^\vee$ is modular of weight $a$ in the sense
of \cite[Def.~2.1.9]{blggU2}.
\end{defn} We remark that if $\rbar$ is modular then
 $\rbar^c\simeq\rbar^\vee\otimes\varepsilonbar$. We  globalise the definition of the set
$\Wconj(\rhobar)$ in the following natural fashion.
\begin{defn}
  If $\rbar\col G_F\to\GL_2(\Fpbar)$ is a continuous representation, then
  we define $\Wconj(\rbar)$ to be the set of Serre weights
  $a\in(\Z^2_+)_0^S$ such that for each place $w|p$ the corresponding
  Serre weight $a_w\in(\Z^2_+)^{\Hom(k_w,\Fpbar)}$ is an element of
  $\Wconj(\rbar|_{G_{F_w}})$.
\end{defn}

One then has the following conjecture.
\begin{conj}\label{conj: global Serre weight explicit conjecture}
  Suppose that $\rbar\col G_F\to\GL_2(\Fpbar)$ is a continuous irreducible
  modular representation, and that
  $a\in(\Z^2_+)_0^S$ is a Serre
  weight. Then $\rbar$ is modular of weight $a$ if and only if
  $a\in\Wconj(\rbar)$.
\end{conj}
If $\rbar\col G_F\to\GL_2(\Fpbar)$ is a continuous representation,
  then we say that $\rbar$ has \emph{split ramification} if any finite
  place of $F$ at which $\rbar$ is ramified is split over $F^+$.
For the remainder of this section, we
 place ourselves in the following situation.
  \begin{hypothesis}\label{hypothesis: usual TW-type conditions}
    Let $F$ be an imaginary CM field with maximal totally real
    subfield~$F^+$, and let $\rbar\col G_F\to\GL_2(\Fpbar)$ be a
    continuous representation. Assume that:
    \begin{itemize}
    \item $p>2$,
  \item  $[F^+:\Q]$ is even,
    \item $F/F^+$ is unramified at all finite places,
 \item $p$ is unramified in $F$,
 \item each place of $F^+$ above $p$ splits in $F$, and
    \item $\rbar$ is an irreducible modular
  representation with split ramification.
\end{itemize}
  \end{hypothesis}

The  following result is \cite[Thm.~5.1.3]{blggU2}, one of the
main theorems of that paper, specialised to the case of interest to us
where $p$ is unramified in $F$.  (Note that in \cite{blggU2}, the set
of Serre weights $\Wconj(\rbar)$ is often denoted
$W^{\operatorname{explicit}}(\rbar)$. Note also that the assumption
that $p$ is unramified in $F$ implies that $\zeta_p \not\in F$.)
\begin{thm}
  \label{thm: explicit local lifts implies Serre weight}Suppose that
  Hypothesis \ref{hypothesis: usual TW-type conditions} holds.  Suppose further that $\rbar(G_{F(\zeta_p)})$ is adequate.   Let $a\in(\Z^2_+)_0^S$ be a Serre weight. Assume that $a\in
  \Wconj(\rbar)$. Then $\rbar$ is modular of weight $a$.
\end{thm}

Here \emph{adequacy} is a group-theoretic condition, introduced in
\cite{jack}. For subgroups of $\GL_2(\Fpbar)$ with $p > 5$, adequacy is
equivalent to the usual condition that the action is
irreducible; for $p=3$ it is equivalent to irreducibility and the
requirement that the projective image is not conjugate to
$\PSL_2(\F_3)$, and for $p=5$ it is equivalent to irreducibility and the
requirement that the projective image is not conjugate to
$\PSL_2(\F_5)$ or $\PGL_2(\F_5)$. (See
\cite[Prop.~A.2.1]{blggU2}.)  We also remark that the hypotheses that $F/F^+$
is unramified at all finite places, that every place of $F^+$ dividing
$p$ splits in $F$, and that $[F^+:\Q]$ is even, are in fact part of
the definition of ``modular'' made in
\cite{blggU2}. 

  Theorem~\ref{thm: explicit local lifts implies Serre weight}
  establishes one direction of Conjecture \ref{conj: global Serre
    weight explicit conjecture}, and we are left with the problem of
  ``elimination,'' i.e., the problem of proving that if $\rbar$ is
  modular of weight $a$, then $a\in \Wconj(\rbar)$. The following is
\cite[Cor.~4.1.8]{blggU2}.

\begin{prop}
  \label{prop: modular of some weight implies crystalline lifts
    exist}Suppose that Hypothesis \ref{hypothesis: usual TW-type
    conditions} holds. Let $a\in(\Z^2_+)_0^S$ be a Serre
  weight. If $\rbar$ is modular of weight $a$, then for each place
  $w|p$ of $F$, there is a crystalline representation
  $\rho_w\col G_{F_w}\to\GL_2(\Qpbar)$ lifting $\rbar|_{G_{F_w}}$, such
  that $\rho_w$ has Hodge type $a_w$.
\end{prop}

We stress that Proposition~\ref{prop: modular of some weight implies
  crystalline lifts exist} does not already complete the proof of
Conjecture \ref{conj: global Serre weight explicit conjecture},
because the representation $\rho_w$ may for example be irreducible
when $\rhobar_w$ is reducible (compare with Definition~\ref{defn: W?
  niveau 1}).  However, in light of this result, it is natural to
conjecture that the following result holds.
  \begin{thm}
    \label{thm: crystalline lift implies explicit crystalline lift}
    Let $K/\Qp$ be a finite  unramified extension, and let
    $\rhobar\col G_K\to\GL_2(\Fpbar)$ be a continuous representation. Let
    $a\in(\Z^2_+)^{\Hom(k,\Qpbar)}$ be a Serre weight, and suppose that there is a
    crystalline representation $\rho\col G_{K}\to\GL_2(\Qpbar)$ lifting
    $\rhobar$, such that $\rho$ has Hodge type $a$.  Then $a \in \Wconj(\rhobar)$.
  \end{thm}

 Theorem \ref{thm: crystalline lift implies explicit crystalline
    lift} is the main local result of this paper, and the remainder of
  the paper is concerned with its proof.      In the case that $\rhobar$ is irreducible, this is Theorem
    \ref{thm: reduction mod p for GL2 in the irreducible case} below;
    and in the reducible case it follows immediately from Theorem
    \ref{thm: elimination in the reducible case}.   Our methods are purely
  local.  We have the
  following global consequence, which essentially resolves Conjecture
  \ref{conj: global Serre weight explicit conjecture}.

\begin{thm}
  \label{thm: main global theorem - modular of a weight iff it is a
    predicted weight}Suppose that Hypothesis \ref{hypothesis: usual
    TW-type conditions} holds.  Suppose further that
  $\rbar(G_{F(\zeta_p)})$ is adequate.  Let $a\in(\Z^2_+)_0^S$ be a
  Serre weight. Then $\rbar$ is modular of weight $a$ if and only if
  $a\in\Wconj(\rbar)$.
\end{thm}
\begin{proof}
  This is an immediate consequence of Theorem \ref{thm: explicit local
    lifts implies Serre weight}, Proposition \ref{prop: modular of some weight implies crystalline lifts
    exist}, and Theorem \ref{thm: crystalline lift implies explicit crystalline lift}.
\end{proof}

\section{Kisin modules with coefficients} \label{sec:coeffs}

We begin to work towards the proof of Theorem~\ref{thm: crystalline
  lift implies explicit crystalline lift} by recalling some facts
about the theory of Kisin modules (or Breuil--Kisin modules) as
initiated by \cite{BreuilAnnals,BreuilUnpublished} and
developed in \cite{KisinCrys}, and giving some (essentially formal)
extensions of these results in order to allow for nontrivial coefficients.
Throughout this section we allow $K$ to be an arbitrary finite
extension of $\Qp$, and recall that $e = e(K/\Qp)$ is the ramification index
of $K$.  Recall also that our coefficient field $E$ is a finite extension of $\Qp$ contained in
$\Qpbar$ and containing the image of every embedding of $K$ into
$\Qpbar$.

\begin{defn}
  \label{defn:Kisin-module}
 A \emph{$\varphi$-module
over $\fS$} is an $\fS$-module $\M$ equipped with a
$\varphi$-semi-linear map $\varphi_\M\col \M \to \M $.
The subscript
on $\varphi_{\M}$ will generally be omitted.
A  morphism between two $\varphi$-modules $(\M_1, \varphi_1)$ and $(\M_2,
\varphi_2)$ is an $\fS$-linear morphism compatible with the maps
$\varphi_i$.  The map
$1\otimes \varphi\col \fS \otimes_{\varphi, \fS}\M \to \M$ is
$\fS$-linear, and we say that
$(\M,\varphi)$ has \emph{height $r$} if the cokernel of $1\otimes \varphi$ is
killed by $E(u)^r$; we say that $(\M,\varphi)$ has finite height if it
has height $r$ for some $r \ge 0$.

Denote by $'\Mod^{\varphi, r}_{\fS}$ the category of
$\varphi$-modules of  height $r$. By definition, a \emph{finite free
  Kisin module} (of height $r$) is a $\varphi$-module (of height $r$)
$\M$ such that the underlying $\fS$-module is finite free.  \emph{A
  torsion Kisin module} $\M$ is a $\varphi$-module of height $r$ which
is killed by $p^n$ for some $n \ge 0$, and such that the natural map
$\M \to \M [\frac 1 u]$ is injective. By
\cite[Prop 2.3.2]{liu-Fontaine}, this is equivalent to asking that $\M$ can be
written as the quotient of two finite free Kisin modules of equal
$\fS$-rank.

Throughout this article, a \emph{Kisin module} $\M$ is
either a finite free  Kisin module or a torsion Kisin module, of some
height $r$. We denote by $\Mod_{\fS}^{\varphi, r}$ the category of
finite free Kisin modules, and $\Mod_{\fS,\tor}^{\varphi, r}$ the
category of torsion Kisin modules.
 \end{defn}

Define contravariant functors $T_\fS$ from $\KisinFree$
and $\KisinTor$ to the category $\Rep_{\Zp}(G_\infty)$ of $\Z_p[G_\infty]$-modules as follows:
$$ T_\fS (\M):= \Hom_{\fS, \varphi} (\M, W(R)) \text{ if } \M \text{ is a finite free Kisin module}$$ and
$$ T_\fS(\M):= \Hom_{\fS, \varphi} (\M, \Q_p / \Z_p
\otimes_{\Z_p}W(R)) \text{ if } \M \text{ is a torsion Kisin module.}$$
These definitions are slightly different from the ones that are sometimes
given (e.g. \cite[Lem.~2.1.2, Cor.~2.1.4]{KisinCrys}), but in fact
the various definitions are equivalent by \cite[Prop.~B.1.8.3]{fo4}.
We summarize some important properties of the functor~$T_{\fS}$.

\begin{thm} $($\cite{KisinCrys,MR2388556}$)$
  \label{thm:kisin-module-results}
  \begin{enumerate}
\item The functor $T_{\fS}$ from $\KisinFree$ to
  $\Rep_{\Zp}(G_\infty)$ is exact and fully faithful.

 \item For any $\M \in \KisinFree$ of rank $d$, the functor $T_{\fS}$ restricts to a bijective equivalence of categories between
  the set of $\varphi$-stable $\fS$-submodules $\frakN \subset
  \mathcal{E} \otimes_{\fS} \M$
  of finite height
and rank $d$, and the set of $G_{\infty}$-stable finite free $\Zp$-sublattices of $V =
T_{\fS}(\fM)[1/p]$.

  \item If $V$ is a semi-stable representation of $G_{K}$ with non-negative
     Hodge--Tate weights in some range $[0,r]$, and $L \subset V$ is a
    $G_K$-stable $\Zp$-lattice, there exists $\M \in \KisinFree$ such that
    $T_{\fS}(\M) \simeq L |_{G_\infty}$.

  \item With notation as in (3), if $D$ is the filtered $(\varphi,N)$-module corresponding to
    the representation $V$, then there is a canonical isomorphism
$$S_{K_0} \otimes_{\varphi,\fS} \fM \cong S_{K_0} \otimes_{K_0} D$$
compatible with $\varphi$ and filtrations, as well as with the monodromy
operator (whose definition on the left-hand side we will not discuss).
\end{enumerate}
\end{thm}

\begin{proof} Exactness in (1) is \cite[Lem.~(2.1.2),
 Cor.~(2.1.4)]{KisinCrys}, while full faithfulness is
 \cite[Prop.~(2.1.12)]{KisinCrys} or \cite[Cor.~4.2.6]{liu-Fontaine}.
Part (2) follows  from \cite[Lem.~(2.1.15)]{KisinCrys} together with the full
 faithfulness of (1).  Part (3) is \cite[Cor.~(1.3.15),
 Lem.~(2.1.15)]{KisinCrys}.  Finally, part (4) is \cite[Cor.~3.2.3]{MR2388556}.
\end{proof}

\begin{defn}
  \label{defn:kisin-module-assoc-to-lattice}
  With notation as in Theorem~\ref{thm:kisin-module-results}(3), we
  say that $\fM$ is the Kisin module attached to the lattice $L$;
  by Theorem~\ref{thm:kisin-module-results}(1) this is well-defined up
  to isomorphism.
\end{defn}

Let $A$ be a finite commutative $\Z_p$-algebra, by which we mean a
commutative $\Zp$-algebra that is finitely generated as a $\Zp$-module. We say $\M$ has a
\emph{natural $A$-action} (or \emph{$A$-coefficients}) if $\M$ is an
$A$-module such that the $A$-action commutes with the $\fS$-action and
$\varphi$-action on $\M$, and such that the $\Zp$-module structures on
$\M$ arising from $\Zp \subset \fS$ and $\Zp \to A$ are the same.  If $\M$ has a natural $A$-action then it is
easy to see that $T_{\fS}(\M)$ is an $A[G_\infty]$-module.

\begin{prop}  Let $A$ be a finite commutative $\Z_p$-algebra.
  \label{prop:kisin-natural-action-stuff}
\begin{enumerate}
\item Suppose $V$ is a semi-stable representation of $G_K$ with
  non-negative Hodge--Tate weights and $L \subset V$ is a $G_K$-stable
  $\Zp$-lattice.  If $L$ is an $A$-module such that the $A$-action
  commutes with the action of $G_K$, then the Kisin module attached
  to $\M$ has a natural $A$-action.

\item If $L_1, L_2$ are lattices with $A$-action as in (1) and $f \col
  L_1 \to L_2$ is an $A[G_\infty]$-module homomorphism, then the map
   $g \col \fM_2 \to \fM_1$ such that $T_{\fS}(g) = f$ is a morphism of
   Kisin modules with natural $A$-action.

\item   If $\fM \in \KisinFree$ has a natural $\O_E$-action, then
    $\fM$ is free as a $\fS\otimes_{\Zp}\O_E$-module.  Furthermore there
    is a natural isomorphism of $\Zp[G_{\infty}]$-modules
$$ T_{\fS}(\fM) = \Hom_{\varphi,\fS}(\fM,\fS^{\ur}) \simeq
\Hom_{\varphi,\fS\otimes_{\Zp}\O_E}(\fM,\fS^{\ur} \otimes_{\Zp} \O_E).$$
\end{enumerate}

\end{prop}
\begin{proof}
  The existence of the natural $A$-action on $\fM$ in (1) follows from
  the equivalence of categories in
  Theorem~\ref{thm:kisin-module-results}(2), and then
  the full faithfulness of $T_{\fS}$ gives (2).
  The first part of (3) follows from the fact that $\fS \otimes_{\Zp} \O_E$ is a semilocal
  ring whose maximal ideals are permuted transitively by $\varphi$
  together with the injectivity of the map $(1\otimes \varphi) \col \fS
  \otimes_{\varphi,\fS} \fM \to \fM$.
  See \cite[Lemma (1.2.2)]{KisinModularity} for details. 

The remainder of the proof concerns the last part of (3).   The
argument that we give is motivated by the proof of \cite[Lem.~(1.4.1)]{MR2373358}.  Fix once
and for all an isomorphism $\eta \col \O_E \simeq \O_E^\vee :=
\Hom_{\Z_p} (\O_E, \Z_p)$ of $\O_E$-modules; our natural isomorphism
will depend on this choice.  Write $\fS_{\O_E} :=\fS  \otimes
_{\Z_p} \O_E$, $\fS ^\ur_{\O_E} :=\fS^\ur  \otimes
_{\Z_p} \O_E $, and
$\O_{\E, E} := \O_\E  \otimes
_{\Z_p} \O_E$.   Further define
$$M = \O_\E \otimes_\fS\M, \qquad \varphi^* M = \O_\E
\otimes_{\varphi, \O_\E } M \simeq  \O_{\E, E} \otimes_{\varphi, \O_{\E, E}} M$$
and 
$$ M ^\vee  = \Hom_{\O_\E} (M, \O_\E), \qquad M^\vee_E =
\Hom_{\O_{\E, E}} (M, \O_{\E, E}).$$ 
Define a $\varphi$-action on $M^\vee_E$ as follows.  For any $f \in
M_E^{\vee}$, let $f^* \in \Hom_{\O_{\E,E}}(\varphi^* M,\O_{\E,E})$ be
the map sending the basic tensor $a\otimes m$ to $a \varphi(f(m))$.  
Note that $\varphi^*= 1 \otimes
  \varphi\col \varphi^*M \to M$ is an $\O_{\E, E}$-linear bijection,
  since $E(u) \in \O_{\E}^{\times}$.
Then we can define $\varphi (f) = f^* \circ (\varphi^*)^{-1}$. 

It is routine to check that $\varphi$ on $M^\vee _E$ is a
$\varphi$-semi-linear map and that $\varphi (f)  \circ \varphi =
\varphi \circ f$.  (In particular, beware that $\varphi(f) \neq
\varphi \circ f$.)
Similarly, we have a $\varphi$-action on $M^\vee$ that also satisfies
$\varphi(f) \circ \varphi = \varphi \circ f$.

Extend our fixed isomorphism $\eta$ to isomorphisms $\eta_{\E}\col \O_{\E,
  E}\simeq \O_{\E} \otimes_{\Z_p} \O_E ^\vee$ and 
$\eta^*\col \Hom_{\O_{\E, E}} (M , \O_{\E, E}) \simeq \Hom_{\O_{\E, E}}
(M , \O_{\E} \otimes_{\Zp}  \O_E^\vee)$ of $\O_{\E, E}$-modules.  If
$g = \sum_i x_i \otimes \lambda_i \in \O_{\E} \otimes_{\Zp}
\O_E^\vee$, we write $\theta(g) = \sum_i x_i \lambda_i(1) \in \O_{\E}$.  
Now we can construct a map $\iota\col M_E ^ \vee \to M^\vee$ as follows: for
each $f \in  M_E^{\vee}$ we set $$\iota(f) (m) = \theta(\eta^* (f)
(m))$$ 
for all  $m \in M.$ 
Equivalently, $\iota(f) = \theta \circ \eta_{\E} \circ f$.   It is easy to see that $\iota(f)$ is
$\O_{\E}$-linear.  
We claim that $\iota$ is an isomorphism of $\O_{\E, E}$-modules,
compatible with $\varphi$-actions.  To see the former, it suffices to
assume that $M = \O_{\E, E}$ because  $M$ is a finite free $\O_{\E,
  E}$-module.  Identifying $\O_E \simeq \Hom_{\O_E}(\O_E,\O_E)$
identifies $\eta$ with an isomorphism $\Hom_{\O_E}(\O_E,\O_E) \simeq
\Hom_{\Zp}(\O_E,\Zp)$ sending $a$ to $\theta \circ \eta \circ a$
(where $\theta$ again denotes evaluation at $1$), and so the special
case $M = \O_{\E,E}$ follows by tensoring this isomorphism with $\O_{\E}$ over $\Zp$.
Checking that $\iota$ is $\varphi$-compatible boils down to checking
that $\iota(f)^* = \theta \circ \eta_{\E} \circ f^*$, which follows
directly from the definition since $\varphi$ commutes with $\theta$
and $\eta_{E}$.

Set $\widehat {\O^\ur_E} := \widehat {\O^\ur} \otimes_{\Z_p} \O_E.  $
We claim that the injection $\Hom_{\varphi, \fS_{\O_E}} (\M,
\fS^\ur_{\O_E}) \hookrightarrow \Hom_{\varphi, \O_{\E, E}} (M ,
\widehat {\O^\ur_E})$ is a bijection.  To see this, first observe that the
$\O_E$-linear map 
\begin{equation}\label{eq:fon-bij}\Hom_{\varphi, \fS} (\M, \fS^\ur_{\O_E}) \into  \Hom_{\varphi, \O_{\E}}
(M , \widehat {\O^\ur_E})\end{equation}
is a bijection: if $g$ is an element of the right-hand side, then the
image of
$g(\M)$ under any 
$\widehat {\O^\ur}$-linear projection  $\widehat {\O^\ur_E} \to
\widehat {\O^\ur}$ must lie in $\fS^{\ur}$ by \cite[Proposition B
1.8.3]{fo4}, hence $g(\M) \subset \fS^{\ur}_{\O_E}$.
Then the claim follows by taking $\O_E$-invariants on both sides of \eqref{eq:fon-bij}. Similarly, we have
$\Hom_{\varphi, \fS} (\M, \fS^\ur) =  \Hom_{\varphi, \O_{\E}} (M , \widehat {\O^\ur})$.

Since $M_E^\vee$ is finite $\O_{\E, E}$-free, we have a canonical
isomorphism $\widehat{\O^\ur _E}\otimes_{\O_{\E, E}}  M_E^\vee  \simeq
\Hom_{\O_{\E, E}} (M , \widehat {\O^\ur_E}) $ sending $\sum_i{a_i }
\otimes f_i \mapsto \sum _i a_i f_i$.  We will now check that this
isomorphism identifies $ (\widehat{\O^\ur _E}\otimes_{\O_{\E,
    E}}M^\vee_E ) ^{\varphi = 1}$ with  $\Hom_{\varphi, \O_{\E, E}} (M, \widehat {\O^\ur_E})$.  The element
 $\lambda = \sum_i a_i \otimes f_i  \in \widehat{\O^\ur _E}\otimes_{\O_{\E,
      E}}M^\vee_E$ is $\varphi$-invariant if and only if
$$\sum_i \varphi(a_i) \otimes \varphi(f_i)= \sum_{i} a_i \otimes f_i,  $$
 and this is equivalent to the identity $\sum_i \varphi(a_i)
 (\varphi(f_i)) (\varphi(x)) = \sum _i {a_i } f_i (\varphi (x))$ for
 all $x \in M$; it suffices to test equality on elements of the form
 $\varphi(x)$ since $\varphi(M)$ spans~$M$.  Recalling that
 $\varphi(f) \circ \varphi = \varphi \circ f$, we see that $\lambda$
 is $\varphi$-invariant if and only if $\sum_i {\varphi(a_i) \varphi
   (f_i(x))} = \sum_i a_i f_i (\varphi (x))$; but this is precisely the
 condition that $f = \sum _i a_i f_i$ is in $\Hom_{\varphi, \O_{\E,
     E}} (M, \widehat{\O^\ur_E})$, as desired.  Similarly, we obtain
 an identification of $ (\widehat{\O^\ur }\otimes_{\O_{\E}}M^\vee )
 ^{\varphi = 1}$ with $\Hom_{\varphi, \O_{\E}} (M , \widehat {\O^\ur})$ as $\O_E$-modules.

From what we have proved above, it suffices to
show that there is a natural isomorphism $(\widehat{\O^\ur_E
}\otimes_{\O_{\E, E }}M_E^\vee ) ^{\varphi = 1} \simeq (\widehat{\O^\ur
}\otimes_{\O_{\E}}M^\vee ) ^{\varphi = 1} $ of
$\O_E[G_\infty]$-modules. But since we have constructed a natural
$\O_{\E,E}$-module isomorphism $\iota\colon M^\vee_E \simeq M^\vee$
compatible with $\varphi$, we see that $$M^\vee_E \otimes_{\O_{\E, E}} \widehat{\O^\ur _E}\simeq M^\vee\otimes_{\O_{\E, E}}\widehat{\O^\ur _E}\simeq M^\vee \otimes_{\O_{\E, E}} (\O_{\E, E} \otimes_{\O_\E} \widehat{\O^\ur})  \simeq M^\vee \otimes_{\O_\E} \widehat{\O^\ur} $$
and the result follows.
\end{proof}

\begin{rem}
  \label{rem:natural-not-canonical}
We stress that because of the choice of isomorphism $\eta \col \O_E
\simeq \O_E^{\vee}$, the isomorphism of
Proposition~\ref{prop:kisin-natural-action-stuff}(3) is natural but
not canonical.   
In fact the functor $T_{\fS,\O_E} \col \fM \leadsto
\Hom_{\varphi,\fS\otimes_{\Zp}\O_E}(\fM,\fS^{\ur} \otimes_{\Zp} \O_E)$
is in some sense the correct version of $T_{\fS}$ for use with
coefficients; for instance it is evidently compatible with extension
of the coefficient field, whereas $T_{\fS}$ is not.  It will be
convenient for us to use $T_{\fS}$ for the most part, e.g.~so that we can
directly apply results from certain references.  Thanks to
Proposition~\ref{prop:kisin-natural-action-stuff}(3), on the occasions
when we need to calculate $T_{\fS}$ we can use $T_{\fS,\O_E}$ instead
(see e.g. Lemmas~\ref{lem: lift rank-1 object} and~\ref{lem:product-of-characters}).
 \end{rem}

\section{The shape of Kisin modules with Hodge--Tate weights in $[0,p]$}\label{sec:shape}

Let $T$ be a $G_K$-stable $\Z_p$-lattice in a  semi-stable
representation $V$ of dimension~$d$ with Hodge--Tate weights in $[0, r]$, and $\fM$ the Kisin
module attached to $T$.  Write $0 \le  r_1
\leq \dots \leq r_d \le r$ for the Hodge--Tate weights of $V$.

We will write $[x_1,\ldots,x_d]$ for the $d \times
d$ diagonal matrix  with
diagonal entries $x_1,\ldots,x_d$.  The aim of this section is to prove the
following:

\begin{thm} \label{shape}
Assume that $K$ is unramified, $V$ is crystalline, $r \le p$, and $p \ge
3$. Then
there exists an $\fS$-basis  $e_1, \dots, e_d$ of $\fM$ such that the
matrix of $\varphi$ is $X\Lambda Y$ where $X$ and $Y$ are invertible
matrices such that $Y$
is congruent to the identity matrix modulo $p$, and where $\Lambda$ is the
matrix $[E(u)^{r_1}, \dots, E(u)^{r_d}]$.
\end{thm}

We proceed in several (progressively less general) steps.

\subsection{General properties of the Hodge filtration}

Let $\D = S_{K_0} \otimes_{\varphi, \fS} \fM$ be the Breuil module
attached to $\fM$. Unless explicitly stated otherwise, we will regard
$\fM$ as a $\varphi(\fS)$-submodule of $\D$ from now on. By
Theorem~\ref{thm:kisin-module-results}(4) (i.e., by
\cite[Cor.~3.2.3]{MR2388556}), $\D$ comes from the
weakly admissible filtered $(\varphi,N)$-module $D_{\st}(V) = (D, \varphi, N, \Fil ^i
D_K)$, in the sense that there is a canonical isomorphism $\D \cong
S_{K_0} \otimes_{K_0} D$ compatible with all structures.
We write $f_\pi \col \D \to D_K$ for
the map induced by $u \mapsto  \pi$.  By \cite[\S 6]{BreuilGriffith}, $\Fil^i \D$ is inductively defined by $\Fil ^0 \D= \D$ and
\begin{equation}\label{formula-filtration}
\Fil^{i} \D= \{x \in \D :  f _\pi(x) \in \Fil ^{i} D_K , N(x) \in \Fil ^{i-1}\D\}.
\end{equation}
Then the filtration $\Fil^i D_K$
coincides with $f_\pi (\Fil ^i \D)$, again by  \cite[\S 6]{BreuilGriffith}.

 Let $\fM^*$ be the
$\fS$-submodule  $\fS \otimes_{\varphi, \fS} \fM\subset \D$. Recall that we have an $\fS$-linear map $ 1 \otimes \varphi \col \M^* \to \M$. Define
$$ \Fil ^i\M^*= \{ x \in \M^* |  (1 \otimes \varphi) (x) \in E(u)^ i \M \}.   $$
\begin{lemma}\label{intersection}  The filtration on $\M^*$ has the
  following properties.
\begin{enumerate}
 \item $\Fil ^i \M^* = \M ^* \cap \Fil ^i \D$.
\item ${\rm{gr}} ^i \M^*$ is finite $\mathcal O_K$-free.
\item ${\rm{rank}}_{\mathcal O_K} {\rm{gr}}^i \M^* = \dim_{K} {\rm{gr}}^i \D. $
\end{enumerate}
\end{lemma}
\begin{proof}
Since $\D= S_{K_0} \otimes_{\varphi, \fS}\M$,  one can
prove (see for example, \cite[\S3.2]{MR2388556}) that
$$\Fil ^i \D= \{ x \in \D : (1\otimes \varphi)(x) \in \Fil ^i S_{K_0} \D \}. $$
Since $\M$  is finite $\fS$-free, (1) then follows from the fact that $\Fil^i S_{K_0}\cap \fS = E(u) ^i \fS. $

From (1) it follows that ${\rm gr}^i \M ^* $ injects in ${\rm gr}^i
\D$, which is a $K$-vector space; this gives (2).

Finally, set $M: = \M ^* \otimes_{\Z_p}\Q_p$ and $\Fil^i M:= \Fil ^i
\M ^* \otimes_{\Z_p }\Q_p= M \cap \Fil ^i \D$.  Observe that $\D = M +
(\Fil^{i+1} S_{K_0})\M^*$,  since $M \subset \D$ is finite $\fS [\frac
1 p]$-free and any $s \in S_{K_0}$ can be written as $s_0 + s_1$ with $s_0 \in K_0 [u] \subset \fS [\frac 1 p]$ and $s_1 \in \Fil ^{i+1} S_{K_0}$.
From this  we deduce that $\Fil ^i \D = \Fil ^i M + (\Fil ^{i+1}
S_{K_0})\M^*$, so ${\rm gr}^i M \simeq {\rm gr}^i \D$  and (3)
follows.
\end{proof}

 Set $M_K : = f_\pi (\M^*)\subset D_K$ and
$\Fil ^i M_K = M_K \cap  \Fil ^i D_K$, so that $\Fil ^i M_K$ is an
$\O_K$-lattice in $\Fil ^i D_K$. By Lemma \ref{intersection}(1),
$f_\pi (\Fil ^i\M^*) \subset \Fil ^ i M_K$ for $i \in \Z_{\ge 0}$.

Consider the positive integers $1 = n_0  \leq  n_1 \leq n_2 \leq \cdots \leq
n_{r_d} \le d $ such that $\dim _K\Fil ^i D_K = d-n_i
+1$. Choose an $\O_K$-basis  ${ e_1 , \dots , e_d }$ of $M_K$ such
that $ e_ {n_i}, \dots , e_d $ forms a $\O_K$-basis of $\Fil ^i M_K$;
the existence of such a basis follows by repeated application of  the following lemma.

\begin{lemma} \label{subbasis} Let $D_K$ be a finite $K$-vector space, $M_K$ an $\O_K$-lattice in $D_K$ and $D'_K \subset D_K$ a $K$-subspace. Then there exists an $\O_K$-basis $e_1 , \dots , e_d$ of $M_K$ such that $\{e_{m},\dots , e_d\} $ is a $ K$-basis of $D'_K$ for some integer $m$.
\end{lemma}

\begin{proof}Consider the exact sequence of $K$-vector spaces $$0 \lto
  D'_K \lto D_K \overset f \lto D''_K \lto 0.$$ Then we get an exact
  sequence $0 \to M'_K \to M_K \to M''_K \to 0 $ where  $M'_K = M _K
  \cap D'_K$ and $M''_K =
  f(M_K)$. Since $D'_K$ is divisible, we see that $M''_K$ is
  torsion free and thus finite $\O_K $-free, so there exists a section
  $s \col M''_K \hookrightarrow M_K$ such that $M_K = M'_K \oplus
  s(M''_K)$.
\end{proof}

\begin{prop}\label{needsurjection} Assume that $f_\pi (\Fil ^i \M^*) = \Fil ^i
  M_K$ for all $i \in \Z_{\ge 0}$.
  \begin{enumerate}
  \item There exists an $\fS$-basis $\hat e_1 ,
  \dots , \hat e_d$ of $\M^*$ such that $f_\pi(\hat e_j) = e_j$ for
  all $j$ and $\hat e_j \in \Fil^i \M^*$ for $j \ge n_i$.

  \item  For any basis as in (1),  the module $\Fil ^{r_d} \M^* $ is generated by $(\hat
  e_1, \dots , \hat e_d)\Lambda^*$, where $\Lambda^*$ is the matrix $[
  E(u)^{r_d-r_1},\ldots,E(u)^{r_d-r_d}]$.
  \end{enumerate}

\end{prop}

\begin{proof}
Since $f_\pi (\Fil ^i \M ^*) = \Fil ^ i M_K$, there exist $\hat
e_1,\ldots,\hat e_d
\in \M^*$ such that $f_\pi (\hat e_j) = e_j $ for all $j$ and $ \hat e_j
\in \Fil ^i \M^* $ for $j \geq n _i $. One easily checks that $\{ \hat
e_i\}$ forms an $\fS$-basis of $\M^*$; this proves (1).
Now define $\widetilde \Fil ^i \M^*$ inductively as  follows: $\widetilde
\Fil ^0 \M^* = \M^*$ and $\widetilde \Fil^ {i}\M^*$ is the
$\fS$-submodule generated by $E(u)\widetilde \Fil ^{i-1}\M^* $ and
$\hat e_{n_i}, \dots , \hat e_d$. It is immediate from this description  that
\begin{equation}\label{tildefil}
\widetilde\Fil ^i \M^* =  \bigoplus_{j=0}^{i-1} (E(u)^{i- j }\fS
\hat e_{n_j} \oplus \cdots \oplus  E(u)^{i -j}\fS \hat e_{n_{j+1}-1 })  \oplus
\bigoplus_{j= {n_{i}}}^d \fS \hat e_j .
\end{equation}

Comparing \eqref{tildefil} with the statement of the Proposition, we
see that we will be done if we can prove that $\Fil ^{r_d} \M^* =
\widetilde\Fil ^{r_d} \M^*$.  In fact we now show
by induction on $i$  that
$\Fil ^ i \M^* = \widetilde\Fil ^ i \M^* $ for $0 \leq i \leq r_d$.
 The statement is clear for $i = 0$. Assume that the statement  is
 true for $i = l$, and let us consider the case $i = l+1$. From the
 construction of $\widetilde \Fil ^ {l+1} \M^*$ we see that
 $\widetilde \Fil ^ {l+1} \M^* \subset \Fil ^{l+1}\M^*$, and so we get
 a surjection $\alpha\col \Fil ^l \M^* / \widetilde \Fil ^ {l+1} \M^* \to
 \Fil ^l \M^*/ \Fil ^{l+1} \M^*$. By \eqref{tildefil} it is
 clear that $\Fil ^l \M^* / \widetilde \Fil ^ {l+1} \M^*=\widetilde
 \Fil ^l \M^* / \widetilde \Fil ^ {l+1} \M^*$ is finite $\O_K $-free
 with rank $n_{l+1}-1= d - \dim_K (\Fil ^{l+1} D_K)$. By Lemma
 \ref{intersection}, we know that  ${\rm{gr}}^l \M^*$ is finite
 $\O_K$-free with ${\rm{rank}}_{\mathcal O_K} {\rm{gr}}^l \M^* =
 \dim_{K} {\rm{gr}}^l \D$, so $\alpha$ is an isomorphism if and only
 if $\dim_{K} {\rm{gr}}^l \D=d - \dim_K (\Fil ^{l+1} D_K)$. But this
 is immediate from the fact that $(\D , \Fil ^i \D)$ has a
 \textit{base adapt\'ee} in the sense of
 \cite[D\'ef. A.1]{BreuilGriffith}, and indeed a  \textit{base
   adapt\'ee}  given as in the display equation
 in the middle of page 223 of \emph{ibid.} Therefore~$\alpha$ is an isomorphism and we have
 $\widetilde \Fil^{l+1}\M^* = \Fil ^{l+1}\M^*.$\end{proof}

\subsection{The range of monodromy}  We retain the notation of the
previous subsection, except that we now let $N$ denote the monodromy
operator on $\D$. In this subsection,  we always regard $\M$ as an
$\varphi(\fS)$-submodule of $\D$. Select a $\varphi(\fS)$-basis $\hat e_1 , \dots, \hat e_d$ of
$\M$ (not necessarily related to the basis of Proposition~\ref{needsurjection}). We have $N(\hat e_1 , \dots, \hat e_d)= (\hat e_1, \dots ,\hat e
_d) U$ with $U$ a matrix with coefficients in $S_{K_0}$. In this
subsection, we would like to control the coefficients of~$U$. Let
$\tilde S = W(k)\llb u ^p, \frac{u ^{ep}}{p}\rrb $, so that $\varphi
(\fS) \subset \tilde S \subset S$ and $N(\tilde S) \subset \tilde S$.  Note that unlike $S$, the ring
$\tilde S$ has the property that if $u^p x \in \tilde S$ for some $x
\in K_0\llbracket u \rrbracket$ then $x \in
\tilde{S}[\frac 1 p]$.

\begin{prop} \label{imageofN} We have $U \in \Md (\tilde S[{\frac 1 p}])$.
 If $V$ is crystalline and $p \ge 3$ then furthermore $U\in u ^p(\Md(\tilde S[\frac 1 p] \cap S) ) $.
\end{prop}

\begin{proof}  Note that  $\{\hat e_1, \dots, \hat e_d\}$ forms an
  $S_{K_0}$-basis of $\D$. Let  $e_i$ be the image of~$\hat e_i$ under
  the natural map $\D \to \D / I_+S\D= D$, where $I_+S= u K_0
  \llb u\rrb \cap S$. Since $D$ has a unique $(\varphi,N)$-equivariant
  section $s\col D \to \D$ (see \cite[Prop.~6.2.1.1]{BreuilGriffith}) we just write
  $e_i$ for $s(e_i)$; obviously $\{e_1, \dots , e_d\}$ forms an
  $S_{K_0}$-basis for $\D$. Let $X \in \Md(S_{K_0})$ be the matrix
  such that $(\hat e_1, \dots, \hat e_d)= (e_1, \dots , e_d )X$.

  We claim that both $X$ and $X^{-1}$ are in $\Md(\tilde S[\frac 1
  p])$. In fact this is a consequence of the proof of
  \cite[Prop.~2.4.1]{liufiltration}, as we now explain.  As in that proof, let $\tilde A \in \Md(\fS)$ denote the
  matrix such that $\varphi(\hat e_1, \dots, \hat e_d) = (\hat e_1,
  \dots, \hat e_d)\tilde A$ in $\fM$; then the matrix of $\varphi$
  on~$\D$
with respect to the same basis is $A = \varphi(\tilde A) \in
\Md(\varphi(\fS))$.   Again as in \emph{loc. cit.} let $A_0 \in
\Md(W(k))$ be the matrix of $\varphi$ on~$\D$ with respect to the
basis $e_1,\ldots,e_d$.  Since $\varphi(E(u))/p \in \tilde{S}^\times$,
observe that the next-to-last paragraph of
\emph{loc. cit.} actually shows that $p^r A^{-1} \in \Md(\tilde S)$
and that $A_0 A^{-1} = I_d +
\frac{u^p}{p^r} Y$ with $Y \in  \Md(\tilde S)$ (note that the matrix
$Y'$ in \emph{loc. cit.} is
actually in $\Md(\varphi(\fS))$).

The main part of the argument in \emph{loc. cit.} shows that $X = X_0 +
\sum_{i=0}^{\infty} \frac {u^{p^{i+1}}}{p^r} Z_{i}$ where $X_0 = A_0
A^{-1}$ and $Z_{i}$ is
defined by the formula
$$Z_{i} = A_0 \varphi(A_0) \cdots \varphi^{i}(A_0) \varphi^{i+1}(Y)
\varphi^{i}(A^{-1}) \cdots \varphi(A^{-1}) A^{-1}.$$
From the previous paragraph the matrices $p^r X_0$ and
$p^{r(i+1)} Z_{i}$ are all in $\Md(\tilde S)$.
Choose any $i_0 \ge 1$ such that
$p^{i}\geq er(i+2-i_0) $ for all $i \geq 0$.  Then $p^{r i_0} \cdot \frac{u ^{p^{i+1}}}{p
  ^{r}} Z_i \in \tilde S$ for $i \geq 0$, and $p ^{ri_0} X \in
\Md(\tilde S)$, as desired.  The argument for $X^{-1}$ is essentially the same,
beginning from an analysis of $AA_0^{-1}$ instead of $A_0 A^{-1}$,
\textit{cf.} the last paragraph of \emph{loc. cit.}

Since $N(\hat e_1, \dots, \hat e_d) = N((e_1, \dots, e_d )X)$ we
compute that $$U = X^{-1} B X + X^{-1} N(X)$$ where $B \in \Md(K_0)$ is the matrix of
$N$ acting on $e_1,\ldots,e_d$.  Since $N(X) \in \Md(\tilde S)$
(indeed it is contained in $u^p \Md(\tilde S)$) this
completes the argument in the semi-stable case.

Suppose for the rest of the argument that $V$ is crystalline, so that
$B=0$, $U = X^{-1} N(X)$, and $U \in u^p \Md (\tilde S[\frac{1}{p}])$. Write $U
= u ^p U'$; we have to show that $U' \in \Md (S)$. Here we use the
argument in the proof of \cite[Prop.~2.4.1]{liulattice2}\footnote{The hypothesis that $p \ge 3$ is
  required by the argument in \cite[Prop.~2.4.1]{liulattice2}.  In
  fact this is the only place in the proof of Theorem~\ref{shape} that
the hypothesis $p \neq 2$ is used.}, and we freely use
the notation of that item; in particular for any $x \in \D$ we define
\begin{equation}\label{eq:tau-on-D}
\tau(x) = \sum_{i=0}^{\infty} \gamma_i(t) \otimes N^i(x)
\end{equation}
Recall that the element $t$ is defined in  Section
\ref{sec:p-adic-period-rings}; since the topological generator $\tau \in  \hat G_{p^{\infty}}$
acts trivially on $t$, one can recursively define $\tau^n(x)$.

Suppose that $x \in \fM$.
The formula (2.4.2) of \emph{ibid.} and the comments
immediately following it show that $(\tau -1)^n (x) \in u
^p B^+_\cris \otimes _{\varphi,\fS} \M$ 
and  $(\tau-1)^n (x) \in
I^{[n]}W(R) \otimes_{\varphi , \fS} \M$. We claim that $(\tau -1)^n
(x) \in u^p I^{[n]}W(R) \otimes_{\varphi , \fS} \M $. In fact, if $y
\in u^p B^+_\cris \cap I^{[n]}W(R)$ then by
\cite[Lem.~3.2.2]{liu2-divisible} we have $y = u ^p z$ with $z \in W(R)$.
Since $uw \in \Fil^n W(R)$ with $w \in W(R)$ implies $w \in \Fil^n W(R)$,
it follows from $u ^p z \in I^{[n]}W(R)$ that $z \in I
^{[n]}W(R)$,   and this proves the claim.

Since ${(\tau -1)^n }(x)/u^p$ is in $ I ^{[n]} W(R)$, it follows
exactly as in the final paragraph of the proof of
\cite[Prop.~2.4.1]{liulattice2} that the elements  $(\tau -1)^n(x)/
(n t u ^p)$ lie in $A_\cris \otimes_{\varphi, \fS}\M$ and tend to
$0$ as $n \to \infty$.   (Recall from Section~\ref{sec:p-adic-period-rings} that $I^{[n]} W(R)$
is a principal ideal generated by $(\varphi(\gt))^n$.)
 Therefore the sum
$$ \sum\limits_{n=1}^\infty (-1)^{n-1} \frac{(\tau -1)^n }{n t u ^p}(x)$$
converges in $A_{\cris} \otimes_{\varphi,\fS} \fM$.  But by (2.4.3)
and (2.4.4) of \emph{ibid.} this sum is precisely $N(x)/u^p$.   Since $A_{\cris} \cap \tilde
  S[\frac 1 p] \subset A_\cris \cap S[\frac 1 p] =S$ (e.g. by recalling that $S$ is the subring of
$G_{\infty}$-invariants in $\Acris$), we are done.
\end{proof}

\begin{rem}
  \label{rem:question-about-U}
  It is possible that the matrices $U$ and $U'$ in the preceding proof
  are in $\Md(\tilde S)$, but we do not know how to show it.
\end{rem}

For later use, we record the conclusion of the next-to-last paragraph
of the preceding proof (with $n=1$) as a separate corollary.

\begin{cor}
  \label{cor:tau-corollary}
  If $V$ is crystalline and $p \ge 3$, then for any $x \in \M$ there
  exists $y \in W(R) \otimes_{\varphi,\fS} \M$ such that $\tau(x) - x
  = u^p \varphi(\mathfrak{t}) y$, with $\tau(x)$ as
  in~\eqref{eq:tau-on-D} and $\mathfrak{t}$ the element defined in the
  last paragraph of Section~\ref{sec:p-adic-period-rings}.
\end{cor}

Write $S' = \tilde S[\frac 1 p]\cap S$,  and let $\mathcal{I}_l$ denote the ideal $\sum_{m=1}^{l}
  p^{l-m} u^{pm} S'$ in $S'$.  If $x \in \fM$, write $x =
(\hat e_1 , \dots, \hat e_d) \cdot v$ with $v$ a column vector whose
entries lie in $\varphi(\fS)$, and with $(\hat e_1, \dots ,\hat e
_d)$ viewed as a row vector.  Let $v_l$ be the column vector such
that $N^l(x) = (\hat e_1, \dots ,\hat e
_d) v_l$.  

\begin{cor}
  \label{cor:technical1}
  Suppose that $V$ is crystalline and $p \ge 3$ .Then $v_l$ has
  entries in  $\mathcal I_l$.
\end{cor}

\begin{proof}
  We proceed by induction on $l$.  For $l=1$ we have $v_1 = U \cdot v +
  N(v)$, and since $U \cdot v$ and $N(v)$ both have entries in $u^p
  S'$ (the former by   Proposition~\ref{imageofN}) the
base case follows.

Suppose the statement is true for $l$,
  and consider the case $l+1$.  We have the recursion formula
$$ v_{l+1} = U \cdot v_{l} + N(v_l),$$
and it suffices to show that the two terms on the right-hand side of
the recursion both have entries in $\mathcal I_{l+1}$.  This is immediate for $U
\cdot v_l$ since $u^p \mathcal{I}_l \subset \mathcal{I}_{l+1}$ and $U \in u^p \Md(S')$.
For the other term we must show that $N(\mathcal{I}_l) \subset \mathcal I_{l+1}$.

Observe that if $z \in S'$ then
  $N(z) \in pS'$.  Indeed,  since $z \in K_0 \llbracket u^p \rrbracket$ and
$N(u^{pi}) = -piu^{pi}$, the valuation of the coefficient of $u^{j}$ in
$N(z)/p$ is at least the valuation of the coefficient of $u^{j}$ in
$z$ for any $j \ge 0$, and we have $N(z)/p \in S'$.  As a
consequence we see that
$$ N(p^{l-m} u^{pm} z) = p^{l-m} (-pm u^{pm} z + u^{pm} N(z)) \in
p^{l+1-m} u^{pm} S' \subset
\mathcal I_{l+1}$$ and the induction is complete.
\end{proof}

In the remainder of this subsection, we prove two  technical lemmas for the next
subsection.  In Lemma~\ref{technical2} we assume for simplicity that $K=K_0$
is unramified, although the analogue for general $K$ can be proved by
exactly the same argument.


\begin{lemma}\label{technical2} Assume that $K$ is unramified.  Suppose that $y \in \mathcal I_l$ for
  some $1 \le l \le p$, and
write $y   = \sum\limits_{i =0} ^\infty a_{i} (u -\pi)^i $ with $a_{i}\in K$.
Then we have $a_{i} \in W(k)$ for $0 \leq i \leq p$. More precisely we
have
$p ^{p+ l-1 } \mid a_{0}$, $p ^{p+l -i} \mid  a_{i}$ for $1 \leq
i \leq p-1$ and $p ^{l-1} \mid a_{p}$.
\end{lemma}
\begin{proof}  By hypothesis we have $y =  \sum\limits_{m =1}^l p
  ^{l-m} u ^{pm}z_{m} $ with $z_{m } \in S'$. We can write $z_m  = \sum\limits_{j =0}^\infty \frac{b_{j, m}u ^{p j}}{(pj)!}$ with $b_{j, m }\in W(k)$. Then
\begin{eqnarray*}
 y &= & \sum_{m=1}^l p ^{l-m} \left(\sum\limits_{j =0}^\infty b_{j, m  } \frac{u ^{p (j+m) }}{(pj)!}\right) \\
&= & \sum_{m=1}^l\sum\limits_{j =0}^\infty p ^{l-m} b _{j, m }\frac{(u-\pi +\pi)  ^{p (j+m)}}{(pj)!} \\
&= & \sum_{m=1}^l\sum_{j =0}^\infty p ^{l-m}\frac{b_{j, m }}{(pj)!} \left ( \sum_{i=0}^{p(j+m) }  {{p(j+m)}\choose {i}} (u-\pi)^i \pi ^{p (j+m)-i}  \right )\\
&= & \sum _{i =0}^\infty  \left (\sum _{m=1}^l \sum_{j \geq s_{i , p ,
      m }}\frac{b_{j, m } \pi^{p(j+m )-i} p^{l -m }}{(pj)!} {{p(j+m )}\choose {i}} \right )  (u-\pi)^i,
\end{eqnarray*}
where $s _{i, p , m}= \max \{0, i/p-m\}$. Since we only consider $a_i$
for $0 \leq i \leq  p$, we have $s_{i, p, m }= 0 $ in all our cases.
Note that $p^{pj}/(pj)! \in \Zp$ for all $j \ge 0$.  We first observe that  $v_p (a_0) \geq (p-1)m +l\geq p -1 +l $ because $m \geq 1$. If $1 \leq i \leq p-1$ then $p$ divides  ${{p(j+m)}\choose {i}}$. So we get $v_p (a_i)\geq pm-i +l -m +1\geq p+l -i $. Finally, we have $v_p (a_p)\geq pm +l -m -p \geq l-1 .$
\end{proof}

\begin{lemma}\label{technical3} We have $N^l ((u -\pi)^k)  = \sum \limits_{m
    =0}^k c_m \pi^m (u -\pi)^{k-m}$ for some $c_m \in \Z$.

\end{lemma}
\begin{proof} We induct on $l$, with trivial base case
  $l=0$. Assume that the statement is true for $l$. Then
\begin{eqnarray*}
N^{l+1} ((u -\pi)^k) &=&  N \left(\sum \limits _{m=0}^k c_m \pi^m (u -\pi) ^{k-m}\right)\\
&=& \sum_{m=0}^k c_m \pi^m (k-m) (u-\pi)^{k-m-1}(-u+\pi-\pi)
\end{eqnarray*}
which rearranges to
$$\sum_{m=0}^k c_m(m-k) \pi^m (u-\pi)^{k-m} + \sum_{m=0}^k c_m (m-k)\pi ^{m+1} (u-\pi)^{k-m -1}.$$
The induction follows.
\end{proof}

\subsection{The proof of Theorem~\ref{shape}}
Retain the notation of the previous subsections, but assume now that
$K=K_0$ is unramified, $V$ is crystalline, and $r \le p$.  Recall that
$\pi$ denotes our fixed choice of uniformiser in $W(k)$.  The
essential remaining input that we need for the proof of
Theorem~\ref{shape} is the statement that
$f_\pi (\Fil ^i \M^*) = \Fil ^i M_K  $ for $i \in \Z$
when $p \ge 3$.  The proof of that statement is the key point where
the hypothesis $r \le p$ is used (see Remark~\ref{rem:why-not-p+1} below).  
We begin the proof with the following lemma.

\begin{lemma}\label{unramified-only} Assume that $K$ is unramified.  There exists a $\varphi(\fS)$-basis $\e_1 , \dots , \e_d\in \M$ such that for $0 \leq i \leq r$,
$f_\pi( \e_{n_i}), \dots , f_\pi (\e_d)$ forms a $\O_K$-basis for $\Fil ^ i M_K. $
\end{lemma}

\begin{proof} There
  exists an $\fS$-basis $\e'_1 , \dots, \e'_d $ of~$\M^*$ such that
  $f_\pi (\e'_{n_i}), \dots,  f_\pi (\e'_d)$ forms an $\O_K$-basis of
  $\Fil ^i M_K$ for all $0 \le i \le r$.  (Choose any basis of $M_K$ as
  in the sentence preceding
  Lemma \ref{subbasis}, and lift it to $\M^*$.)
Select any $\varphi (\fS)$-basis $\tilde \e_1 ,\dots,
  \tilde \e_d$ of $\M$. We have $(\e'_1 , \dots , \e'_d) = (\tilde
  \e_1, \dots, \tilde \e_d) B$ where $B \in \Md(\fS)$ is an
  invertible matrix.  Let $B_0= f_\pi (B)\in \GL_d (W(k))$ and set
  $(\e_1, \dots , \e_d):  = (\tilde \e_1, \dots, \tilde \e_d) B_0
  $.  Evidently $\e_1,\dots, \e_d$ is a $\varphi(\fS)$-basis of $\M$. Also note that
$f_\pi (\e_i)= f_\pi (\e'_i)$, so $\e_1,\dots, \e_d$ is just the basis we need.
\end{proof}
\begin{remark} The fact that $B_0$ has entries in $W(k)$ in the above lemma makes essential use of the hypothesis that
  $K$ is unramified. We are not aware of any way to extend this lemma to the
  case of a ramified base.
\end{remark}

\begin{prop}
  \label{prop:surjectivity}
  Suppose that $K$ is unramified, $V$ is crystalline, $r \le p$ and $p
  \ge 3$.  Let $\e_1,\ldots,\e_d$ be a basis of $\fM$ as in
  Lemma~\ref{unramified-only} .  Then there exists an $\fS$-basis
  $\e'_1,\ldots,\e'_d$ of $\fM^*$ with the properties that
  $f_{\pi}(\e'_j) = f_{\pi}(\e_j)$ and
  $\e'_j - \e_j \in p \sum_{j'=1}^{d} \fS \e_{j'}$ for all $1 \le j \le d$, and moreover $\e'_j \in \Fil^i
  \fM^*$ whenever $n_i \le j < n_{i+1}$ (taking $i = r_d$ when $j \ge
  n_{r_d}$).

 In particular  $f_\pi (\Fil ^i \M^*) = \Fil ^i M_K$ for all $i \ge 0$.
\end{prop}

\begin{proof}
Let $\e_1,\ldots,\e_d$ be a basis of $\fM$ as in Lemma~\ref{unramified-only},
and  set $e_j = f_\pi (\e_j) \in M_K$.  The desired statement is only
nontrivial for $1 \le i \le r$.  To prove $f_\pi (\Fil ^i \M^*) = \Fil ^i M_K$ for $1 \leq i \leq r$, we consider the following assertion:

$(\star)$ For each $i= 1, \dots , r$, there exist $\e^{(i)}_{n_i},
\dots, \e^{(i)}_d \in \Fil ^i \M^*$ such that for all $n_i \leq j \leq
d$ we have
$$\e^{(i)}_j = \e_j  + \sum_{n=1}^d \sum_{s=1}^{i-1 }\alpha^{(i)}_{j , n, s
}p^{p-s} (u - \pi)^s \e_n$$ with $\alpha_{j, n, s} ^{(i)}\in W(k)$.  (Recall
that the integers $n_i$ are defined immediately above the statement
of Lemma~\ref{subbasis}.)  Since $f_{\pi}(\e_j^{(i)}) = f_{\pi}(\e_j)
= e_j$, this assertion is sufficient to establish the result, taking
$\e'_j = \e_j^{(i)}$ whenever $n_i \le j < n_{i+1}$ (and using $i = r_d$ when $j \ge
  n_{r_d}$.

We will prove the statement $(\star)$ by induction on $i$. Let us
first treat the case that $i=1$. We just set $\e^{(1)}_j = \e_j $ for
$n_1 \leq j \leq d$. Note that $f_\pi (\e^{(1)}_j) = f_\pi (\e_j) =
e_j \in \Fil ^1 M_K \subset \Fil^1 D_K$. By the construction of $\Fil
^1 \D$, we see that $\e^{(1)}_j \in \Fil^1 \D$, and therefore $
\e^{(1)}_j \in
\M^* \cap \Fil ^1 \D = \Fil ^1 \M^*$. This settles the case that $i
=1$.

Now assume that $(\star)$ is valid for some $i  < r$, and let us consider
the case $i+1$.  Set $H(u)= \frac{u-\pi}{\pi} $. If $n_{i+1} \le j \le d$ we set
$$\tilde \e_j ^{(i+1)} :  = \sum _{l=0} ^{i} \frac{H(u)^l N^l (\e_j ^{(i)})}{l !}.$$
We claim that $\tilde \e^{(i+1)}_j \in \Fil ^{i+1}\D$.  Since
$f_{\pi}(\tilde  \e^{(i+1)}_j ) = e_j \in \Fil^{i+1} D_K$, from~\eqref{formula-filtration} it suffices to
check that $N(\tilde  \e^{(i+1)}_j ) \in \Fil^i \D$.   One computes,
after rearranging, that
$$ N(\tilde  \e^{(i+1)}_j ) = \frac{H(u)^i N^{i+1}(\e_j^{(i)})}{i!} +
\sum_{l=1}^i \frac{(1+N(H(u))) H(u)^{l-1} N^l(\e_j^{(i)})}{(l-1)!}.$$
Now the claim follows from the facts that $N^l(\e_j^{(i)}) \in
\Fil^{i-l} \D$ (apply \eqref{formula-filtration} again, together with
the inductive assumption that $\e_j^{(i)}\in\Fil^i\M^*\subset\Fil^i\D$) and $H(u)^l \in \Fil^l S_{K_0}$,
together with the observation that $1 + N(H(u))= 1
-u/\pi \in \Fil ^ 1 S_{K_0}$.

Now by induction, we have
$$\tilde \e_j ^{(i+1)}- \e^{(i)}_j  = \sum _{l=1} ^i  \frac{(u -\pi) ^l
}{\pi^l l !} N^l \left ( \e_j  + \sum_{n=1}^d \sum_{s=1}^{i-1
  }\alpha^{(i)}_{j , n, s}p^{p-s} (u - \pi)^s \e_n \right )$$
which rearranges to
\begin{multline}\label{eq:big}\tilde \e_j ^{(i+1)}- \e^{(i)}_j  =  \sum_{l=1} ^i \frac{(u -\pi) ^l }{\pi ^l l !} N ^l (\e_j) +\\
 \sum_{n=1}^d \sum _{l =1}^i \sum_{s=1 }^{i-1}\frac{p ^{p -s}
   (u-\pi)^{l}}{\pi^l l!} \alpha^{(i)}_{j, n , s}\sum_{t=0}^l \binom{l}{t}
 N^{l-t} ((u-\pi )^s ) N^{t}(\e_n).\end{multline}
Now write $N^{t}(\e_n)  = \sum\limits_{k=1}^d \sum\limits_{m=0}^\infty
c^{n, t}_{m , k} (u-\pi)^m\e_k $ with $c^{n, t}_{m , k} \in
K_0$.  Using Lemma \ref{technical3} and noting that we always have $l \ge 1$, we can write
$$\tilde \e_j ^{(i+1)}=  \e^{(i)}_j + \sum_{k=1}^d \sum_{m = 1}^\infty
b_{m, k} (u -\pi)^m \e_k $$
for some elements $b_{m,k} \in K_0$.  Now we remove all terms of $(u
-\pi)$-degree at least $i +1$ from this
expression, and define
$$\e^{(i+1)}_j= \e^{(i)}_j + \sum\limits_{k=1}^d \sum\limits_{m
  = 1}^{i} b_{m, k} (u -\pi)^m \e_k.$$
Since $(u-\pi)^{i+1} = E(u)^{i+1} \in \Fil^{i+1} S_{K_0}$, we see that $\e_j ^{(i+1)}$ is still in
$\Fil^{i+1}\D$.  Comparing with $(\star)$, it remains to prove that $p
^{p-m} \mid b _{m, k }$, which we do by showing that every occurrence
of $(u-\pi)^m \e_k$ on the right-hand side of \eqref{eq:big} has coefficient
divisible by~$p^{p-m}$.  There are two cases to consider.

We begin with terms coming from the first sum $\sum\limits_{l=1} ^i \frac{(u -\pi) ^l
}{\pi^l l !} N ^l (\e_j)$  in \eqref{eq:big}. By
Corollary~\ref{cor:technical1} and Lemma~\ref{technical2}, each term
coming from this sum is of the form $\frac{(u-\pi)^l}{\pi^l l!} \cdot a_h
(u-\pi)^h \e_k$ with $l+h \le i$, and with $a_h$ as in
Lemma~\ref{technical2}  applied to
$\mathcal I_l$.  In all cases $a_h$ is divisible by
$p^{p+l-h-1}$, and so this occurrence of $(u-\pi)^{l+h} \e_k$ has
coefficient divisible by $p^{p-h-1}$.  Since $l \ge 1$, the claim
follows in this case.

 For the large second sum in \eqref{eq:big}, by Corollary~\ref{cor:technical1}, Lemma~\ref{technical2},
 and Lemma~\ref{technical3}, each term coming from this sum is of the
 form
$$ \alpha^{(i)}_{j,n,s}\frac{p^{p-s} (u-\pi)^l}{\pi^l l!} \binom{l}{t}
\cdot \left[  c_m \pi^m (u-\pi)^{s-m} \right]  \cdot \left[ a_h
  (u-\pi)^h \right] \e_k$$ with $l+(s-m)+h \le i$, with $c_m \in \Z$ as in Lemma~\ref{technical3}, with
$a_h$ as in Lemma~\ref{technical2} applied to $\mathcal I_t$ if $t\ge 1$, and with $a_h =
\delta_{k,n} \delta_{h,0}$ if $t = 0$.  (Here $\delta_{x,y}$ is $1$ if
$x=y$ and $0$ otherwise.)  In all cases we have $a_h \in W(k)$, which
is all that we will need here.  In particular this occurrence of
$(u-\pi)^{l + s-m+h}$ has coefficient divisible by $p^{p-s}
\pi^{m-l}$, or equivalently by $p^{p-s+m-l}$.    Since $h \ge 0$, this
gives what we need.
\end{proof}

\begin{rem}
  \label{rem:why-not-p+1}
  If we had $r=p+1$, then the induction in the above argument would
  fail when trying to deduce the case $i=p+1$ from the case $i=p$.   Indeed if we had $i=p$ in the last
  paragraph of the proof, then the term with $l = p$, $t=h=0$, and
  $m=s$ and $k = n$ would have the form $\alpha^{p}_{j,n,s} c_s
  \frac{p^{p-s} (u-\pi)^p}{\pi^p p!} \pi^s \e_n$, whose coefficient
  need not be in $W(k)$.
\end{rem}

Combining Propositions~\ref{needsurjection} and~\ref{prop:surjectivity} immediately gives
the following.

\begin{cor} \label{key7}
  Suppose that $K$ is unramified, $V$ is crystalline, $r \le p$, and
  $p \ge 3$.   There exists an $\fS$-basis $\hat e_1 ,
  \dots , \hat e_d$ of $\M^*$ such that $\Fil ^{r_d} \M^* $ is generated by $(\hat
  e_1, \dots , \hat e_d)\Lambda^*$, where $\Lambda^*$ is the matrix $[
  E(u)^{r_d-r_1},\ldots,E(u)^{r_d-r_d}]$.
\end{cor}

Finally, we can prove Theorem~\ref{shape}, which we re-state here for
the convenience of the reader.

\begin{thm} \label{shape-redux}
Assume that $K$ is unramified, $V$ is crystalline, $r \le p$, and $p \ge
3$. Then
there exists an $\fS$-basis  $e_1, \dots, e_d$ of $\fM$ such that the
matrix of $\varphi$ is $X\Lambda Y$ where $X$ and $Y$ are invertible
matrices such that $Y$
is congruent to the identity matrix modulo $p$, and where $\Lambda$ is the
matrix $[E(u)^{r_1}, \dots, E(u)^{r_d}]$.
\end{thm}

\begin{proof}
Let $\e_1,\ldots,\e_d$ be a basis of $\fM$ as in
Lemma~\ref{unramified-only}.  For each $1 \le j \le d$ choose $i$ such
that $n_i \le j < n_{i+1}$ (taking $i=r_d$ when $n_{r_d} \le j$)
and set
$\e'_j = \e^{(i)}_{j}$ as in the proof of
Proposition~\ref{prop:surjectivity}.  
We have $\e'_j \in \Fil^i \fM^{*}$, and by construction Proposition~\ref{needsurjection}(2)
shows that $\Fil^{r_d} \fM^{*}$ is
generated by $(\e'_1,\ldots,\e'_d)\Lambda^*$.

We now consider $\M$ as an $\fS$-module in its own right, rather than
as a $\varphi(\fS)$-submodule of $\D$. Let
$A$ be the matrix of $\varphi$ on $\M$ with respect to the basis $\e_1,
\dots,\e_d$. Then there exists a matrix $B$ such that $AB = BA =
E(u)^{r_d} I _d$. It follows straightforwardly from the definition of $\Fil^{r_d}
\fM^{*}$ that $(\e_1, \dots, \e_d)B$ forms a basis of $\Fil
^{r_d} \M^*$, and therefore there exists a matrix $X^{-1} \in {\rm
  GL}_d (\fS)$ such that $(\e'_1 , \dots, \e'_d)\Lambda^* X^{-1} = (\e_1,
\dots, \e_d) B $. If we write $(\e'_1, \dots, \e'_d) = (\e_1, \dots ,
\e_d) Y^{-1}$ then we get $Y^{-1}\Lambda^* X^{-1}= B $. Hence $A= X
E(u)^{r_d} (\Lambda^*)^{-1} Y$, and since $E(u)^{r_d} (\Lambda^*)^{-1} =
[E(u)^{r_1},\ldots,E(u)^{r_d}] = \Lambda$ we have $A = X\Lambda Y$.

Finally, observe from the formula for $\e_j^{(i)}$ in $(\star)$ that
$\e'_j - \e_j$ is divisible by $p$ (since the index $s$ in $(\star)$ is always
at most $p-1$).  It follows that $Y$ is congruent to the identity
modulo $p$, as claimed, and $(\e_1,\ldots,\e_d)$ is the basis we want.
\end{proof}

\subsection{Coefficients}
\label{sec:coefficients}

We now prove an analogue of Theorem~\ref{shape} for representations
with nontrivial coefficients.  Assume as before that $K = K_0$ is
unramified, let $E$ be a finite extension of $\Qp$ containing the
images of all the embeddings $K\into\Qpbar$, and let $T$ be a
$G_K$-stable $\O_E$-lattice in a crystalline representation $V$ of
$E$-dimension~$d$ with Hodge--Tate weights in $[0,p]$.  Let $\fM$ be the
Kisin module with coefficients attached to $T$, so that $\fM$ is a
free module of rank $d$ over $(W(k) \otimes_{\Zp} \O_E)\llb u\rrb $ by
Proposition~\ref{prop:kisin-natural-action-stuff}(3).
Write $f = [K_0 : \Qp]$, and assume that $p \ge 3$.

Let $\mathcal S = \{ \hodgetateembedding \col K \into E \}$ be the set of embeddings of
$K$ into $E$.  Fix one such embedding $\hodgetateembedding_0$, and recursively define
$\hodgetateembedding_{s+1}$ to be the embedding such that $\hodgetateembedding_{s+1}^p \equiv \hodgetateembedding_s
\pmod{p}$; these subscripts are to be taken mod $f$, so that $\hodgetateembedding_f = \hodgetateembedding_0$.
Let $\ve_s \in W(k) \otimes_{\Zp} \O_E$ be the unique idempotent
element such that $(x \otimes 1)\ve_s = (1 \otimes \hodgetateembedding_s(x))\ve_s$
for all $x \in W(k)$.  Then we have $\ve_s(W(k) \otimes_{\Zp} \O_E) \simeq \O_E$.

\begin{defn}
  \label{defn:labeled-HT-wts}
The filtered $(\varphi,N)$-module $D$ is a $K\otimes E$-module, and  decomposes as a product $D = D_0
\times \cdots \times D_{f-1}$ with $D_s = \ve_s D$ an $E$-vector
space of dimension $d$. Since $D =
D_K$ we have a similar decomposition of $\Fil^i D_K$ for all $i$.
Write $0 \le r_{1,s} \le \cdots \le r_{d,s} \le p$ for the jumps in
the filtration $\Fil^i D_s := \ve_s(\Fil^i D_K)$ on $D_s$.  The
integers $r_{j,s}$ are  the \emph{$\hodgetateembedding_s$-labeled Hodge--Tate
 weights} of $V$, as defined in Section \ref{subsec:notation}.   Note that the multiset $\{ r_{j,s} : 1 \le j \le
d, \, 0 \le s \le f-1\}$ taken $[E : K_0]$ times is precisely the set of Hodge--Tate weights of
$V$ regarded as a $\Qp$-representation.
\end{defn}

The object $\D$ can be
formed from $\M$ by the same formula as in the preceding section, and since $\M$ is free as an $\fS \otimes_{\Zp}
\O_E$-module, $\D$ is free as an $S_{K_0} \otimes_{\Zp}
\O_E$-module and has a decomposition $\D = \D_0 \times \cdots \times
\D_{f-1}$ with $\D_s = \ve_s \D$.  Similar statements hold for $\M$, $\M^*$,
and $M := M_K$ (with $S_{K_0}$ replaced by $\fS$ and $W(k)$ respectively),
so in particular each $M_s \subset D_s$ is an $\O_E$-lattice.
However, note
that when we regard $\M$ as a $\varphi(\fS)$-submodule of $\M^*$,
we are regarding $\M_{s-1}$ (rather than
$\M_s$) as a submodule of $\M^*_s$ because $\varphi(\varepsilon_{s-1}) = \varepsilon_{s}$.

\begin{thm} \label{shape-redux-coeffs}
Assume that $K$ is unramified, $V$ is crystalline with Hodge--Tate
weights in $[0,p]$, and $p \ge
3$. Then
there exists an $\O_E \llb u \rrb$-basis $\{ e_{j,s}\}$ of $\fM$ such
  that
\begin{itemize}
\item
$e_{1,s},\ldots,e_{d,s}$ is an $\O_E \llb u \rrb$-basis of $\fM_s$ for each
$0 \le s \le f-1$, and
\item we have
$$ \varphi(e_{1,s-1},\ldots,e_{d,s-1}) = (e_{1,s},\ldots,e_{d,s})
X_s \Lambda_s Y_s$$ where $X_s$ and $Y_s$ are invertible
matrices, $Y_s$
is congruent to the identity matrix modulo $p$, and $\Lambda_s$ is the
matrix $[E(u)^{r_{1,s}}, \dots, E(u)^{r_{d,s}}]$.
\end{itemize}
\end{thm}

\begin{proof}
  Setting $\Fil^i M = M \cap \Fil^i D$ as before, we have $\ve_s
\Fil^i M = M_s \cap \Fil^i D_s$, which must therefore be an
$\O_E$-lattice in $\Fil^i D_s$.  Let $1 = n_{0,s} \le n_{1,s} \le
\cdots \le n_{r_d,s} \le d$ be the positive integers such that $\dim_E
\Fil^i D_s = d-n_{i,s} + 1$.  By the same argument as in the
paragraph before Lemma~\ref{subbasis},  there exists an $\O_E$-basis
$e'_{1,s},\ldots,e'_{d,s}$ of $M_s$ such that $e'_{n_{i,s},s},\ldots,e'_{d,s}$
forms an $\O_E$-basis of $\ve_s \Fil^i M$.  Now the same argument as
in Lemma~\ref{unramified-only} produces an $\O_E\llb u^p \rrb$-basis
$\e_{1,s},\ldots,\e_{d,s}$ of $\M_{s-1} \subset \M_s^*$ such that
$f_{\pi}(\e_{n_{i,s},s}),\ldots,f_{\pi}(\e_{d,s})$ forms an $\O_E$-basis
for $\Fil^i M_s$, and $f_\pi(\e_{i,s})=e'_{i,s}$.

Choose any $\O_K$-basis $y_1,\ldots,y_g$ of $\O_E$ with $y_1 =1$.
Then $\{ y_m \e_{j,s} \}_{m,j,s}$ is a $\varphi(\fS)$-basis of $\M$
as in Lemma \ref{unramified-only}, and so
Proposition~\ref{prop:surjectivity}
produces an $\fS$-basis $\e'_{m,j,s}$ of
$\fM^*$ with the properties that $f_{\pi}(\e'_{m,j,s}) = y_m e'_{j,s}$,
$\e'_{m,j,s} - y_m \e_{j,s}
\in p  \sum_{m',j',s'} \fS y_{m'} \e_{j',s'} $, and $\e'_{m,j,s} \in \Fil^i \fM^*$
for $i$ as in the Proposition.

Set $\e''_{j,s} = \ve_s \e'_{1,j,s}$.  From the above we see that
$f_{\pi}(\e''_{j,s}) = e'_{j,s}$, $\e''_{j,s} \in (\Fil^i \fM^*)_s$, and
$\e''_{j,s} - \e_{j,s} \in p \sum_{j'} \O_E \llb u \rrb  \e_{j',s} $, and one
checks easily that $\{\e''_{j,s}\}$ forms an
$\O_E\llb u \rrb$-basis of $\M^*$.  Let $r_d = \max_s \{r_{d,s}\}$.  Now the argument of
Proposition~\ref{needsurjection} proves that
$\Fil^{r_d} \M^*$ is generated over $\O_E\llb u \rrb$ by the elements of the form $E(u)^{r_d - i}
\e''_{j,s}$ where $i$ is determined by $n_{i,s} \le j < n_{i+1,s}$ (or $i = r_{d,s}$
when $n_{r_{d,s}} \le j  \le d$), i.e. where $i = r_{j,s}$.

Let $A$ be the matrix of $\varphi : \fM_{s-1} \to \fM_s$ with respect to the
$\O_E\llb u \rrb$-bases $\e_{j,s}$ and $\e_{j,s+1}$.  
Let $B$ be the matrix such that $AB=BA =
E(u)^{r_d} I_{d}$.  It follows as in the proof of Theorem~\ref{shape}
that the image of $\{ \e_{j,s} \}$ under $B$ forms a basis of
$\Fil^{r_d} \fM^*$.  
It follows as in the proof 
of Theorem~\ref{shape} that the matrix $A$ has the form $X_s\Lambda_s Y_s$,
where the matrix $X_s$ is invertible, 
the matrix $Y_s$ is congruent to $I_{d}$
modulo $p$, 
and $\Lambda_s = [E(u)^{r_{1,s}},\ldots,E(u)^{r_{d,s}}]$.  
Therefore $e_{j,s-1} := \e_{j,s}$ is the basis that we want.
\end{proof}

\begin{rem}
 Theorem~\ref{shape-redux-coeffs} is best possible, in the sense that
 it is false if the Hodge--Tate weight range $[0,p]$ is replaced with
 $[0,r]$ for any $r > p$; see
 Example~\ref{ex:best-possible} for an explanation.
\end{rem}

\section{$(\varphi, \hat G)$-modules and crystalline
  representations}\label{sec: crystalline phi Ghat modules}

We recall that
the theory of $(\varphi,\hat G)$-modules,
introduced by the
second author  in \cite{LiuLattices},
has been used to
classify lattices in semi-stable Galois representations.  In this
section we review the theory of $(\varphi,\hat G)$-modules, and
discuss some properties of the
$(\varphi,\hat G)$-modules arising from crystalline representations.
As in Section~\ref{sec:coeffs}, we allow $K$ to be an arbitrary finite
extension of $\Qp$, and recall that $e = e(K/\Qp)$ is the ramification index
of $K$.

\subsection{$(\varphi, \hat G)$-modules}
Define a subring inside $B^+_\cris$:
$$\mc{R}_{K_0}: =\left\{x = \sum_{i=0 }^\infty f_i t^{\{i\}} : f_i \in S_{K_0} \text{
and } f_i \to 0\text{ as }i \to +\infty \right\}, $$ where $t ^{\{i\}}= \frac{t^i}{p^{\tilde q(i)}\tilde q(i)!}$ and $\tilde q(i)$ satisfies $i = \tilde q(i)(p-1) +r(i)$ with $0 \leq r(i )< p-1$. Define $\hR = W(R)\cap \mc{R}_{K_0}$. One
can show that $\mc{R}_{K_0}$ and $\hR$ are stable under the action of
$G_K$,
and that the $G_K$-action factors through  $\hat G$ (see \cite[\S 2.2]{LiuLattices}).
Recall that the ring $R$ is a valuation ring whose valuation we have
denoted $v_R$, and let $I_+ R= \{x \in R : v_R (x) >0\}$ be the maximal ideal of $R$.

We have an exact sequence
$$0 \lto W(I_+R) \lto W(R) \overset{\nu}{\lto} W(\overline k) \lto 0.$$  By the
discussion in the paragraphs leading up to \cite[Lem.~2.2.1]{LiuLattices}
one can naturally extend $\nu$ to a map $\nu \col B^+_\cris \to W(\overline
k)[\frac{1}{p}]$.

For any
subring $A$ of $B^+_\cris$, we write $I_+A =
\ker(\nu) \cap A$, and we also write $I_+= I_+ \hR$.  Since $\nu(u)=0$, it is not hard to see that $I_+ \fS = u \fS$
and $$I_+S= \left\{ x \in S :  x= \sum\limits_{i = 1 }^\infty a _i \frac{u ^i}{q (i)!},\  a_i \in W(k) \right\},$$  where $q(i)$ satisfies $i = q(i)e + r(i)$ with $0 \leq r(i) <e$.
By \cite[Lem.~2.2.1]{LiuLattices}, one has
$\hR / I_+  \simeq S/I_+S \simeq \fS/ u \fS \simeq W(k)$, and that
$\hR$ is $\varphi$-stable.

\begin{defn}
  \label{defn:phi-ghat-module}
 Following \cite{LiuLattices} and \cite{Car-Liu}, a \emph{$(\varphi, \hat
G)$-module of height $ r$} is a triple $(\M , \varphi_{\M},\hat \cG)$ in which:
\begin{enumerate}
\item $(\M, \varphi_{\M})$ is an (either finite free or  torsion) Kisin module  of height $ r$,
\item $\hat \cG$ is  an $\hR$-semi-linear $\hat G$-action on $\hat \M: =\hR
\otimes_{\varphi, \fS} \M$,
\item the $\hat G$-action commutes with $\varphi_{\hat \M} := \varphi
  \otimes \varphi_{\M}$ on $\hat \M$, i.e., for
any $g \in \hat G$ we have $g \varphi_{\hat \M} = \varphi_{\hat \M} g$,
\item regarding $\M$ as a $\varphi(\fS)$-submodule in $ \hat \M $, we have $\M
\subset \hat \M ^{H_{K}}$, and
\item $\hat G$ acts on the $W(k)$-module $M:= \hat \M/I_+\hat \M\simeq \M/u\M$ trivially.
\end{enumerate}
A morphism between two $(\varphi, \hat G)$-modules is a morphism of $\varphi$-modules that
commutes with the $\hat G$-actions on $\hR \otimes_{\varphi, \fS}
\M$.  We will generally allow $\hat \M$ to denote the $(\varphi,\hat G)$-module
$(\M,\varphi_{\M},\hat \cG)$, and (as usual) we will typically suppress
the subscripts on $\varphi_{\M}$ and $\varphi_{\hat\M}$.
\end{defn}

Let $\hat \M = (\M , \varphi, \hat \cG)$ be a $(\varphi, \hat
G)$-module.   We say that $(\M,\varphi)$ is the \emph{ambient Kisin
  module} of $\hat \M$, and we say that a sequence of $(\varphi, \hat
G)$-modules is exact if the sequence of ambient Kisin modules is
exact.  It turns out that the natural map
$$\M \simeq \fS \otimes_\fS \M \lto \fS \otimes_{\varphi, \fS} \M \lto
\hR \otimes_{\varphi, \fS} \M$$ is always injective (see
\cite[Lem.~3.1.2]{Car-Liu} and the discussion preceding it); as a result we
can
regard $\M$ as a $\varphi(\fS)$-submodule of $ \hR \otimes_{\varphi,
  \fS}\M$, and we always do so.

To a $(\varphi, \hat G)$-module $\hat \M= (\M , \varphi, \hat \cG)$,  we can attach a $\Z_p[G_K]$-module as follows:
$$
\hat T (\hat \M) := \Hom_{\hR, \varphi}(\hR \otimes_{\varphi, \fS}
\M, W(R)) \text{ if } \M \text{ is a finite free Kisin module}
$$
and
$$
\hat T (\hat \M) := \Hom_{\hR, \varphi}(\hR \otimes_{\varphi, \fS}
\M, \Q_p/\Zp \otimes_{\Z_p}W(R))\text{ if } \M \text{ is a torsion Kisin module,}  $$
where $G_K$ acts on $\hat T(\hat \M)$ via $g (f)(x) = g (f(g^{-1}(x)))$
for any $g \in G$ and $f \in \hat T(\hat \M)$. There is  a natural map $\theta\col T_\fS(\M) \to  \hat T (\hat \M)$ induced by
$\frakf\mapsto \varphi (\frakf)$.

Let $A$ be a finite commutative $\Z_p$-algebra. We say $\hat \M$ has a
\emph{natural $A$-action} if the ambient Kisin module $\M$ has a
natural $A$-action that also commutes with the $\hat G$-action on $\hR \otimes_{\varphi, \fS} \M$. If $\hat \M$ has a natural $A$-action then it is easy to see that $\hat T(\hat \M)$ is an $A[G_K]$-module.
Now we summarize some useful results about the functor
$\hat T$.
\begin{thm}\label{thm: old facts on (phi hatG)} $($\cite{LiuLattices,Car-Liu}$)$
\begin{enumerate}
\item There is a natural isomorphism $\theta\col T_\fS(\M) \to \hat
  T(\hat \M)|_{G_\infty}.$
\item The functor $\hat T$ is an anti-equivalence between the category of finite free $(\varphi, \hat G)$-modules and the category of $G_K$-stable $\Z_p$-lattices in semi-stable representations with Hodge--Tate weights in $\{0, \dots, r\}$.
\item The functor $\hat T$ is exact.
\item Let $A$ be a finite $\Z_p$-algebra that is free as a $\Zp$-module, and $L \subset L'$ two
  finite free $A$-modules with an action of $G_K$ such that $L[\frac 1 p]= L'[\frac 1 p] $ is
  a semi-stable representation with Hodge--Tate weights in $\{0, \dots
  ,r\}$. Then there exists
an exact sequence of $(\varphi, \hat G)$-modules
$$  0 \lto \hat\frakL' \lto \hat \frakL \lto \hat \M \lto 0 $$
   such that:
\begin{itemize}
\item $\hat \frakL $, $\hat \frakL'$ are finite free $(\varphi,
   \hat G)$-modules with natural $A$-actions,
\item $\hat \M$ is a torsion $(\varphi,\hat G)$-module with a natural
  $A$-action,
\item $\hat T(\hat \frakL' \into \hat \frakL)$ is the inclusion $L
  \into L'$, and
\item there is a natural isomorphism $L'/L = \hat T(\hat \frakL')/\hat
  T(\hat
  \frakL) \simeq \hat T(\hat \M).$
\end{itemize}
\end{enumerate}
\end{thm}
\begin{proof} Parts (1) and (2) are proved in \cite[Thm.~2.3.1]{LiuLattices}.
The functor $T_{\fS}$ is exact from Theorem~\ref{thm:kisin-module-results}, and then (1) implies the exactness of $\hat
T$.  The proof of \cite[Thm.~3.1.3(3), Lem.~3.1.4]{Car-Liu} gives (4) except for
consideration of the
natural $A$-actions.  In particular if $p^n \hat \M = 0$ then the
snake lemma gives
a natural exact sequence of torsion $(\varphi,\hat G)$-modules \cite[Eq.~(3.1.4)]{Car-Liu}:
$$ 0 \to \hat \M \to \hat \frakL'/p^n \hat \frakL' \to \hat \frakL/p^n
\hat \frakL \to \hat \M \to 0$$
and the isomorphism $\hat T(\hat \frakL')/\hat
  T(\hat
  \frakL) \simeq \hat T(\hat \M)$ is induced by applying $\hat T$ to the left-hand part of
  this sequence.
For the $A$-actions, the proof of \cite[Prop.~3.4.1]{liulattice2} shows that there exist natural $A$-actions on $\hat \frakL$ and $\hat\frakL'$ such that the injection $\iota \col \hat\frakL' \hookrightarrow \hat\frakL$ is also a morphism of $A$-modules and
$\hat T(\iota)\col \hat T (\frakL)\hookrightarrow \hat T(\frakL')$ is just the injection $L \into L'$ as $A[G]$-modules. Hence $\hat \M$ has a natural $A$-action and $\hat T(\hat \M)\simeq L'/L$ as $A[G]$-modules.
\end{proof}

 We highlight the following consequence of Theorem~\ref{thm: old facts
  on (phi hatG)}(4).

\begin{prop}
  \label{prop:reduction-of-Kisin-modules}
   Let $V$ be a semi-stable representation of
$G_K$ with $E$-coefficients and Hodge--Tate weights in
$\{0,\ldots,r\}$, and let $L \subset V$ be a $G_K$-stable
$\O_E$-lattice inside $V$.  Let $\hat \frakL$ be a finite free $(\varphi,\hat
G)$-module with natural $\O_E$-action such that $\hat T(\hat \frakL)
\simeq L$.  Then $\hat \frakL/\m_E \hat \frakL$ is a torsion
$(\varphi,\hat G)$-module with natural $k_E$-action such that $\hat
T(\hat \frakL/\m_E \hat \frakL) \simeq L/\m_E L$.
\end{prop}

\begin{proof}
Write $L' = \frac{1}{\varpi} L$ and $\overline L = L'/ L \simeq
L/\fm_E L$.
Let $\iota \col \hat \frakL' \into \hat \frakL$
be the inclusion of $(\varphi,\hat G)$-modules inducing $L \into L'$,
as provided by Theorem
\ref{thm: old facts on (phi hatG)}(4).

Since $\hat T$ is an (anti-)equivalence of categories, there is an
isomorphism $m\col \hat \frakL \simeq \hat \frakL'$ such that $\hat T(m)$ is
the multiplication-by-$\varpi$ map $L' \simeq
L$.  Now $\hat T(m \circ \iota)$ is multiplication by $\varpi$ on $L'$,
hence $m \circ \iota$ is multiplication by $\varpi$ on $\hat{\frakL}'$,
and we deduce that $\hat{\frakL'} = \varpi \frakL = \fm_E
\frakL$.

Now the rest of Theorem
\ref{thm: old facts on (phi hatG)}(4)
 implies that $\hat \M := \hat \frakL / \fm_E \hat
\frakL$ is a $(\varphi,\hat G)$-module with natural $\O_E$-action such
that
\begin{equation}
  \label{eq:phi-ghat-module-reduction}
\hat T(\hat \M) = \hat T(\hat
\frakL/\fm_E \hat \frakL) \simeq \tfrac{1}{\varpi} L/L \simeq \overline{L}
  \end{equation}
as $\O_E$-modules.
The natural $\O_E$-action on the $(\varphi,\hat G)$-module $\hat \M$ evidently induces a natural $k_E$-action, and the isomorphisms
in~\eqref{eq:phi-ghat-module-reduction} are $k_E$-module isomorphisms.
  \end{proof}

\begin{lemma}\label{lem: reducible rep yield reducible Kisin module}
Let $\hat \M$ be a torsion $(\varphi,\hat G)$-module with natural
$k_E$-action, and assume further that $\hat \M$ arises as a quotient
$\hat \M \simeq \hat \frakL/\hat \frakL'$ of finite free
$(\varphi,\hat G)$-modules with natural $\O_E$-action as in Theorem~\ref{thm: old facts on (phi hatG)}(4).
Suppose that $\Lbar := \hat T(\hat \M)$ sits in a
short exact sequence of $k_E[G_K]$-modules
$$\mathcal{L} \col 0 \lto \Lbar' \lto \Lbar \lto \Lbar'' \lto 0.$$
Then there exists a short exact sequence of $(\varphi, \hat
G)$-modules with natural $k_E$-action
$$ \hat{\mathcal{M}} \col 0 \lto \hat \M '' \lto \hat \M \lto \hat \M' \lto 0$$
such that $\hat T(\hat{\mathcal{M}}) = \mathcal{L}$.
\end{lemma}

\begin{proof} If $G$ is a group, $H < G$ is a subgroup, and $\mathcal{N}$ is a short exact sequence of
$G$-representations, let $\mathcal{N}|_{H}$
denote the short exact sequence of $H$-representations obtained from
$\mathcal{N}$ by restriction.

Let $\M$ be the ambient Kisin module of $\hat\M$ and let $M
=k(\!(u)\!)\otimes_{k \llb u \rrb}\M$.
By the  theory of
\'etale $\varphi$-modules (\cite[Proposition A.1.2.6]{fo4}, and see
also the exposition in \cite[\S 2.2]{liu-Fontaine}), there exists an exact sequence of \'etale $\varphi$-modules with natural $k_E$-actions
\begin{equation}\label{eq:exact4}
0 \lto M '' \lto M \overset {\mathfrak f} \lto M' \lto 0
 \end{equation}
 which corresponds to $\mathcal{L}|_{G_\infty}$ under the functor $T$ of \cite[(2.2.4)]{liu-Fontaine}.
Set $\M' := \frakf(\M)$ and $\M''
 : = \ker(\frakf |_{\M})$. By
\cite[Lem.~2.3.6]{liu-Fontaine} applied to the map $\frakf|_{\M} \col \M
\to M'$ we see that $\M',\M''$ are both Kisin
modules with natural $k_E$-actions, and evidently
$$ \mathcal{M} \col  0 \lto \M'' \lto \M \lto \M' \lto 0$$
is a short exact sequence.
It is easy to check that
$k (\!(u)\!) \otimes _{k \llb u \rrb } \mathcal{M}$  is the short
exact sequence of \eqref{eq:exact4}, so that by \cite[Cor
2.2.2]{liu-Fontaine} the short exact sequence $T_\fS(\mathcal{M})$ is
also isomorphic to $\mathcal{L}|_{G_\infty}$.
It remains to show that
 the short exact sequence of Kisin modules $\mathcal{M}$ extends to a
 short exact sequence of $(\varphi,\hat G)$-modules that yields~$\mathcal{L}$.

By
 \cite[Prop.~3.2.1]{liu-Fontaine}, we have the following commutative
 diagram
$$
\xymatrix{ 0 \ar[r] & \fS ^\text{ur} \otimes _\fS   \M '' \ar[r]  \ar@{^{(}->}[d] & \fS ^\text{ur} \otimes _\fS   \M  \ar[r]  \ar@{^{(}->}[d]^{\iota_\M} & \fS ^\text{ur} \otimes _\fS   \M ' \ar[r]  \ar@{^{(}->}[d]  & 0  \\
0 \ar[r] & \fS ^\text{ur}  \otimes_{\Z_p}  {\Lbar''} ^\vee   \ar[r] &
\fS ^\text{ur} \otimes_{\Z_p}   {\Lbar} ^\vee \ar[r] &   \fS
^\text{ur} \otimes_{\Z_p}   {\Lbar'} ^\vee \ar[r] & 0  }
$$
which is compatible with the $G_\infty$-actions,  $\varphi$-actions,
and $k_E$-actions, and where the superscript ${}^{\vee}$ denotes the
$\Qp/\Zp$-dual; the vertical arrows are injective by \cite[Thm.~3.2.2(2)]{liu-Fontaine}.
Now  tensoring with $W(R)$ and $\hR$ respectively, we get another
commutative diagram
$$
\xymatrix{ 0 \ar[r] & \hR  \otimes _{\varphi, \fS}   \M '' \ar[r]  \ar@{^{(}->}[d] &  \hR  \otimes _{\varphi, \fS}   \M   \ar[r]  \ar@{^{(}->}[d] & \ \hR  \otimes _{\varphi, \fS}   \M ' \ar[r]  \ar@{^{(}->}[d]  & 0 \\
0 \ar[r] & W(R) \otimes _{\varphi, \fS}   \M '' \ar[r]
\ar@{^{(}->}[d] & W(R) \otimes _{\varphi, \fS}   \M  \ar[r]
\ar@{^{(}->}[d]^{ W(R)
\otimes_{\varphi, \fS^{\text{ur}}} \iota_\M} & W(R) \otimes _{\varphi, \fS}   \M '\ar[r]  \ar@{^{(}->}[d]  & 0  \\
0 \ar[r] & W(R)  \otimes_{\Z_p}   {\Lbar''} ^\vee \ar[r] & W(R)
\otimes_{\Z_p} {\Lbar} ^\vee \ar[r]^-f & W(R) \otimes_{\Z_p}  {\Lbar'} ^\vee  \ar[r] & 0 .}
$$
The exactness of the rows and the vertical maps
follow from the facts that $\M''$, $\M$ and $\M'$ are all finite $k
\llb u \rrb $-free modules, and that $\hR / p \hR $ and
$\fS^{\text{ur}}/p\fS^{\text{ur}}$ inject into~$R$ (the latter by
\cite[Proposition B.1.8.3(iv)]{fo4}), 
which is a
domain.

Thanks to the hypothesis that $\hat \M$ is the quotient of two finite free
$(\varphi,\hat G)$-modules, \cite[Lem.~3.2.6]{liulattice2} shows that
the map $ W(R)
\otimes_{\varphi, \fS^{\text{ur}}} \iota_\M$ is equal to the map
$\hat\iota_{\hat\M}$ of the diagram \cite[(3.2.4)]{liulattice2}, and
so in particular is $G_K$-equivariant (see e.g. \cite[Thm.~2.2.2]{liulattice2}).
Note also that the $G_K$-actions in the middle column commute with the $k_E$-actions.

Regarding $ \hR \otimes_{\varphi, \fS} \M'$ and $ \hR \otimes_{\varphi,
  \fS} \M$ as submodules of $W(R) \otimes_{\Zp} \Lbar'^\vee $ and
$W(R) \otimes_{\Zp} \Lbar^\vee$ respectively, we have $ \hR
\otimes_{\varphi, \fS} \M'= (f \circ \hat \iota_\M )(\hR
\otimes_{\varphi, \fS} \M )$. So $\hR \otimes_{\varphi, \fS}\M'$
inherits a $G_K$-action which factors through $\hat G$, and then so
does $\hR \otimes_{\varphi, \fS}\M''$; moreover these
$\hat G$-actions commute with the $k_E$-actions. It is easy to check
that these $\hat G$-actions satisfy the axioms for $(\varphi, \hat
G)$-modules, so we obtain an exact sequence of $(\varphi, \hat
G)$-modules that we call $\hat{\mathcal{M}}$.

It remains to check that $\hat T(\hat{\mathcal{M}}) \simeq
\mathcal{L}$. To see this, we note that $\mathcal{M}$ is the
sequence of ambient Kisin modules underlying $\hat{\mathcal{M}}$, and
$T_\fS(\mathcal{M})$ is isomorphic to $\mathcal{L}|_{G_\infty}$.
We therefore have $\hat T(\hat{\mathcal{M}}) |_{G_{\infty}} \simeq T_{\fS}(\mathcal{M}) \simeq
\mathcal{L} |_{G_{\infty}}$, where the first isomorphism comes from
Theorem \ref{thm: old facts on (phi hatG)}(1).  But by hypothesis the middle map
$\hat T(\hat \M)|_{G_{\infty}} \to L |_{G_{\infty}}$ in that complex
is actually a $G_K$-isomorphism, and it follows that
$\hat T(\hat{\mathcal{M}})\simeq \mathcal{L}$ as short exact
sequences of
$k_E[G]$-modules.  (Suppose $G$ is a topological group,
  $\mathcal{L}',\mathcal{M'}$ are short exact sequences of
  continuous $G$-representations, and $f \col \mathcal{L}' \to \mathcal{M}'$ is an
  isomorphism between $\mathcal{L}'$ and $\mathcal{M}'$ regarded as short exact
  sequences of vector spaces.  If the map in the middle of $f$ is a
  isomorphism of continuous $G$-representations, it follows formally
  that the same is true of the two outer maps, and $f$ is an
  isomorphism from $\mathcal{L}'$ to $\mathcal{M}'$.)
\end{proof}

\begin{rem}
  \label{rem:without-quotient-hypothesis}
  Lemma~\ref{lem: reducible rep yield reducible Kisin module} may well
  remain true without the assumption that $\hat \M$ arises as a quotient
$\hat \M \simeq \hat \frakL/\hat \frakL'$ of finite free
$(\varphi,\hat G)$-modules with natural $k_E$-action, but the proof
would require additional work and we will only need the weaker statement.
\end{rem}

Before continuing, we note one additional consequence of the
relationship between torsion Kisin modules and the theory of
\'etale $\varphi$-modules.

\begin{lem}
  \label{lem:isomorphism-criterion}
  Suppose that $\hat f\col \hat\barM \to \hat\barMp$ is a map of torsion $(\varphi,\hat
  G)$-modules with natural $A$-action, and let $\barM,\barMp$ be the
  ambient Kisin modules of $\hat\barM,\hat\barMp$ respectively.  Then
  $\hat T(\hat f)$ is injective (resp. surjective, an isomorphism) if and only if the induced map $\barM[\frac 1u] \to
  \barMp[\frac 1u]$ is injective (resp. surjective, an isomorphism).
\end{lem}

\begin{proof}
  Note that $\barM[\frac 1u] = k(\!(u)\!) \otimes_{k\llb u \rrb} \barM$
  and similarly for $\barMp$.  Let $f$ denote the map of Kisin modules
  underlying $\hat f$, and $f_u$ the map of \'etale $\varphi$-modules
   $\barM[\frac 1u] \to
  \barMp[\frac 1u]$ obtained by inverting $u$.  By Theorem~\ref{thm: old facts on (phi hatG)}(1), the map $\hat T(\hat f)$ is
  injective (resp. surjective)   if and only if the map $T_{\fS}(f)$
  is injective (resp. surjective).  By
  \cite[Cor.~2.2.2]{liu-Fontaine} the map $T_{\fS}(f)$ is naturally isomorphic
to $T(f_u)$.  But the functor $T$ is an equivalence of abelian categories.
\end{proof}

\subsection{$\tau$-actions for crystalline representations}

We now re-state Corollary~\ref{cor:tau-corollary} using the language
of $(\varphi,\hat G)$-modules.

\begin{prop}\label{prop:should be a corollary} Suppose $p > 2$.  Let
  $L$ be a $G$-stable $\Z_p$-lattice in a crystalline representation and
  $\hat \M $ the $(\varphi, \hat G)$-module corresponding to
  $L$, with ambient Kisin module $\M$. Then for any $x \in \M$ we have $\tau(x) - x \in \hat{\M} \cap
  u^p \varphi(\gt) (W(R) \otimes_{\varphi,\fS} \M)$.
\end{prop}

\begin{proof}
  This is a direct consequence of Corollary~\ref{cor:tau-corollary}, since by \cite[Eq.~(3.2.1) and Prop.~3.2.1]{LiuLattices}, the formula~\eqref{eq:tau-on-D} defines the action of $\tau$ on the
  $(\varphi,\hat G)$-module $\frakL$.
\end{proof}

  We let ${\rm{Rep}}_{E}^{\st, r}(G_K)$
denote the category of finite $E$-vector spaces $V$ with an $E$-linear
$G_K$-action such that $V$ is a semi-stable representation with
Hodge--Tate weights in $\{0, \dots,r \}$. We denote by $\Rep_{\O_E}
^{\st, r}(G_K)$ the category of $G$-stable $\O_E$-lattices inside
objects in $\Rep_E ^{\st, r}(G_K)$, and by $\Rep_{\O_E} ^{\cris, r}$
the subcategory of $\Rep_{\O_E} ^{\st , r}$ whose objects are crystalline.

Now assume that $L$ is in  $\Rep_{\O_E} ^{\cris, r}$ and let $\hat
\barM$ be the $(\varphi, \hat G)$-module corresponding to the
reduction $L / \fm_E L$ via Theorem \ref{thm: old facts on (phi
  hatG)}(4), with ambient Kisin module $\barM$.

\begin{cor}\label{cor:valuation of torsion crystalline}  For any $x
  \in \barM$, there exist $\alpha \in R$ and $y \in R
  \otimes_{\varphi,\fS} \barM$ such that $\tau(x) -x = \alpha y $ and $v_R (\alpha) \geq \frac{p}{p-1} + \frac p e $.
\end{cor}
\begin{proof}  Recall from Section~\ref{sec:p-adic-period-rings} that the
  image of $\gt$ in $R$ has valuation $\frac{1}{p-1}$, from which
it follows that the image of $\varphi(\gt) $ in $R$ has valuation
$\frac{p}{p-1}$.  Since the image of $u$ in $R$ is $\underline \pi$,
  and $v_R(\underline \pi) = \frac 1 e$, the result
  follows from Proposition~\ref{prop:should be a corollary}.
\end{proof}

\section{Kisin modules and $(\varphi,\hat G)$-modules of rank one}
\label{sec:kisin-modules-varphi}

We assume for the remainder of this article that $K/\Qp$ is
unramified, with $f = [K:\Qp]$. In this section we determine the isomorphism classes of $(\varphi,\hat
G)$-modules of rank one, compute their corresponding Galois
representations, and show that they arise as the reductions of
crystalline characters with specified Hodge--Tate weights.

 Recall that $E$ is a finite extension of $\Qp$, with ring
of integers $\O_E$ and residue field $k_E$.  As in
Section~\ref{sec:coefficients}, we fix (again for the remainder of the
article) an embedding $\hodgetateembedding_0 \col K \into E$ and recursively
define $\hodgetateembedding_{s+1} \col K \into E$ so that $\hodgetateembedding_{s+1}^p \equiv
\hodgetateembedding_s \pmod{p}$.  Let $\varepsilon_s \in W(k)
\otimes_{\Zp} \O_E$ be the
idempotent defined in Section~\ref{sec:coefficients}, and if $M$ is any module
that can naturally be regarded as a module over $W(k) \otimes_{\Zp}
\O_E$ we write $M_s$ for $\varepsilon_s M$.

\begin{defn}
  \label{defn:rank-one}
   Suppose $r_0,\ldots,r_{f-1}$ are non-negative integers and $a \in
   k_E^{\times}$.  Let $\barM(r_0,\ldots,r_{f-1}; a)$ be the Kisin
   module with natural $k_E$-action that is rank one
   over $\fS \otimes_{\Zp}
  k_E$ and satisfies
  \begin{itemize}
  \item $\barM(r_0,\ldots,r_{f-1}; a)_i$ is generated by $e_i$, and

\item $\varphi(e_{i-1}) = (a)_{i} u^{r_{i}} e_{i}$.
  \end{itemize}
Here $(a)_i = a$ if $i \equiv 0 \pmod{f}$ and $(a)_i = 1$ otherwise.
(For later use, we extend this notation as follows: if $S \subset \Z$,
we write $(a)_S = a$ if $S$ contains an integer divisible by $f$, and
$(a)_S = 1$ otherwise.)
\end{defn}

The following fact is proved by a standard change-of-variables
argument whose details we omit (but see for instance~the paragraph
before the statement of~\cite[Thm.~2.1]{SavittRaynaud} for an analogous argument).

\begin{lem}
  \label{lem:rankone}
  Any rank one $\varphi$-module over $\fS \otimes_{\Zp}
  k_E$ is isomorphic to (exactly) one of the form $\barM(r_0,\ldots,r_{f-1}; a)$.
\end{lem}

Now let $\hat a \in \O_E$ be a lift of $a$. Let $\M (r_0, \dots ,
r_{f-1}; \hat a )$ be the rank one $\varphi$-module over $\fS \otimes_{\Z_p}\O_E$ such that
\begin{enumerate}
\item $\M(r_0 , \dots, r_{f-1}; \hat a)_i $ is generated by $\e_i$, and
\item $\varphi(\e_{i-1}) = (\hat a)_i (u-\pi) ^{r_i} \e_{i}$.
\end{enumerate}

It is obvious that $\M:= \M(r_0 , \dots , r_{f-1}; \hat a )$ is a
finite free Kisin module such that $\M/\m_E\M = \barM(r_0, \dots,
r_{f-1}; a).$   We would like to show that the
$G_\infty$-representation $T_\fS (\M)$  can be uniquely extended
to a crystalline character of $G_K$.

\begin{lemma}\label{lem: lift rank-1 object} There exists a unique
  $(\varphi, \hat G)$-module $\hat\M := \hat \M(r_0,\ldots,r_{f-1};\hat a)$ such that the ambient Kisin
  module of $\hat\M$ is  $\M (r_0, \dots ,
r_{f-1}; \hat a )$ and $\hat T(\hat \M)$ is a crystalline character.
  The $\hodgetateembedding_s$-labeled Hodge--Tate weight of $\hat T(\hat
  \M)$ is $r_s$.
\end{lemma}

\begin{proof}
The uniqueness is a general fact, combining Theorem~\ref{thm: old
  facts on (phi hatG)}(2) with \cite[Thm.~(0.2)]{KisinCrys}.  For
existence, consider the Kisin module $\frakN(j) =\M(0,\ldots,1,\ldots,0;1)$ where
$r_j=1$ and $r_i = 0$ if $i \neq j$.  This is a Kisin
module of height 1, and it follows from
\cite[Thm.~(2.2.7)]{KisinCrys} that $T_\fS(\frakN(j))$ can be uniquely
extended to a crystalline character $\psi_j$ with  Hodge--Tate weights
in $\{0, 1\}$. By Theorem \ref{shape-redux-coeffs} (or, if one prefers,
from Lemma~\ref{intersection}(3) together with the existence of a
\emph{base adapt\'ee} for $\cD$),  $\psi_j$ has
$\hodgetateembedding_s$-labeled Hodge--Tate weights $0$ if $s \neq j$
and $1$ if $s= j$.

Next consider $\frakN(\hat a) = \M(0,\ldots,0;\hat a)$, and define
$\lambda_{\hat a} = T_{\fS}(\frakN(\hat a))$.  Let $\Zp^{\ur}$
denote the maximal unramified extension of $\Zp$.  Since there exists $x \in \Zp^{\ur}
\otimes_{\Zp} \O_E$ 
with
$\varphi^f(x) = (1 \otimes \hat a) x$, it is easy to check using the
functor $T_{\fS,\O_E}$ that $\frakN(\hat a)$ is the Kisin module
attached to the unramified character of $G_K$ sending arithmetic
Frobenius to $\hat a$.  Now it suffices to show that the Kisin module
associated to the crystalline character $\lambda_{\hat a} \psi_0
^{r_0} \cdots \psi_{f-1}^{r_{f-1}}$ is just $\M (r_0, \dots, r_{f-1};
\hat a)$.  This is a consequence of the following general fact.
\end{proof}

\begin{lem}
  \label{lem:product-of-characters}
   Let $\chi$ and $\chi'$ be two crystalline $\O_E$-characters of $G_K$
   whose Kisin modules $\frakN$, $\frakN'$ are defined by the conditions
\begin{itemize}
\item $\frakN_i$, $\frakN'_i$ are generated by $\e_i$, $\e'_i$ respectively, and
\item $\varphi
   (\e_{i-1})= \alpha_{i} \e_{i}$ and $\varphi(\e'_{i-1}) = \alpha'_i
   \e_{i}$ with $i = 0 , \dots , f-1$ and $\alpha_i,\alpha'_i \in
   \O_E \llbracket u \rrbracket$. 
\end{itemize}
Then
   the Kisin module $\tilde\frakN$ of $\chi \cdot \chi'$ has the form $\varphi
   (\mathfrak{f}_{i-1})= \alpha_ {i}\alpha'_{i} \mathfrak{f}_{i}$,
   with $\mathfrak{f}_i$ a generator of $\tilde\frakN_i$. \end{lem}

\begin{proof}
We compute using the functor $T_{\fS,\O_E}$.  Pick generators $f,f'$ of the rank one $\O_E$-modules $T_{\fS,\O_E}(\frakN)$ and
$T_{\fS,\O_E}(\frakN')$,
and write $\beta_i,\beta'_i$ for the elements $f(\e_i),f(\e'_i)$ in
$\fS^{\ur} \otimes_{\Zp}\O_E$.
Then $\varphi (\beta_{i-1})= \alpha_{i} \beta_{i}$ and similarly for
$\varphi(\beta'_{i-1})$.

Let $\tilde\frakN$ be as in the statement of the lemma, and consider the map
$\tilde f\col \tilde\frakN \to \fS^{\ur} \otimes_{\Zp} \O_E$ which sends $\mathfrak{f}_i$ to
$\beta_{i}\beta'_{i}$.  Evidently $\tilde f \in T_{\fS,\O_E}
(\tilde\frakN)$,
and the latter
is an $\O_E$-character of $G_\infty$. As $\tilde{f}= f \cdot f'$, we see that
$T_{\fS,\O_E}(\tilde\frakN)= (\chi \chi')|_{G_\infty}$ as
$\O_E[G_\infty]$-modules. That is, $\tilde\frakN$ is the Kisin module
associated to $\chi \cdot \chi'$.
\end{proof}

\begin{cor}\label{cor:(phi, hatG) for rank 1}There is a unique
  $(\varphi, \hat G)$-module $\hat \barM :=
  \hat\barM(r_0,\ldots,r_{f-1};a)$ whose ambient Kisin module is
  $\barM(r_0 , \dots ,r_{f-1}; a)$.  Furthermore, $\hat T(\hat \barM)$
  is the reduction of the crystalline character $\hat
  T(\hat\M(r_0,\ldots,r_{f-1};\hat a))$ for any lift $\hat a \in \O_E$
  of $a$.
\end{cor}

\begin{proof} The existence of $\hat\barM$ follows from
  Lemma~\ref{lem: lift rank-1 object} and Theorem \ref{thm: old facts
    on (phi hatG)}(4).  For uniqueness, it suffices to see that the
  action of $\tau$ on $\barM$ is uniquely determined.  Write
  $\tau(e_i) = \alpha_i e_i$ with $\alpha_i \in R$.  We see that
  $\alpha_{i} = \ue^{-pr_i} \varphi(\alpha_{i-1})$, and it follows that
  $\varphi^{f}(\alpha_i) = \alpha_i \ue^{m_i}$ for some integer $m_i$
  which is determined by the $r_j$.  Lemma \ref{valuation} below shows
  that $\alpha_i = c \eta^{m_i}$ for some $c \in k$, where the element
  $\eta \in R$ is defined in Lemma~\ref{valuation}(2). Since $\eta-1
  \in I_{+} R$ and $\Ghat$ must act trivially on $\barM/u\barM$, we
  have $c=1$ and $\alpha_i$ is uniquely determined for all $i$.
\end{proof}

\begin{lemma}\label{valuation}
Recall that $\ue = \ue(\tau)$ is the image in $R$ of $\tau(u)/u $.
\begin{enumerate}
\item
Write $m = p^s m _0 \in \Z_p $ with $ m_0\in\Z_p^\times$ and $s \in
\Z_{\ge 0}$ a
non-negative integer. Then $v_R(\ue^{-m} -1)= p^s(\frac{p}{p-1})$.

\item If $m \in \Z$, then the solutions to the equation $\varphi^f (x)= x \ue^m $  with $x
  \in R$ are precisely the $c\eta^{m} $ where  $\eta= \prod\limits_{n=0}^\infty (\ue^{-1})^{p^{nf}}$ and $c\in k$.

\item If $m \in \Z$ then  $v_R(\eta^{m}-1) = v_R(\ue^m - 1)$.

\end{enumerate}
\end{lemma}
\begin{proof}
(1)  It suffices to prove that $v_R(\ue^m -1) =
p^s(\frac{p}{p-1})$.  If $m=m_0 \in \Zp^{\times}$, then
$v_R(\ue^m - 1) = \lim_{n\to \infty} p^n v_p (\zeta_{p^n}^{m} - 1) =
\frac{p}{p-1}$ where $\zeta_{p^n}^m$ is defined in the usual way for
$m \in \Zp^{\times}$.   For the general case, note that $\ue^m -1 =
(\ue^{m_0})^{p^s}-1= (\ue^{m_0}-1)^{p^s}$.

(2) One checks that (1) implies the convergence of $\eta$ in $R$, and
that $c\eta^m$ is a solution to the equation.  Comparing valuations
on both sides of the equation $\varphi^f(x) = x \ue^m$, one sees that if
$x \neq 0$ then $v_R(x) = 0$; it follows that if $x,y$ are two solutions with the
same image in $\overline{k} \simeq R/\fm_R$, then $x-y=0$.  Also note that since $\ue
\equiv 1 \pmod{\fm_R}$ and $k$ is the fixed field of $\varphi^f$ in $\overline{k}$, the image of $x$ in $\overline{k}$ must lie in~$k$.
It is easy to see that $\eta \equiv 1 \pmod{\fm_R}$, and we
conclude that if $c \in k$ then $c \eta^m$ is the unique solution with
image $c$ in $R/\fm_R$.

(3) Write $\varphi^f(\eta^m - 1) = \eta^m(\ue^m - 1) + (\eta^m-1)$.
Since $\eta^m-1$ has positive valuation, the term on the left-hand
side has greater valuation than the second term on the right-hand
side; therefore the two terms on the right-hand side must have equal valuation.
\end{proof}

Recall that  in
Section~\ref{sec:galois-theory}, for each $\sigma \in \Hom(k,\Fpbar)$ we have defined the fundamental character
$\omega_{\sigma} \col I_K \to \Fpbar^{\times}$ corresponding to $\sigma$.
 Let $\hodgetateembeddingbar_s \col k \into \Fpbar$ be
the embedding obtained by reducing
$\hodgetateembedding_s$ modulo $p$, and for brevity we write
$\omega_s$ for $\omega_{\hodgetateembeddingbar_s}$ (throughout the
rest of the paper).

\begin{prop}
  \label{prop:calculation-on-inertia}
Write $\hat{\barM} = \hat\barM(r_0,\ldots,r_{f-1};a)$
and $\hat{\barMp} = \hat\barM(r'_0,\ldots,r'_{f-1};a')$
 for some $a,a' \in k_E$ and
non-negative integers $r_0,r'_0,\ldots,r_{f-1},r'_{f-1}$.
Let $\barM,\barMp$ denote the ambient Kisin modules of
 $\hat\barM,\hat\barMp$. 

\begin{enumerate}
\item We have $\hat T(\hat\barM)|_{I_K} \simeq \omega_0^{r_0} \cdots
  \omega_{f-1}^{r_{f-1}}.$

\item We have $\hat T(\hat\barM) \simeq \hat T(\hat\barMp)$ if and
  only if  $T_{\fS}(\barM) \simeq T_{\fS}(\barMp)$.

\item The isomorphism in (2) occurs  if and
  only if $a=a'$ and $\sum_{i=0}^{f-1} p^{f-i-1} r_i \equiv
  \sum_{i=0}^{f-1} p^{f-i-1} r'_i \pmod{p^f-1}$.

\end{enumerate}
\end{prop}

\begin{proof}
  (1) By Lemma~\ref{lem: lift rank-1 object} and
  Corollary~\ref{cor:(phi, hatG) for rank 1}, it suffices to check
  that $\overline\psi_s |_{I_K} = \omega_s$, where $\overline\psi_s$ is the
  reduction modulo $p$ of the character $\psi_s$ whose
  $\hodgetateembedding_j$-labeled Hodge--Tate weight is $1$ if $j=s$
  and $0$ otherwise.  By \cite[Prop.~B.3]{conradlifting} 
 we have
 $(\psi_s \circ \Art_K)|_{\O_K^{\times}}  \simeq \hodgetateembedding_s
 |_{\O_K^{\times}}$; comparing with the definition of $\omega_s$, the
 result follows.

(2) Since $K_{\infty}/K$ is totally wildly ramified but the kernels of
mod $p$ characters of $G_K$ correspond to tame extensions, a mod $p$
character of $G_K$ that is trivial on $G_{\infty}$ must be trivial.

(3) Let us first check that the given conditions are sufficient.
Choose any integers $r''_0,\ldots,r''_{f-1}$ such that $r''_i \ge
\max(r_i,r'_i)$ and $\sum_{i=0}^{f-1} p^{f-i-1} r''_i \equiv
\sum_{i=0}^{f-1} p^{f-i-1} r_i \pmod{p^f-1}$, and define $\barMpp =
\barM(r''_0,\ldots,r''_{f-1};a)$.   It is enough to
check that  $T_{\fS}(\barM) \simeq T_{\fS}(\barMpp)$ (for then we must
have $
T_{\fS}(\barMp) \simeq  T_{\fS}(\barMpp)$ by the same argument).
Set $m_i = \frac{1}{p^f-1} \sum_{j=0}^{f-1} p^{f-i-1}(r''_{i+j+1} - r_{i+j+1})$,
which by construction is a non-negative integer for all $i$.  Then
there is a map $f \col \barMpp \to \barM$ sending $e''_i \mapsto u^{m_i} e_i$
(with the obvious meaning for $e''_i$).  Since $f$ is an isomorphism
after inverting $u$, it follows from the theory of \'etale
$\varphi$-modules (as in Section~\ref{sec: crystalline phi Ghat
  modules}) that $T_{\fS}(f)$ is an isomorphism.

In the reverse direction, it follows from (1) that the condition  $\sum_{i=0}^{f-1} p^{f-i-1} r_i \equiv
  \sum_{i=0}^{f-1} p^{f-i-1} r'_i \pmod{p^f-1}$ is necessary.  The
  calculation of the unramified character $\lambda_{\hat a}$ in the proof of
  Lemma~\ref{lem: lift rank-1 object}, together with
  Lemma~\ref{lem:product-of-characters} and Corollary~\ref{cor:(phi,
    hatG) for rank 1}, shows that changing $a'$ must change $\hat
  T(\hat\barMp)$.  Thus for fixed values of $r_0,r'_0,\ldots,r_{f-1},r'_{f-1}$
  and $a$  the isomorphism in (2) holds for at most one
  value of $a'$, and so the necessity of $a=a'$ follows from the
  result of the previous paragraph.
\end{proof}

\begin{ex}
  \label{ex:best-possible}  We can now show that
  Theorem~\ref{shape-redux-coeffs} is best possible.
  Suppose that $V$ is a two-dimensional crystalline representation of
  $G_{\Qp}$ with Hodge-Tate weights $(0,r)$ for some $r > 0$, and
  assume that the reduction mod $p$ of $V$ is reducible.  Possibly after
  extending the coefficients of $V$, it is possible to choose a
  lattice $T \subset V$ with associated Kisin module $ \M$ such that $\barM$ is a direct sum $\barM(h; a)\oplus
  \barM(h'; a')$ for some $h,h'$ with $h+h'=r$.  (This follows by
  essentially the same argument by which it is possible to choose a
  lattice in $V$ whose reduction is split, again after possibly
  extending the coefficients.)  

If the conclusion of
  Theorem~\ref{shape-redux-coeffs} were to hold for the Kisin module
  $\M$, then~$\varphi$ on
  $\barM$ would be nontrivial mod $u$.  It would then follow that
  $\{h,h'\} = \{0,r\}$, and $\overline{V}^{\mathrm{ss}} \cong 1 \oplus
  \overline{\varepsilon}^r$.  But if $r=p+1$, it is well-known that there exists $V$ as
  above with $\overline{V}^{\mathrm{ss}} \cong \overline{\varepsilon}
  \oplus \overline{\varepsilon}$, a contradiction. 
  \end{ex}

\section{Extensions of rank one $\varphi$-modules}
\label{sec:extensions-rank-one}

Recall that we have assumed that $K/\Qp$ is unramified. In this
section we consider possible extensions of Kisin modules. Our analysis
in this section, combined with the results of Section~\ref{sec:shape},
is already sufficient to prove our main results for semisimple
representations; in Section~\ref{sec:extensions-rank-one-G-hat}, we
will extend this analysis to $(\varphi,\hat G)$-modules, in order to
be able to handle extension classes.  

Before we begin our analysis of
extensions of rank one $\varphi$-modules, we give some combinatorial
lemmas, which will be used to determine when an extension of
Kisin modules corresponds to a Galois representation with scalar
semisimplification.  (See Remark~\ref{remark: exceptional case only
  with equal characters} below, and see also the discussion in the opening pages of \cite[\S
3.2]{bdj}.)

\begin{lem}\label{lem:minus-p-to-p}
  Suppose that $r_0,\ldots,r_{f-1}$ are integers in the range $[-p,p]$
  that satisfy $\sum_{i=0}^{f-1} p^{f-1-i} r_i \equiv 0 \pmod{p^f-1}$.  Then
  either:
\begin{enumerate}
\item $(r_0,\ldots,r_{f-1}) = \pm (p-1,\ldots,p-1)$,

\item  the numbers $r_0,\ldots,r_{f-1}$, considered as a cyclic list,
  can be broken up into strings of the form $\pm
  (-1,p-1,\ldots,p-1,p)$ (where there may not be any occurrences of
  $p-1$) and strings of the form $(0,\ldots,0)$, or else

\item $p=2$ and $(r_0,\ldots,r_{f-1}) = \pm (2,\ldots,2)$.

\end{enumerate}
\end{lem}

\begin{proof}
  First suppose that none of the $r_i$ are equal to $\pm p$.  Then
  $|\sum_{i=0}^{f-1} p^{f-1-i} r_i| \le p^f-1$; so the only possibilities for
  that sum are $0$ and $\pm (p^f-1)$, and the latter can occur only for
  $(r_0,\ldots,r_{f-1}) = \pm(p-1,\ldots,p-1)$.  If instead $\sum_{i=0}^{f-1}
  p^{f-1-i} r_i = 0$ then considering divisibility by $p$ we have
  $r_{f-1} = 0$.  Dividing by $p$ and repeating, we see that $r_i = 0$
  for all $i$ in this case.

  Next suppose that $r_i = \pm p$ for some $i$.  We perform a
  ``carrying'' operation, by adding $\mp p$ to $r_i$ and adding $\pm
  1$ to $r_{i-1}$; this preserves the given congruence.  Now move
  left, and if the new $|r_{i-1}|$ is at least $p$ we perform the carrying
  operation there.  Continue this process with
  $r_{i-2},\ldots,r_0,r_{f-1},\ldots,r_{i+1}$ until we have returned to $r_i$
  again.  Note that if we have had to carry for both $r_j$ and
  $r_{j-1}$, then the two carries necessarily had the same sign; so a
  string of consecutive carries has the effect of subtracting $\pm
  (-1,p-1,\ldots,p-1,p)$ from a subsequence of the $r_j$'s, or else $\pm
  (p-1,\ldots,p-1)$ from the full list.

  At the end of this carrying process, we have a new sequence
  $r'_0,\ldots,r'_{f-1}$ satisfying the original congruence condition,
  but with all $r'_j \in [-(p-1),(p-1)]$.  Note also that $r_i \in \{0,\pm
  1\}$ at our starting point.
 If $p > 2$, then the first paragraph implies that $r'_i = 0$ for all
 $i$, and the last sentence of the second paragraph shows that
 $(r_0,\ldots,r_{f-1})$ has the desired shape.  If $p=2$ then it is
 also possible that $r'_i = 1$ for all $i$, or $r'_i = -1$ for all
 $i$.  But note that if we add some number of (non-overlapping)
 strings of the form $(1,-1,\ldots,-1,-2)$ to $(1,\ldots,1)$, the result
 actually has the form (2) again; so the only new possibility when
 $p=2$ is (3).
 \end{proof}

\begin{defn}
  \label{defn:cal-P}
  Let $\mathcal{P}$ be the set of $f$-tuples $(r_0,\ldots,r_{f-1})$
  with $r_i \in \{1,p-1,p\}$ for all $i$, and such that
\begin{itemize}
\item if $r_i = p$ then $r_{i+1} = 1$, and

\item if $r_i \in \{1,p-1\}$ then $r_{i+1}$ in $\{p-1,p\}$,
\end{itemize}
conventionally taking $r_f = r_0$.   (If $p  > 2$, these conditions
are equivalent to: $r_i = p$ if and only if $r_{i+1} = 1$.)
\end{defn}

The preceding definition is motivated by the following Lemma.

\begin{lem}
 \label{lem:combinatorial-J}
Let $r_0,\ldots,r_{f-1}$ be integers in the range $[1,p]$.  Let
  $J$ be a subset of  $\{0,\ldots,f-1\}$, and set $h_i = r_i$ if $i \in
  J$ and $h_i = 0$ if $i \not\in J$.  Then $$\sum_{i=0}^{f-1} p^{f-1-i} h_i \equiv \sum_{i=0}^{f-1}
 p^{f-1-i}(r_i-h_i) \pmod{p^f-1}$$ if and only if $(r_0,\ldots,r_{f-1})
 \in \mathcal{P}$ and $J$ satisfies:
\begin{itemize}
\item if $(r_{i-1},r_i) = (p,1)$ then $i+1 \in J$ if and only if $i \not\in J$
\item if $(r_{i-1},r_i) = (1,p-1)$ or $(p-1,p-1)$ then $i+1 \in J$ if and only if $i \in J$,
\end{itemize}
or else $p=2$,  $(r_0,\ldots,r_{f-1}) = (2,\ldots,2)$, and $J =
\varnothing$ or $\{0,1,\ldots,f-1\}$.
\end{lem}

\begin{proof}
  The congruence is equivalent to $\sum_{i=0}^{f-1} (-1)^{[i
    \in J]}
  p^{f-1-i} r_i \equiv 0 \pmod{p^f-1}$, where we write $[i \in J] = 1$
  if $i \in J$ and $[i \in J] = 0$ otherwise.
Since none of the $r_i$ are zero, by Lemma \ref{lem:minus-p-to-p} we
see that if $p > 2$ the sequence $((-1)^{[i \in J]}
r_i)_{0 \le i \le f-1}$ must either be $\pm (p-1,\ldots,p-1)$ or else
break up into subsequences of the form $\pm(-1,p-1,\ldots,p-1,p)$;
when $p=2$ we have the additional possibilities $\pm (2,\ldots,2)$.
This is equivalent to the description in the statement of the lemma.
\end{proof}
The following result gives a structure theorem for extensions of Kisin
modules; we will build on it in the following section to prove
Proposition~\ref{prop:maximal-model}, which is the main result we will
need on extensions of $(\varphi,\hat G)$-modules.
\begin{prop}
  \label{prop:extensions-of-phi-modules}
 Let $r_0,\ldots,r_{f-1}$ be integers in the range $[1,p]$.  Let
  $J$ be a subset of  $\{0,\ldots,f-1\}$, and set $h_i = r_i$ if $i \in
  J$ and $h_i = 0$ if $i \not\in J$.   Fix $a,b \in k_E^{\times}$.
  Let $\barM$ be  an
  extension of $\barM(h_0,\ldots,h_{f-1};a)$ by
  $\barM(r_0-h_0,\ldots,r_{f-1}-h_{f-1};b)$; then we can choose bases $e_i,f_i$
  of the $\barM_i$ so that $\varphi$  has the form
  \begin{eqnarray*}
  \varphi(e_{i-1}) & =& (b)_i u^{r_i-h_i} e_{i}\\
    \varphi(f_{i-1})& =& (a)_i u^{h_i} f_{i} + x_i e_{i}
 \end{eqnarray*}
with $x_i \in k_E \llb u \rrb$ a polynomial with $\deg(x_i) < h_i$,
except in the following cases:
\begin{itemize}
\item $(r_0,\ldots,r_{f-1}) \in \mathcal{P}$, $J = \{ i : r_{i-1}
 \neq p \}$, and $a=b$, or

\item $p=2$, $(r_0,\ldots,r_{f-1}) = (2,\ldots,2)$, $J =
  \{0,\ldots,f-1\}$, and
  $a=b$.
\end{itemize}
In that case fix $i_0 \in J$; then $x_i$ may be taken to be a
polynomial of degree $\deg(x_i) < h_i$ for all $i$ except $i=i_0$,
where $x_{i_0}$ is the sum of a polynomial of degree less than
$h_{i_0}$ and a (possibly trivial) term of degree $p$ (for the first
exceptional case) or degree~$4$ (for the second exceptional case).
\end{prop}

\begin{proof}
Let $\barM$ be an extension of  $\barM(h_0,\ldots,h_{f-1};a)$ by
  $\barM(r_0-h_0,\ldots,r_{f-1}-h_{f-1};b)$; then we can choose bases $e_i,f_i$
  of the $\barM_i$ so that $\varphi$  has the form
  \begin{eqnarray*}
  \varphi(e_{i-1}) & =& (b)_i u^{r_i-h_i} e_{i}\\
    \varphi(f_{i-1})& =& (a)_i u^{h_i} f_{i} + x_i e_{i} .
 \end{eqnarray*}
We wish to determine to what extent the $x_i$'s can be simultaneously simplified via
a change of basis of the form $f'_i = f_i + \alpha_{i} e_i$ for some
elements $\alpha_{i} \in k_E \llb u \rrb$.   If $\alpha = \alpha(u) \in k_E \llb u
\rrb$ let $\varphi(\alpha) = \alpha(u^p)$.   Observing that
$$ \varphi(f_{i-1} + \alpha_{i-1} e_{i-1}) = (a)_i u^{h_i} (f_{i} + \alpha_{i}
e_{i}) + (x_i + (b)_i u^{r_i-h_i} \varphi(\alpha_{i-1})  - (a)_i u^{h_i}
\alpha_{i} ) e_{i},$$
we see that such a change of basis replaces each $x_i$ with
$$ x'_i = x_i +  (b)_i u^{r_i-h_i} \varphi(\alpha_{i-1})  - (a)_i u^{h_i}
\alpha_{i} .$$
Observe that we may make $x'_i = 0$ if $i \not\in J$ (at least for any individual
such $i$) by choosing
\begin{equation}\label{eq:splitzero}
\alpha_{i} = (a)_i^{-1} ( x_i + (b)_i u^{r_i}
\varphi(\alpha_{i-1})).
\end{equation}   If $J \neq \varnothing$ then we can take
$x'_i = 0$ simultaneously for all $i \not\in J$ by choosing $\alpha_{i}$
arbitrarily for each $i \in J$ and determining $\alpha_{i}$ recursively
by the formula~\eqref{eq:splitzero} for $i \not\in J$.    If $J = \varnothing$ then
the preceding sentence shows that we can at least have $x'_i = 0$ for
$i \neq f-1$ by choosing $\alpha_{f-1}$ arbitrarily and choosing $\alpha_{i}$
recursively for $i = 0,\ldots,f-2$ using~\eqref{eq:splitzero}.
Suppose now that $x_0 = \cdots = x_{f-2} = 0$.   Taking $\alpha_{f-1}$
arbitrary and choosing $\alpha_{i} = (b/a)_i u^{r_i}
\varphi(\alpha_{i-1})$ for $i = 0,\ldots,f-2$, one computes that
$$x'_{f-1} = x_{f-1} + (b/a) u^{r_{f-1} + p r_{f-2} + \cdots + p^{f-1} r_0}
\varphi^f(\alpha_{f-1}) - \alpha_{f-1}.$$
It is possible to
choose $\alpha_{f-1}$ in the above equation so that $x'_{f-1} = 0$:
indeed, if we set the right-hand side of the above expression
equal to zero, the resulting equation 
$$ \alpha_{f-1} = x_{f-1} + (b/a) u^{r_{f-1} + p r_{f-2} + \cdots + p^{f-1} r_0}
\varphi^f(\alpha_{f-1}) $$
can be regarded as a 
system  of equations for the coefficients of $\alpha_{f-1}$. Since $r_{f-1} + p r_{f-2} + \cdots + p^{f-1} r_0 > 0$, 
the coefficient of $u^i$ on the left-hand side depends only on
lower-degree coefficients of $\alpha_{f-1}$ on the right-hand side, and so this system
can be solved recursively.  With such a choice of $\alpha_{f-1}$ we have $x'_i = 0$ for all $i$.

The preceding paragraph shows that in all cases, we can assume (possibly after a
change of variables) that $x_i = 0$ if $i \not\in J$.  At this point we are
done with the case $J = \varnothing$, so we assume from now on that
$J \neq \varnothing$.  For the
remainder of the argument, whenever we consider a simultaneous change
of basis of the form $f'_i = f_i + \alpha_{i} e_i$, we will make some
choice of $\alpha_{i}$'s for $i \in J$ and then (without further comment) define
$\alpha_{i}$ for $i \not\in J$ by the recursive formula
$\alpha_{i} = (b/a)_i u^{r_i}
\varphi(\alpha_{i-1})$; then the
resulting change of variables preserves the property that $x_i = 0$ if $i \not\in J$.

If $i \in J$, let $\delta_i$ be the least positive integer such that
$i + \delta_i \in J$ (taken modulo~$f$, as usual); then a simultaneous change
of basis of the form $f'_i = f_i + \alpha_{i} e_i$ has the effect
\begin{equation}\label{eq:bigchange}
x'_{i + \delta_i} = x_{i + \delta_i} +
\frac{(b)_{\{i+1,\ldots,i+\delta_i\}}}{(a)_{\{i+1,\ldots,i+\delta_i-1\}}}
    u^{\sum_{j=1}^{\delta_i-1} r_{i+j} p^{\delta_i-j}}
    \varphi^{\delta_i} (\alpha_i) - (a)_{i+\delta_i}
    u^{r_{i+\delta_i}} \alpha_{i+\delta_i}.\end{equation}
If $i \in J$ and $d \ge r_i$, we shall say that the $u^d$-term in
$x_i$ \emph{affects} the $u^{d'}$-term in $x_{i+\delta_i}$ if the
change of variables $f'_{i} = f_{i} + c u^{d-r_i} e_{i}$ (for
just the single $i \in J$)
alters the term of degree $u^{d'}$ in $x'_{i+\delta_i}$, or in other
words if \begin{equation}\label{eq:dprime}
d' = p^{\delta_i} (d-r_i) + \sum_{j=1}^{\delta_i-1} r_{i+j} p^{\delta_i-j}.\end{equation}
In that case,
for brevity
we will write that $(i,d)$ affects $(i+\delta_i,d')$.

Observe that each pair $(i,d)$ affects exactly one pair $(i',d')$ (though possibly with
$d' < r_{i'}$) and similarly is affected by at most one pair (though
often by none).  Observe also, e.g. from~\eqref{eq:dprime}, that if
$(i,d)$ affects $(i+\delta_i,d')$ then $(i,d+1)$ affects
$(i+\delta_i,d' + p^{\delta_i})$; one deduces that there are at most finitely many
pairs $(i,d)$ that affect a pair $(i',d')$ with $d' \le d$.    It
follows that the set of all pairs $(i,d)$ with $i \in J$
and $d \ge r_i$ is partitioned into:
\begin{itemize}
\item a finite number of \emph{loops} $(i_0,d_0),\ldots,(i_{|J|-1},d_{|J|-1})$ in
  which $(i_j,d_j)$ affects $(i_{j+1},d_{j+1})$ (and
  $(i_{|J|-1},d_{|J|-1})$ affects $(i_0,d_0)$),

\item a finite number of \emph{stubs} $(i_0,d_0),\ldots,(i_m,d_m)$ in which
  $(i_0,d_0)$ is not affected by any $(i,d)$, while $(i_m,d_m)$
  affects some $(i',d')$ with $d' < r_{i'}$,

\item a collection of \emph{paths} $(i_0,d_0),\ldots,(i_j,d_j),\ldots$ in
 which $(i_0,d_0)$ is not affected by any $(i,d)$, and in which  $(i_j,d_j)$ affects $(i_{j+1},d_{j+1})$.
\end{itemize}
It is straightforward to see that by making a suitable choice of
$u^{d_0 - r_{i_0}}$-coefficient in $\alpha_{i_0}$ (in the second and third cases) or an arbitrary choice of
$u^{d_0 - r_{i_0}}$-coefficient in $\alpha_{i_0}$ (in the first case),
recursively making suitable choices for $u^{d_j -
  r_{i_j}}$-coefficient in $\alpha_{i_j}$ for $j > 0$ (stopping at $j
= |J|-1$ in the first case and at $j = m$ in the second case), and
doing this simultaneously for all loops, stubs, and paths, the
resulting change of basis ensures that $x'_i$ has degree less than
$r_i$ for all $i \in J$, with the exception that for each loop, $x'_{i_0}$ may
also have a term of degree $d_0$.

 Assume that we have made such a
change of basis, so that now $x_i$ is a polynomial of degree less than
$h_i$ for all $i$, except possibly for a term of degree $d_0$ in $x_{i_0}$
for each loop as above.

It remains to analyze any possible loops more closely.  It follows
immediately from~\eqref{eq:dprime} that in a loop
$(i_0,d_0),\ldots,(i_{|J|-1},d_{|J|-1})$ we have $p \mid d_j$ for all
$j$.  But note that if $d \ge 2p$ and $(i,d)$ affects
$(i+\delta_i,d')$, then since $d' \ge p^{\delta_i} (d -
p)$ we have $d' > d$ unless
$p=2$, $d=4$ and $\delta_i=1$.   It follows that there is at most one
loop, and in any loop we either have $d_i = p$ for all $i$, or else $p=2$ and
$(\delta_i,r_i,d_i) = (1,2,4)$ for all $i$.

The latter is the second exceptional case described in the statement
(except for the condition that $a=b$); now consider the former.  If $\delta_i
> 2$ then $\sum_{j=1}^{\delta_i-1} r_{i+j} p^{\delta_i - j} > p$ since
$r_{i+1} > 0$, so any loop with $d_i = p$ for all $i$ requires
$\delta_i \le 2$ for all $i \in J$.  The possibilities, then, are
either $\delta_i = 1$ and $r_i = p-1$ or else $\delta_i = 2$ and
$$ p = p^2(p-r_i) + p r_{i+1},$$
i.e., $r_i = p$ and $r_{i+1} = 1$.   Conversely, if $\delta_i \in
\{1,2\}$ for all $i \in J$, with $r_i=p-1$ whenever $\delta_i = 1$
and $(r_i,r_{i+1}) = (p,1)$ whenever $\delta_i= 2$, we indeed have a
loop $\{ (i,p) : i \in J\}$.  Observe that this is precisely the
first exceptional case described in the statement of the Proposition, again modulo
the condition that $a=b$.

In fact one checks without difficulty for the first exceptional case
in the statement
(with $d_i = p$ for all $i$)  that making the change of
variables $\alpha_{i_0} = cu^{p-r_{i_0}}$ (and choosing $\alpha_{i_j}$
accordingly for $1 \le j < |J|$ to ensure that $x_{i_j}$ does not acquire a
nonzero term of degree $p$), we find that $x'_{i_0} = x_{i_0} +
(a)_{i_0} (b/a - 1) c u^p$.  Thus if $a\neq b$ we can always choose
$c$ to kill the term of degree $p$ in $x_{i_0}$, and the exceptional
case only occurs when $a=b$.  The argument in the second exceptional
case is analogous.
\end{proof}

Note that in Proposition~\ref{prop:extensions-of-phi-modules} we made
no assumption about $\barM$ having a lift to some $\M$ of characteristic zero (let
alone having a lift to some $\M$ coming from a crystalline
representation).   We now examine what happens when we make such an
assumption.  For the remainder of this section we re-assume the
notation of Section~\ref{sec:coefficients},  so that $p >2$, $T$ is a
$G_K$-stable $\O_E$-lattice in a crystalline representation $V$ of
$E$-dimension $d$ with
Hodge--Tate weights in $[0,p]$, and $\M$ is the associated Kisin
module.  Write $r_{1,s},\ldots,r_{d,s}$ for the $\hodgetateembedding_s$-labeled
Hodge--Tate weights of $V$, and let $\barM := \M \otimes_{\O_E} k_E$, with $k_E$ the residue
field of $E$.

\begin{prop}
  \label{prop:rank-one-subs}
With notation as above (in particular $p > 2$), suppose that $\barN \subset \barM$ is a sub-$\varphi$-module such that
$\barM/\barN$ is free of rank $d-1$ as a $W(k)\llb u \rrb \otimes_{\Zp} k_E$-module.  Then
$\barN \simeq \barM(r_0,\ldots,r_{f-1};a)$ with $r_s \in
\{r_{1,s},\ldots,r_{d,s}\}$ for all $s$, and for some $a \in k_E^{\times}$.
\end{prop}

\begin{proof}
Choose a basis $\{e_{i,s}\}$ for $\M$ as in Theorem~\ref{shape-redux-coeffs}.
 Since we will work in $\barM$ for the remainder of the proof, no
 confusion will arise if we write $\{e_{i,s}\}$ also for the image of
   that basis in $\barM$.

 A generator $f_{s-1}$ of $\barN_{s-1}$ has the form $(e_{1,s-1},\ldots,e_{d,s-1})
 \cdot (v_{1,s},\ldots,v_{d,s})^T$ for some $v_{1,s},\ldots,v_{d,s}
 \in k_E\llb u \rrb$, by hypothesis at least one of which is a unit.  We know from
 Theorem~\ref{shape-redux-coeffs} that
$$ \varphi(f_{s-1}) = (e_{1,{s}},\ldots,e_{d,{s}}) \overline{X}_{s}
\overline{\Lambda}_{s} \cdot
(\varphi(v_{1,s}),\ldots,\varphi(v_{d,s}))^T$$
where $\overline{X}_{s}$ is the reduction mod $\m_E$ of $X_{s}$
and $\overline{\Lambda}_{s} = [u^{r_{1,s}},\ldots,u^{r_{d,s}}]$
Now, observe that each entry of
$(\varphi(v_{1,s}),\ldots,\varphi(v_{d,s}))^T$ is either a unit or
divisible by $u^p$, and at least one is a unit.   Since we have
$r_{i,s} \le p$ for all $i$, it follows that the largest power of
$u$ dividing the column vector $\overline{\Lambda}_{s} \cdot
(\varphi(v_{1,s}),\ldots,\varphi(v_{d,s}))^T$ is $u^{r_{i,s}}$ for
some $i$.  Noting that $\overline{X}_{s}$ is invertible, the same is
true of $\varphi(f_{s-1})$, and the proposition follows.
\end{proof}

\begin{thm}
  \label{thm:extension-crystalline}
 Suppose that $K/\Qp$ is unramified and $p > 2$.  Let $T$ be a
  $G_K$-stable $\O_E$-lattice in a crystalline
  representation $V$ of $E$-dimension $2$ whose $\hodgetateembedding_s$-labeled
  Hodge--Tate weights are $\{0,r_s\}$ with $r_s \in [1,p]$ for all
  $s$.  Let $\M$ be the Kisin module associated to $T$, and  let
  $\barM := \M \otimes_{\O_E} k_E$.

  Assume that the $k_E[G_K]$-module $\overline{T}:=T/\m_ET$ is
  reducible. Then $\barM$ is an extension of two $\varphi$-modules of
  rank one, and there exist $a,b \in k_E^{\times}$ and a subset $J
  \subset \{0,\ldots,f-1\}$ so that $\barM$ is as follows.

 Set $h_i = r_i$ if $i \in
  J$ and $h_i = 0$ if $i \not\in J$.  Then $\barM$ is  an
  extension of $\barM(h_0,\ldots,h_{f-1};a)$ by
  $\barM(r_0-h_0,\ldots,r_{f-1}-h_{f-1};b)$, and we can choose bases $e_i,f_i$
  of the $\barM_i$ so that $\varphi$  has the form
  \begin{eqnarray*}
  \varphi(e_{i-1}) & =& (b)_i u^{r_i-h_i} e_{i}\\
    \varphi(f_{i-1})& =& (a)_i u^{h_i} f_{i} + x_i e_{i}
 \end{eqnarray*}
with $x_i = 0$ if $i \not\in J$ and $x_i \in k_E$ constant if $i \in
J$,
except in the following case:
\begin{itemize}
\item $(r_0,\ldots,r_{f-1}) \in \mathcal{P}$,
\item $J = \{ i : r_i = p-1,p\}$, and
\item $a=b$.
\end{itemize}
In that case fix $i_0 \in J$; then $x_i$ may be taken to be $0$ for
all $i \not\in J$, to be a
constant for all $i$ except $i=i_0$, and to be the sum of a constant
and a term of degree $p$ if $i=i_0$.

Finally, $\overline{T}|_{I_K}\simeq \begin{pmatrix} \prod_{ i\in
    J}\omega_{i}^{r_i}&*\\ 0& \prod_{i\notin
    J}\omega_{i}^{r_i} \end{pmatrix} $.
\end{thm}

\begin{proof}
  It follows from (for example) Lemma \ref{lem: reducible rep yield
    reducible Kisin module} that $\barM$ is an extension of two rank
  one $\varphi$-modules. Then Proposition~\ref{prop:rank-one-subs}
  guarantees that if $\barM$ is an extension of $\barMp$ by $\barMpp$,
  then $\barMpp$ has the form $\barM(r'_0,\ldots,r'_{f-1}; b)$ with
  $r'_i \in \{0,r_i\}$ for all $i$.  Taking $i \in J$ if $r'_i = 0$ and
  $i \not\in J$ if $r'_i = r_i$ puts $\barMpp$ into the correct form;
  considering the determinant of $\varphi$ in
  Theorem~\ref{shape-redux-coeffs} one finds that $\barMp$ then also
  has the correct form.

Now $\barM$ can be taken to have the form given by
Proposition~\ref{prop:extensions-of-phi-modules}, and it remains to
show that each $x_i$ with $i \in J$ cannot have any nonzero terms of
degree between~$1$ and $r_{i}-1$.  But
Theorem~\ref{shape-redux-coeffs} implies
that the image $\varphi(\barM_{i-1}) \subset \barM_i$ is spanned over
$k_E\llb u^p \rrb$ by an
element divisible exactly by $u^0$ and an element divisible exactly by
$u^{r_i}$.  On the other hand, if $x_i$ were to have a term of degree
between~$1$ and~$r_i-1$ then neither $(b)_i e_{i} + \varphi(c) ( (a)_i
u^{r_i} f_{i} + x_i e_{i})$ nor $(a)_i u^{r_i} f_{i} + x_i e_{i} + \varphi(c)
(b)_i e_{i}$ would be divisible exactly by $u^{r_i}$ for any $c \in k_E \llb
u\rrb$.   This is a contradiction.

Finally, that
$\overline{T}|_{I_K}$ is as claimed follows from parts (1) and (2)  of
Proposition~\ref{prop:calculation-on-inertia}, together with the fact
that two mod $p$ characters of $G_K$ that are equal on $G_{\infty}$
must be equal. \end{proof}

\begin{remark}\label{remark: exceptional case only with equal characters}
  It follows easily from Proposition \ref{prop:calculation-on-inertia}
  and Lemma \ref{lem:combinatorial-J} that the exceptional case of
  Theorem \ref{thm:extension-crystalline} in which we allow a term of
  degree $p$ can only occur if $\overline{T}$ is an extension of a
  character by itself.
\end{remark}

\begin{cor}
  \label{cor:form of the characters in reducible crystalline
    reduction, nothing about extensions}
  Suppose that $K/\Qp$ is unramified and $p > 2$. Let
  $\rhobar\col G_K\to\GL_2(\Fpbar)$ be the reduction mod $p$ of a
  $G_K$-stable $\Zpbar$-lattice in a crystalline representation
  $\Qpbar$-representation of dimension $2$ whose
  $\hodgetateembedding$-labeled Hodge--Tate weights are
  $\{0,r_\hodgetateembedding\}$ with $r_\hodgetateembedding \in [1,p]$
  for all $\hodgetateembedding$.

Assume that $\rhobar$ is reducible. Let $S=\Hom(k,\Fpbar)$, and
identify the set $S$ with $\Hom_{\Qp}(K,\Qpbar)$. Then there is
a subset $J\subset S$ such that \[\rhobar|_{I_K}\simeq
  \begin{pmatrix}
 \prod_{ \sigma\in
      J}\omega_{\sigma}^{r_\sigma}&*\\ 0&  \prod_{\sigma\notin
      J}\omega_\sigma^{r_\sigma}  \end{pmatrix}.\]
\end{cor}
\begin{proof}
This follows immediately from Theorem \ref{thm:extension-crystalline},
since $\rho$ is necessarily defined over some finite extension $E/\Qp$.
\end{proof}
Note that Corollary \ref{cor:form of the characters in reducible
  crystalline reduction, nothing about extensions} does not suffice to
prove Theorem \ref{thm: crystalline lift implies explicit crystalline
  lift} in the reducible case, because it says nothing about the
extension classes. In the following sections we will improve on this
result by making a more detailed study of the full $(\varphi,\hat
G)$-modules, rather than just the underlying Kisin modules. However,
Corollary \ref{cor:form of the characters in reducible crystalline
  reduction, nothing about extensions} can be combined with a
combinatorial argument to deduce Theorem \ref{thm: crystalline lift
  implies explicit crystalline lift} in the irreducible case (see
Theorem \ref{thm: reduction mod p for GL2 in the irreducible case}
below).

\section{Extensions of rank one $(\varphi,\hat G)$-modules}
\label{sec:extensions-rank-one-G-hat}

\subsection{From Kisin modules to $(\varphi,\Ghat)$-modules}
\label{sec:from-kisin-modules}
We will now
study how (and whether) the rank two $\varphi$-modules of Section~\ref{sec:extensions-rank-one}
can be extended to $(\varphi,\hat G)$-modules.   Return to the situation of the previous section: that is, suppose
$K=K_0$ and $p > 2$, and let $T$ be a $G_K$-stable $\O_E$-lattice as
in Theorem~\ref{thm:extension-crystalline}.  Let  $\hat {\overline
  \M}= (\bar \M , \varphi, \hat G)$ be the $(\varphi, \hat G)$-module
associated to $\overline T = T/\m_E T$ via Theorem \ref{thm: old facts
  on (phi hatG)}(4). We \emph{further assume}
that $\overline T$ is reducible and sits in an exact sequence $$0
\to \overline \psi_1 \to \overline T \to \overline \psi_2 \to 0.$$ By
Lemma~\ref{lem: reducible rep yield reducible Kisin module}, the
$(\varphi,\hat G)$-module $\hat \barM$ sits in an exact sequence of
$(\varphi, \hat G)$-modules, whose ambient Kisin module is an exact
sequence $0 \to \barM_2  \to  \barM \to \barM_1 \to 0$.  In the
notation of
Theorem \ref{thm:extension-crystalline}, it follows from that result that $\barM_1 =\barM(h_0 ,
\dots, h_{f-1}; a) $ and $\barM_2 = \barM(r_0 -h_0 , \dots, r_{f-1}-
h_{f-1}; b ) $ for some choice of $a$, $b$, and $J$.

\begin{lemma}\label{lem: hat G is determined by phi}
 Except possibly for the case that $r_i = h_i = p$ for all $i
  = 0 , \dots, f-1$, there is at most one way to extend the
 exact
sequence $$0 \lto \barM_2  \lto  \barM \lto \barM_1 \lto 0$$
to an exact sequence of $(\varphi,\hat G)$-modules with natural
$k_E$-action satisfying the conclusion of Corollary~\ref{cor:valuation of torsion crystalline}.
In particular the $\hat
  G$-action on $\hat {\overline \M}  $ is uniquely determined by
  $\overline \M$, except possibly  for the case that $r_i = h_i = p$ for all $i
  = 0 , \dots, f-1$.
\end{lemma}
\begin{proof}
Since
  $\hat\barM$ is assumed to come from a crystalline
  representation, the conclusion of
  Corollary~\ref{cor:valuation of torsion crystalline} holds for $\hat\barM$.
Since by definition $\overline \M$ is contained in the
  $H_K$-invariants of $\hat{\overline{\M}}$, it suffices to show that the
  $\tau$-action on $ \hR \otimes_{\varphi, \fS} \barM$ is uniquely
  determined by  the condition of Corollary~\ref{cor:valuation of torsion crystalline}.  Since  $\hat{\overline{\M}}$ is reducible,
we can write $$\tau (e_{i-1}, f_{i-1}) = (e_{ i-1 },
  f_{i-1}) \begin{pmatrix} \alpha_{i}  & \beta_{i} \\ 0 & \gamma_{i} \end{pmatrix}$$ with
  $\alpha_i, \beta _i,  \gamma_i \in (\hR/ p \hR) \otimes_{\Fp} k_E
  \subset R \otimes_{\Fp} k_E $.   If $\zeta \in R \otimes_{\Fp} k_E$
  is written $\zeta = \sum_{i=1}^{n} y_i \otimes z_i$ with
  $z_1,\ldots,z_n \in k_E$ linearly independent over $\Fp$, write $v_R(\zeta) =
  \min_i \{
  v_R(y_i)\}.$  One checks without difficulty that this is independent
  of the sum representing $\zeta$, so is well-defined and satisfies the
  usual inequality $v_R(\zeta_1 + \zeta_2) \ge
  \min(v_R(\zeta_1),v_R(\zeta_2))$.   The condition of
  Corollary~\ref{cor:valuation of torsion crystalline} implies that
  $v_R(\alpha_i-1), v_R(\gamma_i-1),v_R(\beta_i) \ge \frac{p^2}{p-1}$
  for all $i$.

Recalling that $\barM$
  is regarded as a $\varphi(\fS)$-submodule of $R \otimes_{\varphi,
    \fS}\barM$, by Theorem \ref{thm:extension-crystalline} we may write $\varphi (e_{ i-1}, f_{ i-1 }) = (e_{ i }, f_{i}) \varphi (A_i)$ with $A_i = \begin{pmatrix} (b)_i u ^{r_i -h_i} &  x_i \\ 0 & (a)_i u ^{h_i} \end{pmatrix} $. Since $\varphi$ and $\tau$ commute, we have
$$ \varphi(A_i) \begin{pmatrix} \varphi(\alpha_{i})  & \varphi(\beta_{i})  \\ 0 & \varphi(\gamma_{i}) \end{pmatrix} =  \begin{pmatrix} \alpha_{i+1} & \beta_{i+1} \\ 0 & \gamma_{i+1} \end{pmatrix} \tau (\varphi(A_i)) . $$
Recall that $\tau(u) = \ue u$, and once again let $\eta \in R$ be the element
defined in Lemma~\ref{valuation}(2), so that $\varphi^f(\eta) = \ue \eta$.  We obtain the following formulas:
\begin{equation} \label{eq: unique tau 1}
   u ^{p(r_i -h_i) } \varphi (\alpha_{i})= \alpha_{i+1} (\ue u)
  ^{p(r_i -h_i)},  \ \   u ^{ph_i } \varphi (\gamma_{i})= (\ue u)
  ^{ph_i} \gamma_{i+1}
\end{equation}
 and
 \begin{equation}\label{eq: unique tau 2}  (b)_i u ^{p(r_i -h_i) }
   \varphi (\beta_{i}) + \varphi(x_i)  \varphi(\gamma_{i}) =
   \alpha_{i+1}\tau( \varphi (x _i)) + (a)_i (\ue u )^{ph_i} \beta_{i+1}
 \end{equation}
where for succinctness we have written $(a)_i$, $(b)_i$ in lieu of
$1\otimes (a)_i$, $1\otimes (b)_i$ in the preceding equation.

From~\eqref{eq: unique tau 1} we see that $\varphi^f(\alpha_i) =
\alpha_i \ue^{\sum_{j=0}^{f-1} p^{f-j} (r_{i+j}-h_{i+j})}$, and now
Lemma~\ref{valuation}(2) together with the requirement that
$v_R(\alpha_i-1) > 0$
imply that
$$ \alpha_i = \eta^{\sum_{j=0}^{f-1} p^{f-j} (r_{i+j}-h_{i+j})} \otimes 1$$ for
all $i$.  Similarly we must have $\gamma_i = \eta^{\sum_{j=0}^{f-1}
  p^{f-j} h_{i+j}} \otimes 1$ for all $i$.  So at least the $\alpha_i,\gamma_i$
are uniquely determined.

 Now suppose that there exists some other extension of $\barM$ to
a $(\varphi, \hat
 G)$-module $\hat \barMp$.  
Then the $\tau$-action on $\hat \barMp$ is given by some $\alpha'_i$,
 $\beta'_{i}$ and $\gamma'_i$ that also satisfy~\eqref{eq: unique tau
   1} and~\eqref{eq: unique tau 2}, and indeed we have already seen that
 $\alpha'_i = \alpha_i$ and $\gamma'_i = \gamma_i$.

Let $\tilde \beta_i = \beta_i - \beta'_i$.  Taking the difference
between~\eqref{eq: unique tau 2} for $\hat \barM$ and $\hat \barMp$
gives
$$ (b)_i u^{p(r_i-h_i)} \varphi(\tilde \beta_{i}) = (a)_i (\ue u)^{p
  h_i} \tilde \beta_{i+1},$$
which implies that
$$b u^{\sum_{j = 0}^{f-1} u^{p^{f-j} (r_{i+j} - h_{i+j})}} \varphi^f (\tilde
\beta_i) = a (\ue u)^{\sum_{j = 0}^{f-1} u^{p^{f-j} h_{i+j}}} \tilde
\beta_i.$$
Considering the valuations of both sides, we see that if $\betat_i\ne
0$, then \begin{equation}\label{eq:beta-valuation}  v_R(\tilde
\beta_i) = \frac{1}{p^f-1} \sum_{j=0}^{f-1} p^{f-j} (2h_{i+j} - r_{i+j}). \end{equation}
But since $2h_i - r_i  \in \{\pm r_i\}$ is at most $p$ with equality
if and only if $h_i = r_i = p$,  the right-hand side
of~\eqref{eq:beta-valuation} is at most $\frac{p^2}{p-1}$ with
equality if and only if $h_i = r_i = p$ for all $i$.  In particular,
since $v_R(\beta_i),v_R(\beta'_i) \ge \frac{p^2}{p-1}$, either
$\beta_i = \beta'_i$ for all $i$, or else $h_i=r_i = p$ for all $i$,
as required.
\end{proof}
\begin{remark}
  In the case that each $r_i$ is at most $p-2$, the results of
  \cite{MR1695849} show that there is a canonical $\Ghat$-action on
  $\hat {\overline \M} $.
\end{remark}

\subsection{Comparison of extensions of rank one $(\varphi,\hat G)$-modules}
\label{sec:comp-extens-rank}
We are now in a position to prove our main result on extensions of
rank one $(\varphi,\hat G)$-modules, namely
Proposition~\ref{prop:maximal-model} below, which we will use in the
following section to prove our main local result
(Theorem~\ref{thm: elimination in the reducible case}), showing that
(under appropriate hypotheses) the existence of a crystalline lift
implies the existence of a reducible crystalline lift with the same
Hodge--Tate weights.
\begin{defn}
  \label{defn:model-of-type-J}
Write  $\vec{r} := (r_0,\ldots,r_{f-1})$ with $r_i \in [1,p]$ for all $i$.
  We say that a Kisin module $\barM$  is an extension of type
  $(\vec{r},a,b,J)$ if it has the same shape as the Kisin modules
  described by
  Theorem~\ref{thm:extension-crystalline}; that is, $\barM$
sits in a short exact sequence
$$ 0 \lto \barM(r_0-h_0,\ldots,r_{f-1}-h_{f-1};b) \lto \barM \lto
\barM(h_0,\ldots,h_{f-1};a) \lto 0 $$
in which the extension parameters $x_i$ satisfy $x_i = 0$ if $i
\not\in J$ and $x_i \in k_E$ if $i \in J$ (except that $x_{i_0}$
is allowed to have a term of degree $p$ for one $i_0 \in J$ when
$\vec{r} \in \mathcal{P}$, $J = \{i :  r_i = p-1,p\}$, and $a=b$).
We say that a $(\varphi,\hat G)$-module $\hat \barM$ with natural $k_E$-action is
  of type  $(\vec{r},a,b,J)$ if it is an extension
$$ 0 \lto \hat\barMpp \lto \hat\barM \lto \hat\barMp \lto 0$$ such
that the ambient short exact sequence of Kisin modules is an extension
of type $(\vec{r},a,b,J)$, and if for all $x \in \barM$ there exist
$\alpha \in R$ and $y \in R \otimes_{\varphi,\fS} \barM$ such that $\tau(x)-x = \alpha y$ and
$v_R(\alpha) \ge \frac{p^2}{p-1}$.
\end{defn}

\begin{rem}
 \label{rem:type-remarks}  Thanks to our work in previous sections, we
 have the following.

(1)  Unless $\vec{r} = (p,p,\ldots,p)$ and $J =
\{0,\ldots,f-1\}$, we know from Lemma~\ref{lem: hat G is determined by
  phi} that each extension $\barM$ of type $(\vec{r},a,b,J)$ extends
to an extension $\hat\barM$ of type $(\vec{r},a,b,J)$ in at most one
way.

(2) Suppose
$K=K_0$, $p > 2$, and $T$ is a $G_K$-stable $\O_E$-lattice as
in Theorem~\ref{thm:extension-crystalline}.   If $\overline{T} =
T/\m_E T$ is reducible, then the $(\varphi, \hat G)$-module
$\hat {\overline
  \M}$
associated to $\overline T = T/\m_E T$ via Theorem \ref{thm: old facts
  on (phi hatG)}(4) is of type $(\vec{r},a,b,J)$ for some $a$, $b$,
and $J$.
\end{rem}

Suppose $p > 2$, fix integers $r_0,\ldots,r_{f-1} \in [1,p]$, and let $r =
\sum_{i=0}^{f-1} p^{f-1-i} r_i \pmod{p^f-1}$, so that $r$ is an
element of $\Z/(p^f-1)\Z$.   If $J \subset \{0,\ldots,f-1\}$, let
the integers $h_i$ be defined as in Theorem~\ref{thm:extension-crystalline}, and
write $h(J) := \sum_{i=0}^{f-1} p^{f-1-i} h_i  \pmod{p^f-1}$.

Fix $h \in \Z/(p^f-1)\Z$ and suppose
that there exists a subset $J \subset \{0,\ldots,f-1\}$ such that $h = h(J)$.
For
fixed $h$ there may be several such choices of $J$, as we now
describe.  Let $i_1,\ldots,i_{\delta}$ be the distinct integers in the
range $\{0,\ldots,f-1\}$ such that:
\begin{itemize}
\item $(r_{i_j},\ldots,r_{i_j + s_j}) = (1,p-1,\ldots,p-1,p)$ for some
  $s_j > 0$, and
\item either $i_j \in J$ and $i_j + 1,\ldots,i_j + s_j \not\in J$, or
  vice versa.
\end{itemize}
According to Lemma~\ref{lem:minus-p-to-p}, all other sets $J'$ such
that $h = h(J')$ are obtained from~$J$ by choosing some integers $j \in [1,\delta]$,
and removing $i_j$ from $J$ and adding $i_j+1,\ldots,i_j+s_j$ to $J$
(if $i_j \in J$ to begin with), or vice versa (if $i_j \not\in J$ to begin with); or else $r_i =
p-1$ for all $i$ and $J$, $J' = \varnothing$ or $\{0,\ldots,f-1\}$.

In particular, for each $h$ such that $h = h(J)$ for at least one $J$,
we can define $J_{\max}$ to be the unique subset of $\{0,\ldots,f-1\}$
such that
\begin{itemize}
\item $h = h(J_{\max})$, and
\item $i_j \not\in J$ and $i_j+1,\ldots,i_j + s_j \in J$ for all $1\le
  j \le \delta$.
\end{itemize}
When  $r_i =
p-1$ for all $i$ and $J = \varnothing$ or $\{0,\ldots,f-1\}$, we set
$J_{\max} = \{0,\ldots,f-1\}$.  (Strictly speaking we should write
$J_{\max}(h)$ instead of $J_{\max}$, but $h$ will always be fixed in
any discussion involving $J_{\max}$.)

The main result of this subsection is the following.

\begin{prop}
  \label{prop:maximal-model}
  Let $\hat\barM$ be a $(\varphi,\hat G)$-module of type $(\vec{r},a,b,J)$, and set $h =
  h(J)$.  Then there exists a $(\varphi, \hat G)$-module $\hat\barN$ of type
  $(\vec{r},a,b,J_{\max})$ such that $\hat T(\hat \barN) \simeq
  \hat T(\hat \barM)$.
\end{prop}

\begin{proof}
  If $\vec{r} = (p-1,\ldots,p-1)$ and $J = \varnothing$ then the ambient Kisin module $\barM$ is split, and by
  Corollary~\ref{cor:(phi, hatG) for rank 1} and
  Remark~\ref{rem:type-remarks}(1) 
 the extension $\hat\barM$ is
  also split. So in this case there is nothing to prove.
 If $J = J_{\max}$ (e.g. if $\delta=0$) there is
  again nothing to prove, so we assume for the remainder
  of the proof that $\vec{r} \neq (p-1,\ldots,p-1)$ and $J\neq
  J_{\max}$.   In particular $\delta > 0$ and there exists some $j$
  such that $i_j \in J$ and $i_{j}+1,\ldots,i_j + s_j \not\in J$.  Let
  $J' = J \cup \{i_j+1,\ldots,i_j+s_j\} \setminus \{i_j\}$.  By
  induction on $\#\{j \in [1,\delta] \, : \, i_j \in J\}$, it
  suffices to prove that there exists an extension $\hat\barMp$ of type
  $(\vec{r},a,b,J')$ with $T(\hat\barMp) \simeq T(\hat\barM)$.   For simplicity
  write $i,s$ for $i_j,s_j$.

 Take a basis of $\barM$ with notation as in
 Theorem~\ref{thm:extension-crystalline}, so that in particular we
 have
 \begin{align*}
   \varphi(e_{i-1}) &= (b)_i e_{i} &   \varphi(e_{i+t-1}) & =
  (b)_{i+t} u^{p-1} e_{i+t} & \varphi(e_{i+s-1}) & = (b)_{i+s} u^p e_{i+s}\\
   \varphi(f_{i-1}) &= (a)_i u f_{i} + x_i e_{i} &   \varphi(f_{i+t-1}) & =
   (a)_{i+t} f_{i+t} & \varphi(f_{i+s-1}) & =  (a)_{i+s} f_{i+s}
\end{align*}
with the middle set of equations holding for $1 \le t \le s-1$.

We will now construct two $\varphi$-submodules $\barMp$ and $\barMpp$ of
$\barM[1/u]$, and check that they are the ambient Kisin modules of
$(\varphi,\Ghat)$-modules that satisfy the conclusion of Corollary~\ref{cor:valuation of torsion crystalline}.

Set  $e'_j = e_j$ and $f'_j = f_j$ for all $j$ except $i \le j \le
i+s-1$, and take $e'_j = u^{-1} e_j$ and $f'_j = u f_j$ for $i\le j \le i+s-1$.
Let $\barMp$ be the $\fS \otimes_{\Zp} k_E$-submodule of $\barM[1/u]$ spanned by the $e'_j$'s and
$f'_j$'s.  Then $\barMp$ is a $\varphi$-submodule of $\barM[1/u]$
with
\begin{align*}
   \varphi(e'_{i-1}) &= (b)_i u e'_{i} &   \varphi(e'_{i+t-1}) & =
  (b)_{i+t} e'_{i+t} & \varphi(e'_{i+s-1}) & = (b)_{i+s} e'_{i+s}\\
   \varphi(f'_{i-1}) &= (a)_i  f'_{i} + x_i u e'_{i} &   \varphi(f'_{i+t-1}) & =
   (a)_{i+t} u^{p-1} f'_{i+t} & \varphi(f'_{i+s-1}) & =  (a)_{i+s} u^p f'_{i+s}
\end{align*}
together with defining equations for $\varphi$ on $\overline{\M'_j}$ with $j \not\in
\{i-1,\ldots,i+s-1\}$ that are identical to those of $\varphi$ on
$\barM_j$.

Next set
$e_j'' = e_j$ for all $j$, set $f_j'' = f_j$ for all $j$ except $i \le
j \le i+s-1$, and take $f_j'' = u f_j$ for $i \le j \le i+s-1$.     Let
$\barMpp$ be the $\fS \otimes_{\Zp} k_E$-submodule of $\barM[1/u]$ spanned by the $e''_j$'s and
$f''_j$'s.  Then $\barMpp$ is a $\varphi$-submodule of $\barM[1/u]$ with
 \begin{align*}
   \varphi(e''_{i-1}) &= (b)_i e''_{i} &   \varphi(e''_{i+t-1}) & =
  (b)_{i+t} u^{p-1} e''_{i+t} & \varphi(e''_{i+s-1}) & = (b)_{i+s} u^p e''_{i+s}\\
   \varphi(f''_{i-1}) &= (a)_i  f''_{i} + x_i e''_{i} &   \varphi(f''_{i+t-1}) & =
   (a)_{i+t} u^{p-1} f''_{i+t} & \varphi(f''_{i+s-1}) & =  (a)_{i+s} u^p f''_{i+s}
\end{align*}
and defining equations for $\varphi$ on $\overline{\M''_j}$ with $j \not\in
\{i-1,\ldots,i+s-1\}$ that are identical to those of $\varphi$ on $\barM_j$.

Let us check that the $\Ghat$-action on $\hR \otimes_{\varphi,\fS}
\barM[1/u]$ preserves $\barMp$ and makes it into a
$(\varphi,\Ghat)$-module of type $(\vec{r},a,b,J')$.
Since $H_K$ acts trivially on $u$ and $\barM$, and since $\ue - 1 \in
I_{+}$, the only nontrivial part of of this
claim is that $\barMp$ is preserved by $\tau$.  This is immediate for
the action of $\tau$ on $(e'_j,f'_j)$ if $j \not\in [i,i+s-1]$.  If $i \le j \le
i+s-1$ and
$$\tau(e_j,f_j) = (e_j,f_j)
\begin{pmatrix}
 \alpha_{j+1} & \beta_{j+1} \\ 0 & \gamma_{j+1}
\end{pmatrix}$$
then an easy calculation shows that
$$\tau(e'_j,f'_j) = (e'_j,f'_j)
\begin{pmatrix}
 \alpha_{j+1} \ue^{-p} & \beta_{j+1} u^{2p} \ue^p \\ 0 & \gamma_{j+1} \ue^p
\end{pmatrix}$$
and again the conclusion is clear.
In fact, since $v_R(\beta_{j+1} u^{2p} \ue^p) \ge v_R(\beta_{j+1}) \ge
p^2/(p-1)$, not only do we obtain a $(\varphi,\Ghat)$-module
$\hat\barMp$ with ambient Kisin module $\barMp$, we have also shown that
for all $x \in \barMp$ there exist $\alpha \in R$ and $y \in R
\otimes_{\varphi,\fS} \barMp$ such that $\tau(x)-x = \alpha y$ and
$v_R(\alpha) \ge \frac{p^2}{p-1}$.  The argument for $\hat\barMpp$ is
essentially the same, with the same conclusion.

By construction we have natural inclusions $\hat\barMpp \into
\hat\barM$ and $\hat\barMpp \into \hat\barMp$.  It follows from
Lemma~\ref{lem:isomorphism-criterion} that $\hat T(\hat\barMp) \simeq
\hat T(\hat\barMpp) \simeq
\hat T(\hat\barM)$.

Note that we have not quite finished showing that $\hat\barMp$ is an
extension of type $(\vec{r},a,b,J')$: because of the presence of the
term $x_i u e'_{i}$ in $\varphi(f'_{i-1})$, the presentation for
$\barMp$ that we have given does not have exactly the same shape as in
Theorem~\ref{thm:extension-crystalline}.   To conclude, we must show
that there is a change of variables as in the proof of
Proposition~\ref{prop:extensions-of-phi-modules} that puts $\barMp$
into the correct form.  First replace $f'_{i}$ with $f'_{i} + x_i u
(a)_i^{-1} e'_{i}$, so that now $\varphi(f'_{i-1}) = (a)_i f'_{i}$; this
introduces a term of the form $c u^p e'_{i+1}$ into the formula for
$\varphi(f'_{i})$, with $c \in k_E$.  Noting that $i+1\in J'$, we can now use the terminology of the proof of
Proposition~\ref{prop:extensions-of-phi-modules}: since $p \ge
h_{i+1} \in \{p-1,p\}$, the pair $(i+1,p)$ affects some other pair.
We now distinguish three possibilities.
\begin{itemize}
\item If the pair $(i+1,p)$ is part of a \emph{loop}, then $J'$ must
  be as in the exceptional case of
  Proposition~\ref{prop:extensions-of-phi-modules}; in that case
  $\barMp$ is already written as an extension of type  $(\vec{r},a,b,J')$,
 because the
 term $cu^p e'_{i+1}$ is permitted.

\item If the pair $(i+1,p)$ is part of a \emph{stub}, suppose that the
 last term $(i_m,d_m)$ in the stub affects $(i',d')$ with $d' <
 r_{i'}$.   Then in fact $d' = 0$, because $p \mid d'$
 by~\eqref{eq:dprime}  and $r_{i'} \le p$. It follows that there is a change of
 variables as in the proof of
 Proposition~\ref{prop:extensions-of-phi-modules} that removes the
 term $cu^p e'_{i+1}$ from $\varphi(f'_{i})$, and adds a term of the form
 $x e'_{j}$ into some $\varphi(f'_{j-1})$ with
 $j \in J'$ and $x \in k_E$.  After such a change of variables,
 $\barMp$ is written as an  extension of type  $(\vec{r},a,b,J')$.

\item If the pair $(i+1,p)$ is part of a \emph{path}, then just as in
  the proof of Proposition~\ref{prop:extensions-of-phi-modules} there
  is a change of variables which eliminates the  $c u^p
  e'_{i+1}$  term.   After such a change of variables,
 $\barMp$ is written as an  extension of type  $(\vec{r},a,b,J')$.
\end{itemize}
This completes the proof.
\end{proof}

\section{The reducible case}\label{sec:reducible}We now prove Theorem
\ref{thm: crystalline lift implies explicit crystalline lift} in the
reducible case. Let $K/\Qp$ be a finite unramified extension with
residue field $k$.  As usual, we will identify $\Hom_{\Qp}(K,\Qpbar)$
with $\Hom(k,\Fpbar)$. From Definition \ref{defn: W? niveau 1}, we see
that we need to prove the following result, whose proof will occupy
the remainder of this section.
\begin{thm}
  \label{thm: elimination in the reducible case}  Suppose $p > 2$.  Let
  $\rho\col G_K\to\GL_2(\Zpbar)$ be a continuous representation such that
  $\rhobar\col G_K\to\GL_2(\Fpbar)$ is reducible. Suppose that $\rho$ is
  crystalline with $\hodgetateembedding$-Hodge--Tate weights $\{b_{\hodgetateembedding,1},b_{\hodgetateembedding,2}\}$
   for each $\hodgetateembedding\in\Hom_{\Qp}(K,\Qpbar)$, and
  suppose further that $1\le b_{\hodgetateembedding,1}-b_{\hodgetateembedding,2}\le p$ for each
  $\hodgetateembedding$.

  Then there is a reducible crystalline representation
  $\rho'\col G_K\to\GL_2(\Zpbar)$ with the same  $\hodgetateembedding$-Hodge--Tate weights as $\rho$ for each
  $\hodgetateembedding$, such that $\rhobar\simeq\rhobar'$.
\end{thm}

  Write $\rhobar\simeq
  \begin{pmatrix}
    \psibar_1&*\\0&\psibar_2
  \end{pmatrix}.$ Note that by Corollary \ref{cor:form of the characters in reducible crystalline
    reduction, nothing about extensions} there is a decomposition $\Hom_{\Qp}(K,\Qpbar)=J\coprod
  J^c$ such that $\psibar_1|_{I_K}= \prod_{ \hodgetateembedding\in
    J}\omega_{\hodgetateembeddingbar}^{b_{\hodgetateembedding,1}}\prod_{\hodgetateembedding\in
    J^c}\omega_{\hodgetateembeddingbar}^{b_{\hodgetateembedding,2}}$ and $\psibar_2= \prod_{\hodgetateembedding\in
    J^c}\omega_{\hodgetateembeddingbar}^{b_{\hodgetateembedding,1}}\prod_{\hodgetateembedding\in
    J}\omega_{\hodgetateembeddingbar}^{b_{\hodgetateembedding,2}}$. In
  fact there may be several such $J$; temporarily fix one choice.

  Let $\psi_1$, $\psi_2 \col G_K\to\Zpbartimes$ be crystalline lifts of $\psibar_1$,
  $\psibar_2$ respectively with the properties that
  $\HT_\hodgetateembedding(\psi_1)=b_{\hodgetateembedding,1} $ if $\hodgetateembedding\in J$ and $b_{\hodgetateembedding,2}$
  otherwise, and $\HT_\hodgetateembedding(\psi_2)=b_{\hodgetateembedding,2}$ if $\hodgetateembedding\in J$ and
  $b_{\hodgetateembedding,1}$ otherwise. (The characters $\psi_1$ and
  $\psi_2$ exist by Corollary~\ref{cor:(phi, hatG) for rank 1} and Proposition~\ref{prop:calculation-on-inertia},
  and are easily seen to be unique up to an unramified twist.)

 We
naturally identify $\Ext_{G_K}(\psibar_2,\psibar_1)$ with
$H^1(G_K,\psibar_1 \psibar_2^{-1})$ from now on.
\begin{defn}
  Let $L_{\psi_1,\psi_2}$ be the subset of
  $H^1(G_K,\psibar_1\psibar_2^{-1})$ consisting of all elements such that the corresponding
  representation has a crystalline lift of the form \[
  \begin{pmatrix}
    \psi_1&*\\0&\psi_2
  \end{pmatrix}.\]
\end{defn}
We have the following variant of \cite[Lem.~4.2.2]{GLS11} (which is
in turn a variant of \cite[Lem.~3.12]{bdj}).
\begin{lem}
  \label{lem: dimension of H^1_f spaces} $L_{\psi_1,\psi_2}$ is an
  $\Fpbar$-vector subspace of $ H^1(G_K,\psibar_1\psibar_2^{-1})$ of
  dimension $|J|$, unless $\psibar_1=\psibar_2$, in which case it has
  dimension $|J|+1$.
\end{lem}
\begin{proof} Let $\psi=\psi_1\psi_2^{-1}$.
  Recall that
  $H^1_f(G_K,\Zpbar(\psi))$ is by definition the preimage of
  $H^1_f(G_K,\Qpbar(\psi))$ under the natural map
  $\eta \col H^1(G_K,\Zpbar(\psi))\to H^1(G_K,\Qpbar(\psi))$, so that
  $L_{\psi_1,\psi_2}$ is the image of $H^1_f(G_K,\Zpbar(\psi))$ in
  $H^1(G_K,\psibar)$.  The kernel of $\eta$ is precisely the torsion
  part of $H^1(G_K,\Zpbar(\psi))$.  Since $\psi\neq 1$,
  e.g. by examining Hodge--Tate weights, this torsion is non-zero  if and only if
  $\psibar=1$, in which case it has the form $\lambda^{-1} \Zpbar/\Zpbar$
  for some  $\lambda \in \mf{m}_{\Zpbar}$.  (To see this, note that if
  $\psi \neq 1$ is defined over $E$, then
  the long exact sequence associated to $0 \to
  \cO_E(\psi)  \stackrel{\varpi}{\to} \cO_E(\psi) \to k_E(\psibar) \to 0$ identifies
  $k_E(\psibar)^{G_K}$ with the $\varpi$-torsion in $\ker(\eta)$.)

  By \cite[Prop.~1.24(2)]{nekovar} and the assumption that
  $b_{\hodgetateembedding,1}>b_{\hodgetateembedding,2}$ for each $\hodgetateembedding$, we see that $\dim_{\Qpbar}
  H^1_f(G_K,\Qpbar(\psi))=|J|$, again using $\psi \neq 1$. Since
  $H^1(G_K,\Zpbar(\psi))$ is a finitely generated $\Zpbar$-module, the
  result follows.
  \end{proof}

The following Lemma is a slight variant of
\cite[Lem.~6.1.6]{blggord} and \cite[Prop.~5.2.9]{GLS11}, and has an
almost identical proof.
\begin{lem}
  \label{lem: we can always lift in the cyclotomic parallel weight p
    case}Suppose that for each $\hodgetateembedding$ we have
  $b_{\hodgetateembedding,1}-b_{\hodgetateembedding,2}=p$ and that
  $(\psibar_1\psibar_2^{-1})|_{I_K}=\varepsilonbar$. Then $\rhobar$ has
  a reducible crystalline lift $\rho'$ with
  $\HT_\hodgetateembedding(\rho')=\{b_{\hodgetateembedding,1},b_{\hodgetateembedding,2}\}$
  for each $\hodgetateembedding$.
\end{lem}

\begin{proof} Suppose firstly that $\psibar_1\ne\psibar_2\varepsilonbar$. By
  assumption, we can take $J=S$ in the above. Then for any choice of
  $\psi_1$, $\psi_2$, we have
  $L_{\psi_1,\psi_2}=H^1(G_K,\psibar_1\psibar_2^{-1})$ by Lemma
  \ref{lem: dimension of H^1_f spaces} and the local Euler
  characteristic formula, completing the proof in this case.

  Assume now that $\psibar_1\psibar_2^{-1}=\varepsilonbar$. By
  twisting we can reduce to the case
  $(b_{\hodgetateembedding,1},b_{\hodgetateembedding,2})=(p,0)$ for
  each $\hodgetateembedding$.  Let $L$ be a given line in
  $H^1(G_K,\varepsilonbar)$, and choose an unramified character $\chi$
  with trivial reduction.  Let $E/\Qp$ be a finite extension with ring
  of integers $\cO$, uniformiser $\varpi$ and residue field $\F$, such
  that $\chi$ is defined over $E$ and $L$ is defined over $\F$ (that
  is, there is a basis for $L$ which corresponds to an extension
  defined over $\F$). Since any extension of $1$ by $\chi\varepsilon^p$
  is automatically crystalline, it suffices to show that we can choose
  $\chi$ so that $L$ lifts to $H^1(G_K,\cO(\chi\varepsilon^p))$.

Let $H$ be the
  hyperplane in $H^1(G_K,\F)$ which annihilates $L$ under the Tate
  pairing. Let $\delta_1 \col H^1(G_K,\F(\overline{\varepsilon})) \to
  H^2(G_K,\mc{O}(\chi\varepsilon^p))$ be the map coming from
  the exact sequence $0\to \mc{O}(\chi\varepsilon^p)\stackrel{\varpi}{\to}\mc
  O(\chi\varepsilon^p)\to \F(\overline{\varepsilon})\to 0$ of
  $G_K$-modules. We need to show that $\delta_1(L)=0$ for some choice
  of $\chi$.

  Let $\delta_0$ be the map
  $H^0(G_K,(E/\mc{O})(\chi^{-1}\varepsilon^{1-p})) \to
  H^{1}(G_K,\F)$ coming from the exact sequence $0 \to \F \to
  (E/\mc{O})(\chi^{-1}\varepsilon^{1-p}) \stackrel{\varpi}{\to}
  (E/\mc{O})(\chi^{-1}\varepsilon^{1-p}) \to 0$ of $G_K$-modules.  By
  Tate local duality, the condition that $L$ vanish under the map
  $\delta_1$ is equivalent to the condition that the image of the map
  $\delta_0$ be contained in $H$.  Let $n \geq 1$ be the largest
  integer with the property that $\chi^{-1}\varepsilon^{1-p} \equiv 1
  \pmod{\varpi^n}$. Then we can write $\chi^{-1}\varepsilon^{1-p}(x)=
  1+\varpi^n \alpha_\chi(x)$ for some function $\alpha_\chi \col G_K \to
  \mc{O}$. Let $\overline{\alpha}_\chi \col G_K
  \to \F$ denote $\alpha_\chi \pmod{\varpi}$. Then $\overline{\alpha}_\chi$ is a group homomorphism
  (i.e. a $1$-cocycle), and the choice of
  $n$ ensures that it is non-trivial. It is straightforward to check
  that the image of the map $\delta_0$ is the line spanned by
  $\overline{\alpha}_\chi$. If $\overline{\alpha}_\chi$ is in $H$ for some $\chi$, we are
  done. Suppose this is not the case. We break the rest of the proof
  into two cases.

  \medskip{\sl Case 1: $L$ is
    tr\`es ramifi\'e:}  To begin, we observe that it is
  possible to have chosen
  $\chi$ so that
  $\overline{\alpha}_\chi$ is ramified. 
To see this, let $m$ be the largest integer with the property that
$(\chi^{-1} \varepsilon^{1-p})|_{I_K} \equiv 1 \pmod{\varpi^m}$.   Note that $m$ exists since the
Hodge--Tate weights of $\chi^{-1}\varepsilon^{1-p}$ are not all $0$.
If $m = n$ then we are done, so assume instead that $m >n$. Let $g\in
G_K$ be a fixed lift of $\Frob_K$. We claim that
$\chi^{-1}\varepsilon^{1-p}(g)= 1 +\varpi^{n} \alpha_\chi(g)$ such that
$\alpha_\chi (g) \not \equiv 0 \pmod{\varpi}$. In fact, if $\alpha_\chi
(g)\equiv 0 \pmod{\varpi}$ then $\chi^{-1}\varepsilon^{1-p}(g) \in  1
+ \varpi^{n+1} \mc{O}_K$. Since $m > n$ we see that
$\chi^{-1}\varepsilon^{1-p}(G_K) \subset 1 + \varpi^{n+1} \mc{O}_K$
and this contradicts the selection of $n$. Now let $\chi'$ be the
unramified character sending our fixed $g$ to $1+\varpi^n \alpha_\chi(g)$.
Then $\chi'$ has trivial reduction, and after replacing $\chi$ by $\chi \chi'$ we
see that $n$ has increased but $m$ has not changed.  After finitely
many iterations of this procedure we have  $m=n$, completing the
claim.

Suppose, then, that $\overline{\alpha}_\chi$ is ramified.    The fact that $L$ is tr\`es
  ramifi\'e implies that $H$ does not contain the unramified line in
  $H^1(G_K,\F)$. Thus there is a unique $\overline{x} \in
  \F^\times$ such that $\overline{\alpha}_\chi+u_{\overline{x}} \in H$
  where $u_{\overline{x}} \col G_K\to \F$ is the unramified
  homomorphism sending $\Frob_K$ to $\overline{x}$. Replacing $\chi$ with $\chi$ times
  the unramified character sending $\Frob_K$ to $(1+\varpi^n x)^{-1}$,
  for $x$ a lift of $\overline{x}$, we are done.

  \medskip{\sl Case 2: $L$ is peu ramifi\'e:} Making a ramified
  extension of $\mc{O}$ if necessary, we can and do assume that $n\geq
  2$ (for example, replacing $E$ by $E(\varpi^{1/2})$ has the effect
  of replacing $n$ by $2n$). The fact that $L$ is peu ramifi\'e implies that $H$ contains the
  unramified line. It follows that if we replace $\chi$ with $\chi$
  times the unramified character sending $\Frob_K$ to $1+\varpi$, then
  we are done (as the new $\overline{\alpha}_\chi$ will be unramified).
\end{proof}

\begin{proof}[Proof of Theorem \ref{thm: elimination in the reducible
    case}.] We maintain the notation established above, so that
in
particular  we have $\rhobar\simeq
  \begin{pmatrix}
    \psibar_1&*\\0&\psibar_2
  \end{pmatrix}.$ If $(\psibar_1\psibar_2^{-1})|_{I_K}=\varepsilonbar$
  and $b_{\hodgetateembedding,1}-b_{\hodgetateembedding,2}=p$ for all $\hodgetateembedding$ then the result
  follows from Lemma \ref{lem: we can always lift in the cyclotomic
    parallel weight p case}, so assume from now on that either
  $(\psibar_1\psibar_2^{-1})|_{I_K}\ne\varepsilonbar$ or
  $b_{\hodgetateembedding,1}-b_{\hodgetateembedding,2}\ne p$ for some $\hodgetateembedding$. Twisting, we can
  and do assume in addition that $b_{\hodgetateembedding,2}=0$ for each
  $\hodgetateembedding$. Write
  $r_\hodgetateembedding:=b_{\hodgetateembedding,1}$ for each
  $\hodgetateembedding$.

Choose a finite extension $E/\Qp$ which is sufficiently large.  In
particular, choose $E$ such that: $\rho$ is defined over
$\cO_E$;  and for each tuple of integers $\{s_\hodgetateembedding\}$ in the
  range $[0,p]$ such that if $\overline{\psi}_i$ ($i=1,2$) has a
  crystalline lift
 $\psi_i$  with $\HT_{\hodgetateembedding}(\psi_i) =
 s_{\hodgetateembedding}$ for all~$\hodgetateembedding$,
  it has such a lift defined over~$\O_E$.  Fixing one choice for each possible
  $\psi_i$ (for each choice of Hodge--Tate weights) in the previous
  clause, further enlarge $E$ so that each space
  $H^1_f(G_K,\Zpbar(\psi_1 \psi_2^{-1}))$ is defined over $\O_E$.

From now on, we will allow $\rho$ (and thus $\rhobar$) to vary over all
crystalline representations $G_K\to\GL_2(\cO_E)$ which have  $\rhobar\simeq
  \begin{pmatrix}
    \psibar_1&*\\0&\psibar_2
  \end{pmatrix}$ (where the extension class~$*$ is allowed to vary)
  and have $\hodgetateembedding$-labelled Hodge--Tate weights $\{0,
  r_\hodgetateembedding\}$ for each $\hodgetateembedding$. By
  Theorem~\ref{thm:extension-crystalline} together with Remark
  \ref{rem:type-remarks}(2), Proposition \ref{prop:maximal-model}, and
  the discussion between them,
  we see that there exist $a$, $b\in k_E$ and a subset
  $J_{\max}\subset\{0,\dots,f-1\}$ so that for any such $\rho$,
  there is a $(\varphi,\Ghat)$-module $\hat\barN$ of type
  $(\vec{r},a,b,J_{\max})$ such that $\hat T(\hat \barN) \simeq
  \rhobar$.  (Apply Proposition~\ref{prop:calculation-on-inertia} to
  see that $a,b$ are uniquely determined.)   By Theorem \ref{thm:extension-crystalline} and the
  assumption that we are not in the case that
  $(\psibar_1\psibar_2^{-1})|_{I_K}=\varepsilonbar$ and each
  $r_\hodgetateembedding=p$, we see that we are not in the exceptional
  case in Lemma \ref{lem: hat G is determined by phi}; there are
  thus at most $(\#k_E)^{|J_{\max}|}$ isomorphism classes of
  $(\varphi,\Ghat)$-modules $\hat\barN$ of type $(\vec{r},a,b,J_{\max})$,
  and thus (by Theorem \ref{thm:extension-crystalline} and Remark
  \ref{remark: exceptional case only with equal characters}) at most
  $(\#k_E)^{|J_{\max}|}$ elements of $H^1(G_K,\psibar_1\psibar_2^{-1})$
  corresponding to representations $\rhobar$, \emph{unless}
  $\psibar_1=\psibar_2$, in which case $(\#k_E)^{|J_{\max}|}$ must be
  replaced with $(\#k_E)^{|J_{\max}|+1}$.

  Now apply the discussion at the beginning of this section with
  $J=J_{\max}$; that is, choose (as we may, by, for example,
  Proposition \ref{prop:rank-one-subs}) crystalline characters
  $\psi_1$, $\psi_2$ lifting $\psibar_1$, $\psibar_2$ respectively
  such that $\HT_\hodgetateembedding(\psi_1)=r_\hodgetateembedding$ if
  $\hodgetateembedding\in J_{\max}$ and $0$ otherwise, and
  $\HT_\hodgetateembedding(\psi_2)=0$ if $\hodgetateembedding\in J_{\max}$
  and $r_\hodgetateembedding$ otherwise.  Note that by our choice of
  $E$ we may further
  suppose
that $\psi_1$, $\psi_2$, and $H^1_f(G_K,\Zpbar(\psi_1\psi_2^{-1}))$
are all defined over $\O_E$.

By Lemma \ref{lem: dimension
    of H^1_f spaces}
 we see that there are $(\#k_E)^{|J_{\max}|}$
  extension classes which arise as the reductions of crystalline
  representations which are extensions of $\psi_2$ by $\psi_1$, unless
  $\psibar_1=\psibar_2$, in which case there are
  $(\#k_E)^{|J_{\max}|+1}$ extension classes. Since we have already
  shown that there are at most $(\#k_E)^{|J_{\max}|}$ (or
  $(\#k_E)^{|J_{\max}|+1}$ if $\psibar_1=\psibar_2$) extension classes
  arising from the reduction of crystalline representations with
  $\hodgetateembedding$-labelled Hodge--Tate weights $\{0,
  r_\hodgetateembedding\}$, the result follows.
\end{proof}
 
\section{The irreducible case}\label{sec:irreducible} We now explain how to deduce the
irreducible case of Theorem \ref{thm: crystalline lift implies
  explicit crystalline lift} from the reducible one. A usual, let $K=K_0$ be the unramified extension of
$\Qp$ of degree $f$, and let \[\rho \col G_K\to\GL_2(\Qpbar)\] be a
continuous irreducible representation such that
$\rhobar \col G_K\to\GL_2(\Fpbar)$ is also irreducible. Suppose that $\rho$
is crystalline with $\hodgetateembedding$-Hodge--Tate weights
$\{b_{\hodgetateembedding,1},b_{\hodgetateembedding,2}\}$ for each $\hodgetateembedding\in\Hom(K,\Qpbar)$, and
suppose further that $1\le b_{\hodgetateembedding,1}-b_{\hodgetateembedding,2}\le p$ for each
$\hodgetateembedding$.

Let $k$ denote the residue field of $K$, and let $K_2$ be the
quadratic unramified extension of $K$, with residue field $k_2$. We
write $S=\Hom(k,\Fpbar)$ and $S_2=\Hom(k_2,\Fpbar)$. We say that
$J\subset S_2$ is a \emph{balanced subset} if it consists of precisely
one element of $S_2$ extending each element of $S$. If $\sigma\in S$
is the reduction mod $p$ of $\hodgetateembedding\in\Hom(K,\Qpbar)$, we write
$b_{\sigma,i}$ for $b_{\hodgetateembedding,i}$. Recalling the
definition of $W^{\textrm{BDJ}}(\rhobar)$ when $\rhobar$ is
irreducible (Definition~\ref{defn: W? niveau 2}) we see that in order to complete the proof of
Theorem \ref{thm: crystalline lift implies explicit crystalline lift},
we need to prove the following result.
\begin{thm}
  \label{thm: reduction mod p for GL2 in the irreducible case}There is
  a balanced subset $J\subset S_2$ such that \[\rhobar|_{I_K}\simeq
  \begin{pmatrix}
    \prod_{\sigma\in J}\omega_\sigma^{b_{\sigma|_k,1}}\prod_{\sigma\notin
      J}\omega_\sigma^{b_{\sigma|_k,2}}&0\\0&\prod_{\sigma\in J}\omega_\sigma^{b_{\sigma|_k,2}}\prod_{\sigma\notin
      J}\omega_\sigma^{b_{\sigma|_k,1}}
  \end{pmatrix}.\]

\end{thm}
\begin{proof}
 Since $\rhobar|_{G_{K_2}}$ is
reducible, by Corollary \ref{cor:form of the characters in reducible crystalline
    reduction, nothing about extensions} we certainly have a decomposition as in the statement of
the Theorem for some $J\subset
S_2$, but we do not know that $J$ is balanced. Indeed, this is not
completely automatic, but we will show that a balanced choice of $J$
always exists.

To see this, note that since $\rhobar|_{I_K}$ is irreducible, we must
have \[\prod_{\sigma\in
  J}\omega_\sigma^{b_{\sigma|_k,1}}\prod_{\sigma\notin
  J}\omega_\sigma^{b_{\sigma|_k,2}}=\prod_{\sigma\in
  J}\omega_{\sigma\circ\varphi^f}^{b_{\sigma|_k,2}}\prod_{\sigma\notin
  J}\omega_{\sigma\circ\varphi^f}^{b_{\sigma|_k,1}}.\] Write $J_1$ for
the set of places in $S$ both of whose extensions to $S_2$ are in $J$,
and $J_2$ for the set of places in $S$ neither of whose extensions to
$S_2$ are in $J$. Then we see that we have \[\prod_{\sigma\in
  J_1}\omega_\sigma^{b_{\sigma,1}-b_{\sigma,2}}=\prod_{\sigma\in
  J_2}\omega_\sigma^{b_{\sigma,1}-b_{\sigma,2}}.\] If both $J_1$,
$J_2$ are empty, then $J$ is balanced, and we are done. Assume
therefore that this is not the case.

Define $x_\sigma$ as follows: $x_\sigma=b_{\sigma,1}-b_{\sigma,2}$ if
$\sigma\in J_1$, $x_\sigma=b_{\sigma,2}-b_{\sigma,1}$ if $\sigma\in
J_2$, and $x_\sigma=0$ otherwise.  Note that since $\rhobar$ is
irreducible, there is at least one place $\sigma$ with
$x_\sigma=0$. We have $\prod_{\sigma\in S}\omega_\sigma^{x_\sigma}=1$,
and each $x_\sigma\in[-p,p]$. Choose an element $\sigma_0\in S$, and
recursively define $\sigma_i=\sigma_{i+1}^p$. Writing $\omega_i$ for
$\omega_{\sigma_i}$, we have $\omega^p_{i+1}=\omega_i$. From now on,
we identify $S$ with $\{0,\dots,f-1\}$ by identifying $\sigma_i$ with
$i$. 
By Lemma~\ref{lem:minus-p-to-p}, the cyclic set of those $i$
with $x_i\ne 0$ must break up as a disjoint union of sets of the form
$(i,i+1,\dots,i+j)$ with
$(x_i,x_{i+1},\dots,x_{i+j})=\pm(-1,p-1,p-1,\dots,p-1,p)$ (where there
may not be any occurrences of $p-1$). For each such interval $(i,i+j)$,
we may choose a lift of $i$ to $S_2$, and replace $J$ with
$J\Delta\{i,\dots,i+j\}$. It is easy to see that this choice does not
change $\rhobar|_{I_K}$, and results in a balanced choice of $J$, as
required.
\end{proof}
\begin{rem}It is perhaps worth illustrating the proof of Theorem
  \ref{thm: reduction mod p for GL2 in the irreducible case} with an
  example. Take $f=4$, and consider a representation of the form \[
  \begin{pmatrix}
    \omega_1^{p-1}\omega_2^p\omega_3^b\omega_5^{p-1}\omega_6^p&0\\0&\omega_0\omega_4\omega_7^b
  \end{pmatrix},\]with $0<b\le p-1$. This is certainly a possible
  restriction to inertia of an irreducible representation, but it is
  not written in the balanced form of the statement of Theorem
  \ref{thm: reduction mod p for GL2 in the irreducible case}. However,
  if we write it as \[\begin{pmatrix}
    \omega_1^{p-1}\omega_2^p\omega_3^b\omega_4&0\\0&\omega_0\omega_5^{p-1}\omega_6^p\omega_7^b
  \end{pmatrix},\] then we obtain a balanced expression, as required.
\end{rem}

\bibliographystyle{amsalpha}
\bibliography{savitt}

\end{document}